%% file: main.tex
\documentclass[conference]{IEEEtran}
\IEEEoverridecommandlockouts
\usepackage{cite}
\usepackage{amsmath,amssymb,amsfonts, amsthm}
\newtheorem{lemma}{Lemma}
\newtheorem{corollary}{Corollary}
\newtheorem{theorem}{Theorem}
\newtheorem{remark}{Remark}
\newtheorem{observation}{Observation}
\newtheorem{definition}{Definition}

\newtheorem{proposition}{Proposition}
\usepackage{algorithmic}
\usepackage{graphicx}
\usepackage{textcomp}
\usepackage{tabularx}
\usepackage{xcolor}
\usepackage{todonotes}
\usepackage{booktabs}
\usepackage{amssymb} 
\usepackage{array}   

\usepackage{enumitem}

\usepackage{bbm}             

\def\BibTeX{{\rm B\kern-.05em{\sc i\kern-.025em b}\kern-.08em
    T\kern-.1667em\lower.7ex\hbox{E}\kern-.125emX}}
\begin{document}

\title{Gramians for a New Class of Nonlinear Control Systems Using Koopman and a Novel Generalized SVD
}

\author{Brian Charles Brown$^1$, Michael King$^1$
\thanks{1 - Department of Computer Science, Brigham Young University, UT, 84602.
This work was funded by DOE Grant \#SC0021693.
Correspondence should be addressed to Brian Brown at \texttt{bcbrown365@gmail.com}.}}

\maketitle

\input{subsections/new_intro_3}
\input{subsections/background_and_notation}
\input{subsections/results}
\input{subsections/gsvd_4}
\input{subsections/results_gramians_strong_assumptions}

\input{subsections/appendix_summary}

\input{subsections/numeric_results}

\input{subsections/conclusion}

\bibliographystyle{IEEEtran}
\bibliography{references}
\appendix

\input{subsections/results_gramians_weak_assumptions}

\input{subsections/appendix_new}

\end{document}

%% file: subsections/new_intro_3.tex
\begin{abstract}
Certified model reduction for high-dimensional nonlinear control systems remains challenging: unlike balanced truncation for LTI systems, most nonlinear reduction methods either lack computable worst-case error bounds or rely on intractable PDEs. Data-driven Koopman/DMDc surrogates improve tractability, but standard \emph{input lifting} can distort the physical input-energy metric, so $H_\infty$ and Hankel-based bounds computed on the lifted model may be valid only in a lifted-input norm and need not certify the original system. 
We address this metric mismatch by a Generalized Singular Value Decomposition (GSVD)-based construction that represents general (including non-affine) input nonlinearities in an LTI-like lifted form with a \emph{pointwise norm-preserving} input map $v(x,u)$ satisfying $\|v(x,u)\|_2=\|u\|_2$ and constant matrices $A,B$. This preserves strict causality (constant $B$, no input-history augmentation) and yields computable Hankel-singular-value-based $H_\infty$ error certificates in the physical input norm for reduced-order surrogates. 
We illustrate the method on a 25-dimensional Hodgkin--Huxley network with saturating optogenetic actuation, reducing to a single dominant mode while retaining certified error bounds.
\end{abstract}

\section{Introduction}

\subsection{Linear Certification via Bounds}

Model reduction for linear systems is, in many respects, a solved problem. This maturity is largely due to the existence of a complete theoretical framework rooted in the Controllability and Observability Gramians for linear time-invariant (LTI) systems \cite{kalman1963mathematical}. These Gramians provide a fundamental decomposition of the system's Hankel operator, which maps past inputs to future outputs, allowing for the systematic identification of states that contribute negligibly to the system's input-output energy transfer.

The primary value of this framework lies in its ability to provide certification. The Hankel singular values derived from these Gramians yield a rigorous a priori $H_\infty$ upper bound on the error between the full-order LTI model and its reduced surrogate \cite{dullerud2000chapter4}. This bound guarantees the worst-case performance deviation across all frequencies. While Dullerud and Paganini \cite{dullerud2000chapter4} note that this theoretical bound is often conservative rather than tight for real world systems, it remains the standard for validation. In safety-critical scenarios where stability margins must be guaranteed before deployment, the existence of such a certifiable error bound distinguishes the linear theory from the vast majority of nonlinear reduction techniques.

The linear guarantee belies a structural point for nonlinear surrogate modeling, namely that certificates are norm-dependent statements, and balanced truncation is no exception. The classical \(H_\infty\) reduction bound controls the induced gain from the surrogate’s input norm to the output norm. The interpretability of this bound therefore relies on that input norm being physically meaningful. In Koopman/DMDc-style surrogates, the “lifted input” \(v(x,u)\) is typically introduced as a regressor, but unless \(u \mapsto v(x,u)\) is norm-calibrated, the resulting Hankel singular values certify error only with respect to \(\|v\|\), not the physical control norm \(\|u\|\). Consequently, an \(H_\infty\) certificate computed on the lifted model can fail to certify the original control system. 
Motivated by this limitation, we next contrast existing nonlinear reduction approaches with a Generalized Singular Value Decomposition (GSVD) based-based framework that explicitly resolves the input-metric ambiguity and enables computable $H_\infty$ certification for nonlinear systems.

\begin{table*}[t]
    \centering
    \caption{Comparison of Nonlinear Model Reduction Frameworks}
    \label{tab:full_comparison}
    \resizebox{\textwidth}{!}{%
    \begin{tabular}{l c c c c l}
        \toprule
        \textbf{Approach} & \textbf{A priori} & \textbf{Computable?} & \textbf{Non-Affine} & \textbf{Causal} & \textbf{Resulting} \\
        & \textbf{Bounds?} & & \textbf{Inputs?} & \textbf{Structure?} & \textbf{Model} \\
        \midrule
        \textbf{Energy \& Differential} & & & & & \\
        Scherpen / Gray / Fujimoto \cite{scherpen1993balancing, gray1998hankel, fujimoto2005nonlinear} & $\times$ & $\times$ & $\times$ & $\checkmark$ & Balanced (not reduced) \\
        Besselink \cite{besselink2014incremental} & $\checkmark$ & \textasciitilde & $\times$ & $\checkmark$ & Nonlinear \\
        Gray \& Verriest \cite{gray2006algebraic} & \textasciitilde & $\checkmark$ & $\times$ & $\checkmark$ & Nonlinear \\
        \midrule
        \textbf{Empirical \& Data-Driven} & & & & & \\
        Lall / Hahn / Condon \& Ivanov \cite{lall2002subspace, hahn2002improved, condon2004empirical} & $\times$ & $\checkmark$ & $\times$ & $\checkmark$ & Nonlinear \\
        Himpe \cite{himpe2014cross, himpe2018emgr} & $\times$ & $\checkmark$ & $\checkmark$ & $\checkmark$ & Nonlinear \\
        Kawano \cite{kawano2021empirical} & $\times$ & $\checkmark$ & $\times$ & $\checkmark$ & Nonlinear \\
        \midrule
        \textbf{Koopman} & & & & & \\
        Proctor / Yeung \cite{proctor2016dynamic, yeung2018koopman} & $\times$ & $\checkmark$ & \textasciitilde & $\times$ & Linear (LTI) \\
        Liu et al. \cite{liu2018decomposition} & $\times$ & $\checkmark$ & \textasciitilde & $\times$ & Linear (LTI) \\
        Goswami \cite{goswami2017global} & $\times$ & \textasciitilde & $\times$ & $\checkmark$ & Bilinear \\
        Haseli \cite{haseli2025koopman, haseli2025roadskoopmanoperatortheory} & $\checkmark$ & \textasciitilde & $\checkmark$ & $\checkmark$ & LPV / Infinite \\
        Par\'e (MBAM) \cite{pare2019unifying} & $\times$ & \textasciitilde & $\checkmark$ & $\checkmark$ & Simplified Nonlinear \\
        \midrule
        \textbf{This Work (Generalized SVD)} & $\checkmark$ & $\checkmark$ & $\checkmark$ & $\checkmark$ & \textbf{Linear (LTI)} \\
        \bottomrule
    \end{tabular}
    }
    \vspace{1ex}
    
    {\raggedright \footnotesize \textit{Legend:} $\checkmark$ = Yes; $\times$ = No; \textasciitilde = Partial/Moderate. ``Computable?'' refers to reliance on standard linear algebra vs. PDEs/LMIs. ``Causal Structure?'' refers to independence from future inputs or infinite delay embeddings. “A priori bounds?” refers to computable worst-case input-output error bounds (e.g., $H_\infty$ or $L_2$) suitable for robust control certification.
 \par}
\end{table*}
\subsection{The Geometric and Projection Alternatives}

It is useful to separate two distinct objectives that are often conflated in nonlinear reduction: \emph{trajectory compression} (approximating state snapshots efficiently) versus \emph{system certification} (bounding the induced input--output error under worst-case disturbances). Many widely used nonlinear reduction methods primarily target the former, and therefore do not directly address the norm-dependent input--output certification issue highlighted above.

Projection-based methods such as Proper Orthogonal Decomposition (POD) with Galerkin projection \cite{holmes1996book} exemplify this distinction. POD identifies low-dimensional subspaces that capture dominant variance/energy in observed state trajectories and is effective for accelerating simulation in settings such as fluid dynamics and continuum mechanics. However, because the objective is state-energy capture rather than operator gain control, POD-based reduced models typically provide accuracy guarantees of a signal-approximation type, not computable worst-case input--output bounds of the $H_\infty$ form required for robust synthesis.

Complementary geometric approaches reduce complexity by exploiting structure in the model-to-data map rather than by compressing reachable/observable energy. The Manifold Boundary Approximation Method (MBAM) uses the Fisher Information Matrix to identify ``sloppy'' parameter combinations and simplifies the governing equations by moving toward lower-dimensional boundaries of the model manifold \cite{transtrum2014model, transtrum2016bridging, transtrum2017measurement}. Par\'e et al.~\cite{pare2019unifying} show that this geometric viewpoint recovers Balanced Truncation and Singular Perturbation Approximation as limiting cases for linear systems, demonstrating that classical energy-based reduction can be interpreted as a geometric contraction of the input--output map. For general nonlinear systems, however, MBAM is primarily a tool for parameter reduction and physical simplification; it does not, by itself, produce computable induced-gain error bounds for controller certification.

These alternatives motivate the central control-theoretic difficulty of obtaining reduction procedures that retain the energy/Hankel interpretation needed for worst-case input--output bounds while remaining computationally tractable. The next subsection reviews the ``energy lineage'' of nonlinear balancing methods through this lens.

\subsection{The Energy Lineage: Rigor vs. Computability}

This energy/Hankel lineage starts with Scherpen \cite{scherpen1993balancing}, who extended balanced truncation to nonlinear systems through controllability and observability energy functions. This framework provides a rigorous extension of the linear theory, preserving the physical interpretation of the Gramians as metrics for the energy required to reach a state and the energy produced by that state. 
Gray and Scherpen \cite{gray1998hankel} later formally defined a nonlinear Hankel operator and proved an associated factorization into nonlinear controllability and observability operators, which they then related to the corresponding energy functions. Fujimoto and Scherpen \cite{fujimoto2005nonlinear} further refined the theory by analyzing the differential eigenstructure of these operators, demonstrating that the \emph{state-dependent singular value functions} of a nonlinear system can be characterized in a coordinate-independent manner. 
The energy/Hankel-operator framework is theoretically robust, but does not directly yield computable a priori input--output error bounds of the type used for worst-case robust synthesis. 

However, the practical application of these rigorous methods faces a formidable barrier: computing the exact energy functions requires solving Hamilton-Jacobi-Bellman (HJB) partial differential equations. For general nonlinear systems with state dimensions larger than a few variables, solving these PDEs is computationally intractable. To overcome this barrier, Gray and Verriest \cite{gray2006algebraic} proposed replacing the differential equations with algebraic generalized Lyapunov equations, offering a computationally feasible approximation that bounds the true energy functions. Finally, bypassing the need for explicit system equations entirely, the field has pivoted toward empirical methods, such as the work by Kawano and Scherpen \cite{kawano2021empirical}, which approximate these Gramians directly from simulation data along system trajectories.
Lall et al. \cite{lall2002subspace}, Hahn and Edgar \cite{hahn2002improved}, and Condon and Ivanov \cite{condon2004empirical} proposed computing ``empirical Gramians" by averaging state trajectory snapshots generated from specific perturbations, such as impulsive inputs. While this approach successfully achieved computability for nonlinear systems, the specific reliance on impulsive inputs (Dirac deltas) mathematically restricts these initial methods to control-affine systems, as non-affine terms (e.g., $u^2$) render the system response to an impulse undefined.
Himpe \cite{himpe2018emgr} later generalized this framework by allowing for arbitrary training inputs (e.g., step functions or chirps), thereby relaxing the impulse-based control-affine restriction in empirical Gramian computation. However, all these empirical methods necessitate a trade-off: they sacrifice the global validity of the original linear theory, abandoning rigorous a priori error bounds (such as $H_\infty$) in exchange for numerical feasibility and local accuracy within a specific operating region.

In contrast to the heuristic nature of the standard empirical methods, Besselink et al. \cite{besselink2014incremental} established a rigorous framework based on incremental stability, proving that if generalized gramians satisfying specific Linear Matrix Inequalities (LMIs) are found, the reduced model is guaranteed to be stable and satisfy an \textit{a priori} $L_2$ error bound.
Kawano and Scherpen \cite{kawano2021empirical} later bridged the gap between rigorous differential theory and trajectory-wise empirical computability by introducing empirical differential gramians, which allow the variational system properties to be estimated directly from trajectory data rather than solving nonlinear PDEs. However, both approaches currently face topological limitations: they are mathematically restricted to systems with constant input vector fields (a subset of control-affine systems where $g(x) = B$) and, like the other methods, yield reduced models that remain nonlinear.

\subsection{The Data-Driven Era: Koopman and the Control Problem}

Parallel to these structural developments, the renaissance of Koopman operator theory has offered an alternative path comprised of embedding nonlinear dynamics into a higher-dimensional linear framework to leverage standard spectral analysis tools. Schmid \cite{schmid2010dynamic} and Mezi\'c \cite{mezic2013analysis} demonstrated that the global behavior of nonlinear flows could be characterized by the eigenvalues and eigenfunctions of the linear Koopman operator. This insight has led to the development of many data-driven techniques, most notably Extended Dynamic Mode Decomposition (EDMD) \cite{williams2015data}, which approximates the infinite-dimensional operator using a finite dictionary of observable functions \cite{brunton2022modern}.

However, incorporating control inputs into this operator-theoretic framework has proven to be a source of significant theoretical conflict. The most common data-driven approach, as popularized by Proctor et al.~\cite{proctor2016dynamic}, Yeung et al.~\cite{yeung2018koopman}, and Korda and Mezi\'c~\cite{korda2018linear}, relies on ``input lifting'' or linear predictors, often by stacking \(u\) alongside state observables.
This is computationally convenient, though it entangles two issues that are easy to overlook and are fatal for certification: \emph{(i)} causality and \emph{(ii)} metric fidelity.
First, it can blur the distinction between state evolution and actuation, implicitly importing a dependence on embedded input histories that departs from the classical state–space notion of causality when the input is embedded into the lifted observable.
For example, in formulations that evolve a control-augmented Koopman observable \(\varphi(x,u)\), retaining a single linear propagator across time is usually paired with the assumption that \(u\) can be treated as an exogenous signal without its own state-space dynamics (cf.~\cite{yeung2018koopman}).
The second issue is metric fidelity. If the lifted input \(\varphi(x,u)\) is not calibrated to the norm of the input, then any \(H_\infty\) or Hankel-based bound computed on the lifted surrogate is expressed in a different input norm and therefore cannot certify the original control input \(u\).
Even when one can compute Gramians for the lifted system, the resulting Hankel singular values are not interpretable as \(H_\infty\) reduction bounds for the underlying nonlinear system unless the input lifting preserves the relevant input metric.

Recent theoretical work adds that a fully consistent Koopman treatment of open-loop control cannot, in general, retain a single finite-dimensional, time-homogeneous linear propagator without either 1) encoding the full input sequence or 2) allowing the operator to vary with the input. In particular, Haseli et al.~\cite{haseli2025roadskoopmanoperatortheory} show that rigorous formulations lead to either infinite input-sequence representations or operator families (KCF), whose finite-dimensional restrictions are necessarily input-dependent (LPV) \cite{haseli2025koopman}. The next subsection states how we preserve Koopman structure where it is naturally LTI (the autonomous drift) while keeping the remaining input dependence isolated in such a way that reduction bounds can be placed on reduced order models.

\subsection{The Contribution: A Causal, Certified, LTI Synthesis}

We adopt a different synthesis target than a full Koopman representation of the open-loop control system. We use Koopman lifting only for the autonomous (unforced) dynamics to obtain a fixed LTI core on the lifted state ($\varphi(x)$).
Actuation then enters through a constant input matrix ($B$) driven by an instantaneous lifted input ($v(x,u)$), echoing the standard input-lifting template used in Koopman/DMDc-style predictors~\cite{korda2018linear, yeung2018koopman}. Without further constraints this representation suffers the same limitation that the LTI core can be reduced, but the resulting Hankel-based certificate is expressed in the lifted-input metric. We resolve this by imposing additional structure on (v) that restores correspondence with the physical input norm.

Our primary contribution is a Generalized Singular Value Decomposition (GSVD)-based factorization that restores the LTI Hankel framework while remaining consistent with the \emph{physical} input metric. We use this framework to introduce a state-dependent instantaneous lifted input map, \(v\), satisfying the pointwise norm-preservation constraint
\[
\|v(x,u)\|_2=\|u\|_2 \quad \forall(x,u),
\]
and we represent the lifted dynamics in a balanced LTI-like form
\[
D\varphi(x)\,f(x,u) = A\varphi(x)+Bv(x,u)
\]
where $D\varphi(x)$ is the Jacobian of the lifting, and where the matrices $A$ and $B$ are constant. Under this constraint, Hankel singular values computed from the lifted LTI core admit a physically meaningful \(H_\infty\) interpretation with respect to the original control input (u), rather than the surrogate regressor norm. The GSVD construction is the mechanism that makes this representation achievable for general (including non-affine) input nonlinearities by isolating gain into fixed linear factors and confining the remaining nonlinearity to the norm-preserving input channel.

Theorem \ref{theorem:norm_preserving} (calibrated lifting) establishes the norm-preserving lifted representation with constant $B$; Theorem \ref{theorem:final_bound} (certified truncation) gives the certified reduction bound under exact Koopman closure; Theorem \ref{theorem:bound_with_error} (closure-error deformation; Appendix) shows how the certificate deforms under Koopman closure error via a small-gain feedback interpretation.

This approach offers five distinct advantages over the state of the art:
\begin{enumerate}
    \item \textbf{Causal Structure:} Unlike input-lifting approaches \cite{korda2018linear} that embed the control history into the state vector, we preserve an explicit, constant $B$-matrix. This ensures strict causality by treating the control influence as an instantaneous exogenous driver rather than a state to be predicted.
    \item \textbf{LTI Simplicity:} Unlike the bilinear forms derived by Goswami and Paley \cite{goswami2017global} or the parameter-varying families required by Haseli et al. \cite{haseli2025roadskoopmanoperatortheory}, our method yields a reduced model with constant system matrices ($A, B$). This allows for the application of standard LTI control synthesis tools with minor modifications.
    \item \textbf{Non-Affine Support:} By capturing the input-state interaction within the signal $v(x,u)$, our framework handles general non-affine inputs. 
    \item \textbf{Rigorous Bounds:} Most critically, by enforcing norm preservation in the lifting, we prove that the Hankel singular values of the lifted surrogate contribute to a true, certified $H_\infty$ error bound for the original nonlinear input--output map in the \emph{physical} input-energy metric. Conceptually, this aligns with the nonlinear Hankel-operator viewpoint of Gray and Scherpen \cite{gray1998hankel} (which generalizes Hankel factorizations and Gramian-like objects to nonlinear systems), while our contribution is to make the resulting certificate \emph{computable} and \emph{metric-consistent} via input-energy calibration.
    \item \textbf{Certified Neural Representations:} While data-driven methods like \cite{liu2018decomposition} utilize deep learning to approximate Koopman observables, they typically lack safety guarantees for the resulting closed-loop system. We demonstrate training neural networks to parameterize the lifting components ($A, B, \varphi, v$) subject to our norm-preservation constraint, while Theorems 2 and 3 provide rigorous a priori error bounds for model reduction with this neural surrogate system. This effectively enables the ``certification" of deep-learning-based models and their reduced-order derivatives, ensuring they satisfy stability margins required for safety-critical control.
\end{enumerate}

%% file: subsections/background_and_notation.tex
\section{Background and Notation}

We will consider vector norms, norms of infinite-dimensions, function norms, induced norms of finite-dimensional functions, and induced operator norms.
These are distinguished in Table \ref{tabs:norms}.

\begin{table}[h]
\centering
\begin{tabularx}{\columnwidth}{|c|c|X|}
\hline
\label{tabs:norms}
\textbf{Notation} & \textbf{Type of Object} & \textbf{Definition} \\ \hline
$\|f(x)\|_2$ & $f(x) \in \mathbb{R}^m$ & $\left( \sum_{i=1}^m |f_i(x)|^2 \right)^{1/2}$ \\ \hline
$\|u\|_{L_2}$ & $u(t) \in L_2$ & $\left( \int_{-\infty}^{\infty} \|u(t)\|_2^2 \,dt \right)^{1/2}$ \\ \hline
$\|f\|_{2\to2}$ & Mapping $f: \mathbb{R}^n \to \mathbb{R}^m$ & $\sup_x \frac{\|f(x)\|_2}{\|x\|_2}$ \\ \hline
$\|G\|_{H_{\infty}}$ & Operator $G: L_2 \to L_2$ & $\sup_u \frac{\|G(u)\|_{L_2}}{\|u\|_{L_2}}$ \\ \hline
\end{tabularx}
\caption{Summary of norms used throughout the paper.}
\end{table}

We will frequently reference functions with ``finite 2-induced norm", i.e. functions, $f$, satisfying $\|f\|_{2 \to 2} < \infty$, or more concretely:
\begin{align}
\sup_{x\neq0} \frac{\|f(x)\|_2}{\|x\|_2} < \infty.
\end{align}

We will frequently make sue of a restricted induced gain on an admissible input class.
Let $\mathcal{U}\subset L^2$ denote an admissible input class with the property that, for every $u(\cdot)\in\mathcal{U}$ and every system considered, the corresponding solution exists for all $t\ge 0$ and the resulting state trajectory remains in the prescribed compact set $\mathcal{X}$ (equivalently, in balanced coordinates, $z$, with some state-recovering transform, $R$, then $Rz(t)\in\mathcal{X}$).
For such an admissible input class $\mathcal{U}$, define the restricted induced gain
\[
\|G\|_{H_\infty(\mathcal{U})} \triangleq \sup_{u\in\mathcal{U}\setminus\{0\}} \frac{\|G(u)\|_{L^2}}{\|u\|_{L^2}}.
\]
When $G$ is stable LTI, we use the classical $H_\infty$ norm (equivalently, take $\mathcal{U}=L^2$) and write $\|G\|_{H_\infty}$.

Throughout, $D\varphi_0(x)$ denotes the Jacobian matrix of the lifting map
$\varphi:\mathbb{R}^n\to\mathbb{R}^q$, evaluated at the point $x$.
That is, $D\varphi(x)\in\mathbb{R}^{q\times n}$ is the matrix of first-order partial derivatives of $\varphi$ with respect to $x$.


%% file: subsections/results.tex
\section{Results}

%% file: subsections/gsvd_4.tex
\subsection{Preliminary Results: Generalized SVD}
\label{sec:gsvd_prelims}

The aim of this subsection is to characterize how a nonlinear map contributes induced gain in a form that is geometric, anisotropic, and compatible with certified LTI analysis. Rather than summarizing the map by a single worst-case scalar, we seek a representation that isolates and orders the amplification associated with individual output directions through a fixed linear structure. The development generalizes the construction in \cite{brown2024svd}.

We proceed by first introducing a \emph{gain cage}: a coordinate-dependent bound that constrains amplification along each output axis in a chosen output basis. This replaces a global Lipschitz constant with axis-dependent gain limits and yields an anisotropic description suitable for dynamical analysis. We then show that a gain cage with strict margin implies a structural factorization: the map can be written as a fixed linear gain operator acting on a nonlinear lift that is injective and norm-preserving. In this factorization, all amplification is confined to a constant linear object, while the remaining nonlinearity preserves input energy pointwise.

\begin{definition}[Diagonal gain cage]
\label{def:diagonal_gain_cage}

Let \(f:\mathbb{R}^n \to \mathbb{R}^m\) satisfy \(f(0)=0\).
Fix an orthogonal matrix \(U \in \mathbb{R}^{m\times m}\) and a diagonal matrix
\(D=\mathrm{diag}(\sigma_1,\dots,\sigma_m)\succ 0\).
For a constant \(\beta\ge 0\), we say that the pair \((U,D)\)
\emph{\(\beta\)-cages} \(f\) if
\begin{equation}
\label{eq:gain_cage_condition}
\bigl\| D^{-1} U^\top f(x) \bigr\|_2
\;\le\;
\beta\,\|x\|_2
\qquad
\forall x\in\mathbb{R}^n\setminus\{0\}.
\end{equation}

Equivalently, the image of the unit ball under \(f\) satisfies
\[
\|D^{-1}U^\top f(x)\|_2 \le \beta
\qquad \forall x \text{ with } \|x\|_2 \le 1.
\]
\end{definition}

\begin{remark}[Geometric interpretation]
\label{rem:gain_cage_geometry}

The matrix \(U\) selects an output coordinate system, while the diagonal
matrix \(D\) prescribes axis-dependent gain limits.
The constant \(\beta\) quantifies how tightly the image of the unit input
ball is confined within the resulting ellipsoid.

The threshold \(\beta\le 1\) ensures that the
the lift constructed in Lemma~\ref{lem:gain_caged_lift} can be chosen real-valued, while a strict margin
\(\beta<1\) guarantees injectivity.
\end{remark}

\begin{lemma}[Gain-caged lift via a support/kernel split]
\label{lem:gain_caged_lift}
Let \(f:\mathbb{R}^n\to\mathbb{R}^m\) satisfy \(f(0)=0\), and set \(l\triangleq n+m\).
Assume there exist an orthogonal matrix \(U\in\mathbb{R}^{m\times m}\), a diagonal matrix \(D=\mathrm{diag}(\sigma_1,\dots,\sigma_m)\succ 0\), and a constant \(\beta\in[0,1)\) such that \((U,D)\) \(\beta\)-cages \(f\) in the sense of Equation \eqref{eq:gain_cage_condition}.

Define the rectangular diagonal matrix
\begin{equation}
\label{eq:Sigma_from_D}
\Sigma \triangleq \begin{bmatrix} D & 0_{m\times n}\end{bmatrix}\in\mathbb{R}^{m\times l}.
\end{equation}
Then there exists an injective mapping \(v:\mathbb{R}^n\to\mathbb{R}^l\) satisfying
\(\|v(x)\|_2=\|x\|_2\) for all \(x\in\mathbb{R}^n\), such that
\begin{equation}
\label{eq:gain_caged_factorization}
f(x)=U\Sigma v(x)\quad \forall x\in\mathbb{R}^n.
\end{equation}
\end{lemma}

\begin{proof}
\emph{Construction.}
Since \(\Sigma=[D\ \ 0]\) with \(D\succ 0\), its Moore--Penrose pseudoinverse is
\[
\Sigma^\dagger
=
\begin{bmatrix}
D^{-1}\\
0_{n\times m}
\end{bmatrix}\in\mathbb{R}^{l\times m},
\qquad\text{and hence}\qquad
\Sigma\Sigma^\dagger = I_m.
\]
Define the support component
\begin{equation}
\label{eq:v_support_gain_caged}
v_{\mathrm{support}}(x)\triangleq \Sigma^\dagger U^\top f(x)
=
\begin{bmatrix}
D^{-1}U^\top f(x)\\
0_n
\end{bmatrix}\in\mathbb{R}^l.
\end{equation}
For \(x\neq 0\), define the scalar
\begin{equation}
\label{eq:alpha_def}
\alpha(x)\triangleq
\sqrt{1-\frac{\|v_{\mathrm{support}}(x)\|_2^2}{\|x\|_2^2}},
\end{equation}
and define the kernel component
\begin{equation}
\label{eq:v_kernel_gain_caged}
v_{\mathrm{kernel}}(x)\triangleq
\begin{bmatrix}
0_m\\
\alpha(x)\,x
\end{bmatrix}\in\mathbb{R}^l.
\end{equation}
Finally set
\begin{equation}
\label{eq:v_def_gain_caged}
v(x)\triangleq v_{\mathrm{support}}(x)+v_{\mathrm{kernel}}(x)
\quad\text{for }x\neq 0,
\;
v(0)\triangleq 0.
\end{equation}

\emph{Real-valuedness (radicand positivity).}
By \eqref{eq:gain_cage_condition} and \eqref{eq:v_support_gain_caged},
\[
\|v_{\mathrm{support}}(x)\|_2=\|D^{-1}U^\top f(x)\|_2 \le \beta\|x\|_2
\quad\forall x\neq 0.
\]
Hence the radicand in \eqref{eq:alpha_def} satisfies
\[
1-\frac{\|v_{\mathrm{support}}(x)\|_2^2}{\|x\|_2^2}\ge 1-\beta^2>0,
\]
so \(\alpha(x)\) is well-defined and strictly positive for every \(x\neq 0\).

\emph{Norm preservation.}
The vectors \(v_{\mathrm{support}}(x)\) and \(v_{\mathrm{kernel}}(x)\) have disjoint support
(first \(m\) coordinates versus last \(n\) coordinates), hence are orthogonal in \(\mathbb{R}^l\).
Therefore, for \(x\neq 0\),
\begin{equation}
\begin{aligned}
\|v(x)\|_2^2
&= \|v_{\mathrm{support}}(x)\|_2^2
  + \|v_{\mathrm{kernel}}(x)\|_2^2 \\
&= \|v_{\mathrm{support}}(x)\|_2^2
  + \alpha(x)^2 \|x\|_2^2 \\
&= \|x\|_2^2 .
\end{aligned}
\end{equation}
by the definition of \(\alpha(x)\). Also \(\|v(0)\|_2=0=\|0\|_2\).

\emph{Reconstruction.}
For any \(x\),
\[
\Sigma v_{\mathrm{support}}(x)
=
\Sigma\Sigma^\dagger U^\top f(x)
=
U^\top f(x),
\quad
\Sigma v_{\mathrm{kernel}}(x)=0
\]
(the latter since the last \(n\) columns of \(\Sigma\) are zero). Thus \(\Sigma v(x)=U^\top f(x)\),
and multiplying by \(U\) gives \(U\Sigma v(x)=f(x)\) for all \(x\), including \(x=0\) because \(f(0)=0\).

\emph{Injectivity.}
Let \(x\neq 0\). The last \(n\) coordinates of \(v(x)\) equal \(\alpha(x)x\) with \(\alpha(x)>0\).
Moreover, \(\|v(x)\|_2=\|x\|_2\), so from \(v(x)\) alone we can recover
\[
\alpha(x)=\frac{\|\,v_{m+1:l}(x)\,\|_2}{\|v(x)\|_2},
\qquad
x=\frac{v_{m+1:l}(x)}{\alpha(x)}.
\]
Hence \(v(x_1)=v(x_2)\) implies \(x_1=x_2\), so \(v\) is injective.
\end{proof}

\begin{definition}[Directional gains and aggregation constant]
\label{def:aggregation_constant}
Let \(f:\mathbb{R}^n\to\mathbb{R}^m\) satisfy \(\|f\|_{2\to 2}<\infty\) and \(f(0)=0\).
Fix an orthogonal matrix \(U=[u_1,\dots,u_m]\in\mathbb{R}^{m\times m}\).

\smallskip
\noindent\textbf{Directional induced gains.}
For each \(i=1,\dots,m\), define the induced gain of the scalar functional \(u_i^\top f\) by
\begin{equation}
\label{eq:directional_gains}
c_i(U)\triangleq \big\|u_i^\top f\big\|_{2\to 2}
=
\sup_{x\in\mathbb{R}^n\setminus\{0\}}
\frac{|u_i^\top f(x)|}{\|x\|_2}.
\end{equation}
Let \(D_U\triangleq \mathrm{diag}(c_1(U),\dots,c_m(U))\), and let \(D_U^\dagger\) denote its diagonal Moore--Penrose pseudoinverse.

\smallskip
\noindent\textbf{Aggregation constant.}
The \emph{aggregation constant} of \(f\) in the \(U\)-coordinates is
\begin{equation}
\label{eq:kappa_def}
\kappa(U)\triangleq
\big\|D_U^\dagger U^\top f\big\|_{2\to 2}
=
\sup_{x\in\mathbb{R}^n\setminus\{0\}}
\frac{\big\|D_U^\dagger U^\top f(x)\big\|_2}{\|x\|_2}.
\end{equation}
\end{definition}

\noindent\textbf{Interpretation.}
The diagonal \(D_U\) captures anisotropy in the directional gains \(c_i(U)\), while \(\kappa(U)\) measures how strongly these normalized directions can \emph{co-saturate} for the same input.
Equivalently, \(\kappa(U)\) is the smallest constant \(\kappa\ge 0\) such that
\begin{equation}
\label{eq:kappa_equiv_ineq}
\big\|D_U^\dagger U^\top f(x)\big\|_2 \le \kappa\,\|x\|_2
\quad \forall x\in\mathbb{R}^n\setminus\{0\}.
\end{equation}
\begin{remark}[Zero directional gain implies an identically zero channel]
\label{rem:zero_directional_gain}
Fix \(U=[u_1,\dots,u_m]\) and let \(c_i(U)\) be defined by \eqref{eq:directional_gains}.
If \(c_i(U)=0\), then \(u_i^\top f(x)=0\) for all \(x\in\mathbb{R}^n\).
\end{remark}
\begin{proof}
By definition,
\[
c_i(U)=\sup_{x\neq 0}\frac{|u_i^\top f(x)|}{\|x\|_2}.
\]
If \(c_i(U)=0\), then \(|u_i^\top f(x)|/\|x\|_2=0\) for every \(x\neq 0\), hence \(u_i^\top f(x)=0\) for all \(x\neq 0\).
Also \(u_i^\top f(0)=0\) since \(f(0)=0\).
\end{proof}

\begin{corollary}[Universal bounds for the aggregation constant]
\label{cor:kappa_bounds}
Let \(f:\mathbb{R}^n\to\mathbb{R}^m\) satisfy \(\|f\|_{2\to 2}<\infty\) and \(f(0)=0\), and fix any orthogonal matrix \(U\in\mathbb{R}^{m\times m}\).
Let \(\kappa(U)\) be defined as in Definition~\ref{def:aggregation_constant}.
Then
\begin{equation}
\label{eq:kappa_bounds}
0 \le \kappa(U)\le \sqrt{m}.
\end{equation}
Moreover, if \(f\not\equiv 0\) (equivalently, \(c_i(U)>0\) for some \(i\)), then
\begin{equation}
\label{eq:kappa_bounds_nontrivial}
1 \le \kappa(U)\le \sqrt{m}.
\end{equation}
\end{corollary}

\begin{proof}
\emph{Upper bound.}
Write \(U=[u_1,\dots,u_m]\) and recall \(D_U=\mathrm{diag}(c_1(U),\dots,c_m(U))\).
For any \(x\neq 0\), the \(i\)-th coordinate of \(D_U^\dagger U^\top f(x)\) equals
\[
\big(D_U^\dagger U^\top f(x)\big)_i
=
\begin{cases}
\dfrac{u_i^\top f(x)}{c_i(U)}, & c_i(U)>0,\\[6pt]
0, & c_i(U)=0,
\end{cases}
\]
by the definition of the diagonal pseudoinverse.
If \(c_i(U)>0\), then by definition of \(c_i(U)\),
\[
\left|\frac{u_i^\top f(x)}{c_i(U)}\right|
\le \|x\|_2.
\]
Hence every coordinate of \(D_U^\dagger U^\top f(x)\) has magnitude at most \(\|x\|_2\), so
\[
\big\|D_U^\dagger U^\top f(x)\big\|_2
\le \sqrt{m}\,\|x\|_2
\quad \forall x\neq 0.
\]
Taking the supremum over \(x\neq 0\) yields \(\kappa(U)\le \sqrt{m}\).
Nonnegativity \(\kappa(U)\ge 0\) is immediate.

\emph{Lower bound for nontrivial \(f\).}
If \(f\equiv 0\), then \(D_U=0\), \(D_U^\dagger U^\top f\equiv 0\), and hence \(\kappa(U)=0\).
Otherwise, choose an index \(i_0\) such that \(c_{i_0}(U)>0\).
By definition of the supremum, there exists a sequence \(\{x_k\}_{k\ge 1}\subset\mathbb{R}^n\setminus\{0\}\) such that
\[
\frac{|u_{i_0}^\top f(x_k)|}{\|x_k\|_2}\to c_{i_0}(U).
\]
For these \(x_k\),
\[
\frac{\big\|D_U^\dagger U^\top f(x_k)\big\|_2}{\|x_k\|_2}
\ge
\frac{1}{\|x_k\|_2}\left|\frac{u_{i_0}^\top f(x_k)}{c_{i_0}(U)}\right|
=
\frac{|u_{i_0}^\top f(x_k)|}{c_{i_0}(U)\,\|x_k\|_2}
\to 1.
\]
Taking the supremum over \(x\neq 0\) gives \(\kappa(U)\ge 1\).
\end{proof}

The bounds above show that, for a fixed orthogonal output basis, the aggregation constant $\kappa(U)$ quantifies how strongly distinct output directions of (f) can be simultaneously excited by a single input. In the worst case this co-saturation produces a $\sqrt{m}$ inflation, while in the best nontrivial case the aggregation penalty collapses to its minimal value (1).

This raises a structural question: under what conditions does the aggregation constant attain its minimum? Equivalently, when do the directional gains in a fixed output basis decouple so that no input direction can simultaneously excite more than one output coordinate? The following definition isolates a sufficient regime: directional gains are realized on mutually orthogonal components of the input, eliminating aggregation and enforcing the geometric lower bound \(\kappa(U)=1\).
\begin{definition}[Orthogonal Energy Partition (OEP)]
\label{def:oep}
Let \(f:\mathbb{R}^n\to\mathbb{R}^m\) satisfy \(f(0)=0\) and \(\|f\|_{2\to 2}<\infty\).
We say that \(f\) admits an \emph{orthogonal energy partition} if there exist
\begin{enumerate}
\item an orthogonal matrix \(U=[u_1,\dots,u_m]\in\mathbb{R}^{m\times m}\),
\item symmetric orthogonal projectors \(P_1,\dots,P_m\in\mathbb{R}^{n\times n}\) such that
\begin{equation}
\label{eq:oep_projectors}
\begin{aligned}
P_i^2=P_i,\quad P_i^\top &=P_i,\quad P_iP_j=0\ (i\neq j),\\
\sum_{i=1}^m P_i &= I_n,
\end{aligned}
\end{equation}
\item nonnegative scalars $\bar c_1,\dots,\bar c_m$,
\item scalar maps \(\phi_i:\mathbb{R}^n\to\mathbb{R}\) satisfying
\begin{equation}
\label{eq:oep_phi_contraction}
|\phi_i(z)| \le \|z\|_2 \quad \forall z\in\mathbb{R}^n,
\end{equation}
and, whenever \(P_i\neq 0\),
\begin{equation}
\label{eq:oep_phi_tight}
\sup_{z\in \mathrm{range}(P_i)\setminus\{0\}} \frac{|\phi_i(z)|}{\|z\|_2} = 1,
\end{equation}
\end{enumerate}
such that
\begin{equation}
\label{eq:oep_form}
U^\top f(x) =
\begin{bmatrix}
\bar c_1\,\phi_1(P_1x)\\
\vdots\\
\bar c_m\,\phi_m(P_mx)
\end{bmatrix}
\quad \forall x\in\mathbb{R}^n.
\end{equation}
\end{definition}

\begin{remark}
Definition~\ref{def:oep} means that, after a fixed output rotation \(U\), each output channel of \(f\) depends only on the energy contained in one orthogonal component \(P_i x\) of the input, so that different output coordinates cannot simultaneously draw energy from the same input direction.
\end{remark}

\begin{proposition}[OEP implies \(\kappa(U)=1\) and gives an exact anisotropic cage]
\label{prop:oep_implies_kappa_one}
If \(f\) satisfies Definition~\ref{def:oep}, then for the corresponding \(U\),
the directional gains defined in \eqref{eq:directional_gains} satisfy
\begin{equation}
\label{eq:oep_directional_gains}
\begin{aligned}
c_i(U) &= \bar c_i,
&\quad& \text{for all } i \text{ with } P_i \neq 0,\\
c_i(U) &= 0,
&\quad& \text{for all } i \text{ with } P_i = 0.
\end{aligned}
\end{equation}
Moreover, the aggregation constant satisfies
\begin{equation}
\label{eq:oep_kappa_one}
\kappa(U)=1 \quad \text{whenever } f\not\equiv 0.
\end{equation}
In particular, with \(D_U=\mathrm{diag}(c_1(U),\dots,c_m(U))\),
\begin{equation}
\label{eq:oep_cage_exact}
\big\|D_U^\dagger U^\top f(x)\big\|_2 \le \|x\|_2 \quad \forall x\neq 0,
\end{equation}
so the ``aggregation penalty'' collapses to \(1\) (no \(\sqrt{m}\)-type inflation).
\end{proposition}

\begin{proof}
\emph{Directional gains.}
From \eqref{eq:oep_form}, \(u_i^\top f(x)= \bar c_i\,\phi_i(P_ix)\).
Using \eqref{eq:oep_phi_contraction} and \(\|P_ix\|_2\le \|x\|_2\),
\[
\frac{|u_i^\top f(x)|}{\|x\|_2}
=
\bar c_i\,\frac{|\phi_i(P_ix)|}{\|x\|_2}
\le
\bar c_i\,\frac{\|P_ix\|_2}{\|x\|_2}
\le \bar c_i,
\]
so \(c_i(U)\le \bar c_i\).
If \(P_i=0\), then \(u_i^\top f(x)\equiv 0\) and hence \(c_i(U)=0\).
If \(P_i\neq 0\), take \(x\in \mathrm{range}(P_i)\) so that \(P_ix=x\); then
\begin{equation}
\begin{aligned}
\sup_{x\in \mathrm{range}(P_i)\setminus\{0\}}
\frac{|u_i^\top f(x)|}{\|x\|_2}
&=
\sup_{x\in \mathrm{range}(P_i)\setminus\{0\}}
\frac{|\bar c_i\,\phi_i(P_i x)|}{\|x\|_2} \\
&=
\sup_{x\in \mathrm{range}(P_i)\setminus\{0\}}
\frac{|\bar c_i\,\phi_i(x)|}{\|x\|_2} \\
&=
\bar c_i\,
\sup_{x\in \mathrm{range}(P_i)\setminus\{0\}}
\frac{|\phi_i(x)|}{\|x\|_2} \\
&=
\bar c_i,
\end{aligned}
\end{equation}
where we used $P_i x = x$ on $\mathrm{range}(P_i)$ and \eqref{eq:oep_phi_tight}.
Hence \(c_i(U)= \bar c_i\) for \(P_i\neq 0\), proving \eqref{eq:oep_directional_gains}.

\emph{Aggregation constant.}
Let \(D_U=\mathrm{diag}(c_1(U),\dots,c_m(U))\).
For indices with \(c_i(U)>0\), we have \(c_i(U)=\bar c_i\) and thus
\[
\big(D_U^\dagger U^\top f(x)\big)_i = \frac{\bar{c}_i\,\phi_i(P_ix)}{\bar{ c}_i}=\phi_i(P_ix).
\]
For indices with \(c_i(U)=0\), we have \(u_i^\top f(x)\equiv 0\), so the corresponding entry is \(0\).
Therefore
\[
\|D_U^\dagger U^\top f(x)\|_2^2
=
\sum_{i:\,c_i(U)>0} |\phi_i(P_ix)|^2
\le
\sum_{i=1}^m \|P_ix\|_2^2
=
\|x\|_2^2,
\]
where we used \eqref{eq:oep_phi_contraction} and the orthogonal decomposition property
\(\sum_i \|P_ix\|_2^2 = \|x\|_2^2\) implied by \eqref{eq:oep_projectors}.
This proves \eqref{eq:oep_cage_exact}, hence \(\kappa(U)\le 1\).
If \(f\not\equiv 0\), then by Corollary~\ref{cor:kappa_bounds} we also have \(\kappa(U)\ge 1\),
so \(\kappa(U)=1\), proving \eqref{eq:oep_kappa_one}.
\end{proof}


\begin{corollary}[OEP and minimal aggregation]
\label{cor:kappa_one_svd}
Let \(f:\mathbb{R}^n\to\mathbb{R}^m\) satisfy \(\|f\|_{2\to 2}<\infty\) and \(f(0)=0\).

Assume \(f\) admits an orthogonal energy partition in the sense of Definition~\ref{def:oep}, and let
\(U\in\mathbb{R}^{m\times m}\) be the corresponding orthogonal matrix.
Let \(c_i(U)\) be the directional gains in the \(U\)-coordinates (Definition~\ref{def:aggregation_constant}),
fix any constant \(c_\star>0\), and define \(\widetilde D_U=\mathrm{diag}(\tilde c_1,\dots,\tilde c_m)\succ 0\) by
\[
\tilde c_i \triangleq
\begin{cases}
c_i(U), & c_i(U)>0,\\
c_\star, & c_i(U)=0.
\end{cases}
\]
Then for every \(\gamma>1\), the gain-cage condition of Lemma~\ref{lem:gain_caged_lift} holds with
\(D=\gamma \widetilde D_U\) and \(\beta=1/\gamma<1\).
Consequently, there exists an injective lift \(v:\mathbb{R}^n\to\mathbb{R}^{m+n}\) satisfying \(\|v(x)\|_2=\|x\|_2\) and
\[
f(x)=U\Sigma v(x),\qquad
\Sigma\triangleq \begin{bmatrix}\gamma \widetilde D_U & 0_{m\times n}\end{bmatrix}.
\]
Moreover, since \(\gamma>1\) may be taken arbitrarily close to \(1\), the diagonal gain cage can be made arbitrarily close to the
directional gains \(c_i(U)\) without any \(\sqrt{m}\)-type inflation.
\end{corollary}

\begin{proof}
Assume \(f\) satisfies Definition~\ref{def:oep} with orthogonal matrix \(U\).
By Proposition~\ref{prop:oep_implies_kappa_one}, the exact cage inequality
\[
\big\|D_U^\dagger U^\top f(x)\big\|_2 \le \|x\|_2 \quad \forall x\neq 0
\]
holds (and in particular \(\kappa(U)\le 1\); if \(f\not\equiv 0\) then \(\kappa(U)=1\)).

For any \(x\neq 0\), we claim \(\widetilde D_U^{-1}U^\top f(x)=D_U^\dagger U^\top f(x)\),
where \(D_U=\mathrm{diag}(c_1(U),\dots,c_m(U))\) and \(D_U^\dagger\) is its diagonal pseudoinverse.
Indeed, if \(c_i(U)>0\) then \(\tilde c_i=c_i(U)\) and the \(i\)-th normalized coordinate is
\((u_i^\top f(x))/c_i(U)\).
If \(c_i(U)=0\), then \(u_i^\top f(x)\equiv 0\) (Remark~\ref{rem:zero_directional_gain}),
so the \(i\)-th normalized coordinate is \(0\) regardless of \(\tilde c_i=c_\star>0\).
Thus \(\widetilde D_U^{-1}U^\top f(x)=D_U^\dagger U^\top f(x)\) for all \(x\).

Now fix \(\gamma>1\). Using the preceding display and Proposition~\ref{prop:oep_implies_kappa_one},
for all \(x\neq 0\),
\begin{equation}
\begin{aligned}
\big\|(\gamma \widetilde D_U)^{-1}U^\top f(x)\big\|_2
&=
\frac{1}{\gamma}\,\big\|\widetilde D_U^{-1}U^\top f(x)\big\|_2 \\
&=
\frac{1}{\gamma}\,\big\|D_U^\dagger U^\top f(x)\big\|_2 \\
&\le
\frac{1}{\gamma}\,\|x\|_2.
\end{aligned}
\end{equation}
Hence Lemma~\ref{lem:gain_caged_lift} applies with \(D=\gamma\widetilde D_U\) and \(\beta=1/\gamma<1\),
yielding the stated factorization and injective norm-preserving lift.
Since \(\gamma\downarrow 1\) is permitted, there is no compulsory inflation beyond the directional gains.
\end{proof}

The OEP condition exhausts what is achievable at the level of aggregation in the sense that if an orthogonal energy partition exists, then the aggregation constant is necessarily minimal, \(\kappa(U)=1\). Consequently, no additional structural assumptions can further reduce aggregation in the fixed \(U\)-coordinates.

What remains unconstrained by OEP is the lift itself. In the general nonlinear case, norm preservation is obtained by a support/kernel decomposition: the support component reproduces \(f(x)\), while a kernel component supplies the residual needed to enforce \(|v(x)|_2=|x|_2\). Since the kernel component lies in \(\ker(\Sigma)\), it is invisible at the output and does not affect aggregation.

Linearity restricts this kernel freedom. When the support component is linear and already attains the directional gains, the kernel component cannot encode independent gain geometry. In the linear injective case, it vanishes as \(\gamma\downarrow 1\), and the lift collapses to an orthogonal rotation, recovering the classical SVD geometry.

\begin{corollary}[Linear injective case: Lemma~\ref{lem:gain_caged_lift} recovers the SVD map and the lift collapses to $V^\top$ as $\gamma\downarrow 1$]
\label{cor:linear_injective_lift_recovers_svd}
Let $A\in\mathbb{R}^{m\times n}$ have full column rank ($\mathrm{rank}(A)=n$, hence $m\ge n$), and set $f(x)\triangleq Ax$.
Let $A=U_{\mathrm{svd}}\Sigma_{\mathrm{svd}}V^\top$ be an SVD, where $U_{\mathrm{svd}}\in\mathbb{R}^{m\times m}$ is orthogonal
(thin SVD orthogonally completed if needed), $V\in\mathbb{R}^{n\times n}$ is orthogonal, and
$\Sigma_{\mathrm{svd}}\in\mathbb{R}^{m\times n}$ is rectangular diagonal with singular values
$\sigma_1\ge\cdots\ge\sigma_n>0$.

Fix any $c_\star>0$ and any $\gamma>1$, and define
\begin{equation}
\begin{aligned}
\widetilde D &\triangleq \mathrm{diag}\big(\sigma_1,\dots,\sigma_n,\underbrace{c_\star,\dots,c_\star}_{m-n}\big)\in\mathbb{R}^{m\times m},
\\
D &\triangleq \gamma \widetilde D,
\\
\Sigma &\triangleq \big[\,D\;\;0_{m\times n}\,\big]\in\mathbb{R}^{m\times(m+n)}.
\end{aligned}
\end{equation}
Let $v:\mathbb{R}^n\to\mathbb{R}^{m+n}$ be the lift constructed in Lemma~\ref{lem:gain_caged_lift} with $U=U_{\mathrm{svd}}$ and this $\Sigma$.

Then:

\smallskip
\noindent\textbf{(i) Exact output-side identity.}
For all $x\in\mathbb{R}^n$,
\[
\Sigma\,v(x)=\Sigma_{\mathrm{svd}}V^\top x,
\]
and hence
\[
U_{\mathrm{svd}}\Sigma\,v(x)=U_{\mathrm{svd}}\Sigma_{\mathrm{svd}}V^\top x=Ax.
\]

\smallskip
\noindent\textbf{(ii) Explicit form and $\gamma\downarrow 1$ limit.}
The lift is linear and equals
\[
v(x)=
\begin{bmatrix}
\frac{1}{\gamma}\begin{bmatrix}V^\top x\\0_{m-n}\end{bmatrix}\\[1mm]
\sqrt{1-\frac{1}{\gamma^2}}\,x
\end{bmatrix}\in\mathbb{R}^{m+n},
\]
so as $\gamma\downarrow 1$ (the relevant limit since Lemma~\ref{lem:gain_caged_lift} requires $\gamma>1$) we have
\[
v(x)\to
\begin{bmatrix}
\begin{bmatrix}V^\top x\\0_{m-n}\end{bmatrix}\\[1mm]
0_n
\end{bmatrix},
\]
i.e., the kernel component vanishes and the lift reduces to the SVD right rotation (embedded in $\mathbb{R}^{m+n}$).
\end{corollary}

\begin{proof}
By Lemma~\ref{lem:gain_caged_lift}, the lift is constructed via the support/kernel split
\[
v(x)=v_{\mathrm{support}}(x)+v_{\mathrm{kernel}}(x),
\qquad
v_{\mathrm{support}}(x)=\Sigma^\dagger U_{\mathrm{svd}}^\top Ax,
\]
with $\Sigma^\dagger=\big[\;D^{-1}\;\;0_{n\times m}\;\big]^\top$ and $v_{\mathrm{kernel}}(x)=\big[\,0_m\;\;\alpha(x)x\,\big]^\top$,
where
\[
\alpha(x)=\sqrt{1-\frac{\|v_{\mathrm{support}}(x)\|_2^2}{\|x\|_2^2}}
\qquad(x\neq 0).
\]

Since $A=U_{\mathrm{svd}}\Sigma_{\mathrm{svd}}V^\top$, we have
\[
U_{\mathrm{svd}}^\top A=\Sigma_{\mathrm{svd}}V^\top.
\]
Moreover, for full column rank ($m\ge n$),
$\Sigma_{\mathrm{svd}}V^\top x=\big[\;\mathrm{diag}(\sigma_1,\dots,\sigma_n)V^\top x\;\;0_{m-n}\big]^\top$.
Thus
\begin{align*}
v_{\mathrm{support}}(x)
&=
\begin{bmatrix}
D^{-1}U_{\mathrm{svd}}^\top Ax\\[0.5mm]
0_n
\end{bmatrix} \\
&=
\begin{bmatrix}
D^{-1}\Sigma_{\mathrm{svd}}V^\top x\\[0.5mm]
0_n
\end{bmatrix} \\
&=
\begin{bmatrix}
\frac{1}{\gamma}\begin{bmatrix}V^\top x\\0_{m-n}\end{bmatrix}\\[0.5mm]
0_n
\end{bmatrix},
\end{align*}
where the entries involving $c_\star$ multiply zeros and hence do not affect the expression.

Because $V$ is orthogonal,
$\|v_{\mathrm{support}}(x)\|_2=\frac{1}{\gamma}\|V^\top x\|_2=\frac{1}{\gamma}\|x\|_2$, and therefore
\[
\alpha(x)=\sqrt{1-\frac{1}{\gamma^2}}
\quad\text{(constant in $x$)}.
\]
Substituting into $v_{\mathrm{kernel}}(x)$ yields the explicit formula in part (ii).

For the exact identity in part (i), note that $\Sigma v_{\mathrm{kernel}}(x)=0$ by construction (the last $n$ columns of $\Sigma$ are zero), hence
\[
\Sigma v(x)=\Sigma v_{\mathrm{support}}(x)=\Sigma\Sigma^\dagger U_{\mathrm{svd}}^\top Ax
=U_{\mathrm{svd}}^\top Ax=\Sigma_{\mathrm{svd}}V^\top x,
\]
and multiplying by $U_{\mathrm{svd}}$ gives $U_{\mathrm{svd}}\Sigma v(x)=Ax$.
Finally, as $\gamma\downarrow 1$, $\frac{1}{\gamma}\to 1$ and $\sqrt{1-\frac{1}{\gamma^2}}\to 0$, proving the stated limit.
\end{proof}

\begin{corollary}[Linear non-injective case: row-space agreement but nullspace information is stored in the kernel block]
\label{cor:linear_rank_deficient_lift_behavior}
Let $A\in\mathbb{R}^{m\times n}$ have rank $r<n$, and set $f(x)\triangleq Ax$.
Let $A=U_{\mathrm{svd}}\Sigma_{\mathrm{svd}}V^\top$ be an SVD with singular values
$\sigma_1\ge\cdots\ge\sigma_r>0$ and $\sigma_{r+1}=\cdots=0$.
Write $V=[V_r\;\;V_0]$ with $V_r\in\mathbb{R}^{n\times r}$ (row-space basis) and $V_0\in\mathbb{R}^{n\times(n-r)}$ (nullspace basis).

Fix any $c_\star>0$ and $\gamma>1$, define
\begin{equation}
\begin{aligned}
\widetilde D &\triangleq \mathrm{diag}\big(\sigma_1,\dots,\sigma_r,\underbrace{c_\star,\dots,c_\star}_{m-r}\big),
\\
D&\triangleq \gamma \widetilde D,
\\
\Sigma &\triangleq \big[\,D\;\;0_{m\times n}\,\big],
\end{aligned}
\end{equation}
and let $v:\mathbb{R}^n\to\mathbb{R}^{m+n}$ be the Lemma~\ref{lem:gain_caged_lift} lift with $U=U_{\mathrm{svd}}$ and this $\Sigma$.
Then for all $x\in\mathbb{R}^n$ we still have the exact identity
\[
\Sigma\,v(x)=\Sigma_{\mathrm{svd}}V^\top x,
\]
but the lift behaves as follows:
\begin{equation}
\begin{aligned}
v_{\mathrm{support}}(x) &=
\begin{bmatrix}
\frac{1}{\gamma}\begin{bmatrix}V_r^\top x\\0_{m-r}\end{bmatrix}\\[0.5mm]
0_n
\end{bmatrix},
\\
\|v_{\mathrm{support}}(x)\|_2 &=\frac{1}{\gamma}\|V_r^\top x\|_2,
\\
\alpha(x) &= \sqrt{1-\frac{1}{\gamma^2}\frac{\|V_r^\top x\|_2^2}{\|x\|_2^2}}.
\end{aligned}
\end{equation}
In particular, if $x\in\ker(A)\setminus\{0\}$ (equivalently $V_r^\top x=0$), then $v_{\mathrm{support}}(x)=0$ and $v(x)=\big[\,0_m\;\;x\,\big]^\top$.
Consequently, as $\gamma\downarrow 1$ the kernel block generally \emph{does not} vanish (it vanishes iff $x$ lies entirely in the row space).
\end{corollary}

\begin{proof}
As in the proof of Corollary~\ref{cor:linear_injective_lift_recovers_svd}, Lemma~\ref{lem:gain_caged_lift} gives
$v_{\mathrm{support}}(x)=\big[\,D^{-1}U_{\mathrm{svd}}^\top Ax\;\;0_n\,\big]^\top$ and $\Sigma v(x)=U_{\mathrm{svd}}^\top Ax$.

Using $U_{\mathrm{svd}}^\top A=\Sigma_{\mathrm{svd}}V^\top$ and the SVD structure,
$\Sigma_{\mathrm{svd}}V^\top x=\big[\;\mathrm{diag}(\sigma_1,\dots,\sigma_r)V_r^\top x\;\;0_{m-r}\big]^\top$,
so multiplying by $D^{-1}=(1/\gamma)\,\mathrm{diag}(1/\sigma_1,\dots,1/\sigma_r,1/c_\star,\dots)$ yields
\[
D^{-1}\Sigma_{\mathrm{svd}}V^\top x=\frac{1}{\gamma}\begin{bmatrix}V_r^\top x\\0_{m-r}\end{bmatrix},
\]
which gives the stated expression for $v_{\mathrm{support}}(x)$ and its norm.
The formula for $\alpha(x)$ follows immediately from its definition
$\alpha(x)=\sqrt{1-\|v_{\mathrm{support}}(x)\|_2^2/\|x\|_2^2}$.
If $x\in\ker(A)\setminus\{0\}$ then $V_r^\top x=0$, so $v_{\mathrm{support}}(x)=0$ and $\alpha(x)=1$, hence $v(x)=[0_m;\,x]^\top$.
Finally, since $\|V_r^\top x\|_2<\|x\|_2$ whenever $x$ has a nullspace component, $\alpha(x)$ does not generally converge to $0$ as $\gamma\downarrow 1$.
\end{proof}

\smallskip
Corollaries~3--4 close the loop with linear theory: when $f(x)=Ax$ and $U$ is chosen as the SVD left basis, the directional gains coincide with the singular values, and the Lemma~1 lift reproduces the SVD output-side map exactly (with the lift collapsing to $V^\top$ in the injective case and necessarily retaining a nontrivial kernel block in the rank-deficient case).
In other words, for linear maps the SVD provides a \emph{canonical} choice of output coordinates that simultaneously (i) orders the directional gains and (ii) yields a sharp diagonal cage.

For a general nonlinear $f$, there is no a priori analogue of the SVD to indicate \emph{which} orthogonal output rotation $U$ should be used in Definitions~2--3. This is because both the directional gains $c_i(U)$ and the aggregation constant $\kappa(U)$ depend on this choice, and different rotations can produce very different anisotropic portraits.
Thus, before applying the gain-caging lemma as a reusable tool, we need a principled way to select an orthogonal output basis directly from $f$---one that plays the same organizational role that $U_{\mathrm{svd}}$ plays in the linear case by exposing, and ordering, the most amplified output directions.
The next corollary provides a stagewise extremal construction of orthogonal directions that induces an ordered set of directional gains and therefore a canonical diagonal cage (with $\gamma>\kappa(U)$ supplying the only slack needed by Lemma~1).

\begin{corollary}[Extremal-direction orthogonal coordinates and an anisotropic gain cage]
\label{cor:extremal_direction_coordinates}
Let \(f:\mathbb{R}^n\to\mathbb{R}^m\) satisfy \(\|f\|_{2\to 2}<\infty\) and \(f(0)=0\).

\smallskip
\noindent\textbf{Stagewise extremal values and orthogonal directions.}
Let \(\mathcal{U}_1\triangleq \{u\in\mathbb{R}^m:\|u\|_2=1\}\).
Define
\begin{equation}
\label{eq:extremal_u1_new}
\begin{aligned}
L_1 &\triangleq \max_{u\in\mathcal{U}_1} \big\|u^\top f\big\|_{2\to 2},\\
u_1 &\in \arg\max_{u\in\mathcal{U}_1} \big\|u^\top f\big\|_{2\to 2}.
\end{aligned}
\end{equation}
Recursively, for \(k=2,\dots,m\), define
\[
\mathcal{U}_k \triangleq \{u\in\mathcal{U}_1:\; u\perp \mathrm{span}\{u_1,\dots,u_{k-1}\}\},
\]
and define
\begin{equation}
\label{eq:extremal_uk_new}
\begin{aligned}
L_k &\triangleq \max_{u\in\mathcal{U}_k} \big\|u^\top f\big\|_{2\to 2},\\
u_k &\in \arg\max_{u\in\mathcal{U}_k} \big\|u^\top f\big\|_{2\to 2},
\qquad k=2,\dots,m.
\end{aligned}
\end{equation}
Set \(U\triangleq [u_1\ \cdots\ u_m]\in\mathbb{R}^{m\times m}\).
(If the argmax sets are not singletons, any choice of maximizers defines a valid \(U\); the conclusions below hold for any such choice.)

\smallskip
\noindent\textbf{Connection to global directional gains.}
Let \(c_i(U)\) and \(\kappa(U)\) be as in Definition~\ref{def:aggregation_constant}.
Then, for the above \(U\),
\begin{equation}
\label{eq:ci_equals_Li_new}
c_i(U)=\big\|u_i^\top f\big\|_{2\to 2}=L_i,\qquad i=1,\dots,m.
\end{equation}
In particular, the extremal values are ordered
\begin{equation}
\begin{aligned}
L_1\ge L_2\ge &\cdots \ge L_m
\qquad\text{and hence} \\
c_1(U)\ge c_2(U)\ge &\cdots \ge c_m(U),
\end{aligned}
\end{equation}
and the top value recovers the induced norm:
\begin{equation}
\label{eq:c1_equals_induced_norm}
c_1(U)=L_1=\|f\|_{2\to 2}.
\end{equation}

\smallskip
\noindent\textbf{Anisotropic cage and factorization.}
Fix any constant \(c_\star>0\) and define the \emph{strictly positive} diagonal matrix
\(\widetilde D_U=\mathrm{diag}(\tilde c_1,\dots,\tilde c_m)\succ 0\) by
\begin{equation}
\label{eq:Dtilde_def}
\tilde c_i \triangleq
\begin{cases}
c_i(U), & c_i(U)>0,\\
c_\star, & c_i(U)=0.
\end{cases}
\end{equation}
Fix any \(\gamma>\kappa(U)\) and define
\begin{equation}
\label{eq:Sigma_gamma_Dtilde}
\Sigma \triangleq \begin{bmatrix}\gamma \widetilde D_U & 0_{m\times n}\end{bmatrix}\in\mathbb{R}^{m\times (m+n)}.
\end{equation}
Then the gain-cage condition of Lemma~\ref{lem:gain_caged_lift} holds with \((U,D)=(U,\gamma \widetilde D_U)\) and
\(\beta=\kappa(U)/\gamma<1\). Consequently, there exists an injective lift
\(v:\mathbb{R}^n\to\mathbb{R}^{m+n}\) point-wise satisfying \(\|v(x)\|_2=\|x\|_2\) for all \(x\) such that
\[
f(x)=U\Sigma v(x)\quad \forall x\in\mathbb{R}^n.
\]
\end{corollary}

\begin{proof}
\emph{Existence of maximizers.}
We first show that the map \(u\mapsto \|u^\top f\|_{2\to 2}\) is Lipschitz on the unit sphere.
Let \(u,v\in\mathbb{R}^m\) satisfy \(\|u\|_2=\|v\|_2=1\). Then, using the reverse triangle inequality and submultiplicativity,
\begin{equation}
\begin{aligned}
\Big|\big\|u^\top f\big\|_{2\to 2}-\big\|v^\top f\big\|_{2\to 2}\Big|
&\le
\big\|(u-v)^\top f\big\|_{2\to 2} \\
&=
\sup_{x\neq 0}\frac{|(u-v)^\top f(x)|}{\|x\|_2} \\
&\le
\|u-v\|_2\,\sup_{x\neq 0}\frac{\|f(x)\|_2}{\|x\|_2} \\
&=
\|u-v\|_2\,\|f\|_{2\to 2}.
\end{aligned}
\end{equation}
Hence \(u\mapsto \|u^\top f\|_{2\to 2}\) is continuous on the unit sphere.
Each feasible set \(\mathcal{U}_k\) is a closed subset of the unit sphere (thus compact), so the maxima in
\eqref{eq:extremal_u1_new}--\eqref{eq:extremal_uk_new} are attained.

\emph{Ordering of the stagewise extrema.}
By construction, \(\mathcal{U}_{k}\subseteq \mathcal{U}_{k-1}\) for \(k\ge 2\), hence
\(L_k=\max_{u\in\mathcal{U}_k} \|u^\top f\|_{2\to 2}\le \max_{u\in\mathcal{U}_{k-1}} \|u^\top f\|_{2\to 2}=L_{k-1}\).
Thus \(L_1\ge \cdots \ge L_m\).

\emph{Identity \(c_i(U)=L_i\).}
By Definition~\ref{def:aggregation_constant},
\(c_i(U)=\|u_i^\top f\|_{2\to 2}\).
By the construction \eqref{eq:extremal_u1_new}--\eqref{eq:extremal_uk_new}, we also have \(\|u_i^\top f\|_{2\to 2}=L_i\).
This proves \eqref{eq:ci_equals_Li_new}, and the ordering of the \(c_i(U)\) follows from that of the \(L_i\).

\emph{Top value equals the induced norm.}
For any fixed \(x\neq 0\),
\[
\|f(x)\|_2 = \sup_{\|u\|_2=1} u^\top f(x),
\]
and hence
\begin{equation}
\begin{aligned}
\|f\|_{2\to 2}
&=
\sup_{x\neq 0}\frac{\|f(x)\|_2}{\|x\|_2} \\
&=
\sup_{x\neq 0}\sup_{\|u\|_2=1}\frac{|u^\top f(x)|}{\|x\|_2} \\
&=
\sup_{\|u\|_2=1}\sup_{x\neq 0}\frac{|u^\top f(x)|}{\|x\|_2}.
\end{aligned}
\end{equation}
For any fixed unit vector \(u\), we have
\[
\sup_{x\neq 0}\frac{|u^\top f(x)|}{\|x\|_2}
=
\big\|u^\top f\big\|_{2\to 2},
\]
so the preceding display becomes
\[
\|f\|_{2\to 2}
=
\sup_{\|u\|_2=1}\big\|u^\top f\big\|_{2\to 2}.
\]
By definition, \(L_1=\max_{\|u\|_2=1}\|u^\top f\|_{2\to 2}=\sup_{\|u\|_2=1}\|u^\top f\|_{2\to 2}\), so \(L_1=\|f\|_{2\to 2}\).
Therefore \(c_1(U)=L_1=\|f\|_{2\to 2}\), proving \eqref{eq:c1_equals_induced_norm}.

\emph{Gain-cage inequality.}
For any \(x\neq 0\),
\[
\big\|(\gamma \widetilde D_U)^{-1}U^\top f(x)\big\|_2
=
\frac{1}{\gamma}\,\big\|\widetilde D_U^{-1}U^\top f(x)\big\|_2.
\]
If \(c_i(U)>0\), then \(\tilde c_i=c_i(U)\) and the \(i\)-th normalized coordinate is \((u_i^\top f(x))/c_i(U)\).
If \(c_i(U)=0\), then \(u_i^\top f(x)\equiv 0\), so the \(i\)-th normalized coordinate is \(0\) regardless of \(\tilde c_i=c_\star>0\).
Thus, for all \(x\),
\[
\widetilde D_U^{-1}U^\top f(x)=D_U^\dagger U^\top f(x),
\]
and therefore, by Definition~\ref{def:aggregation_constant},
\[
\big\|(\gamma \widetilde D_U)^{-1}U^\top f(x)\big\|_2
=
\frac{1}{\gamma}\,\big\|D_U^\dagger U^\top f(x)\big\|_2
\le
\frac{1}{\gamma}\,\kappa(U)\,\|x\|_2.
\]
Since \(\gamma>\kappa(U)\), letting \(\beta\triangleq \kappa(U)/\gamma\in[0,1)\) yields
\[
\big\|(\gamma \widetilde D_U)^{-1}U^\top f(x)\big\|_2 \le \beta\,\|x\|_2\quad \forall x\neq 0.
\]
Thus the hypotheses of Lemma~\ref{lem:gain_caged_lift} hold with \(D=\gamma \widetilde D_U\), and the stated factorization follows.
\end{proof}

\begin{remark}[Choosing \(c_\star\)]
The constant \(c_\star>0\) only appears in coordinates \(i\) for which \(c_i(U)=0\), i.e., for which \(u_i^\top f(x)\equiv 0\).
Hence \(c_\star\) does not affect the gain-cage inequality or the factorization.
For maximum interpretability, one may take \(c_\star\) to be arbitrarily small (but strictly positive) to reflect that these directions carry no output energy.
In computational settings, \(c_\star\) may be inflated to avoid poor conditioning of $D_U$.
\end{remark}

\medskip
\noindent\emph{Control-facing extension: norm preservation in the \(u\) argument.}
Before applying the preceding results to control systems, we need a two-argument analogue in which the lift is pointwise norm-preserving in the \emph{control} variable \(u\).
In what follows, we apply the gain-caging lemma to the control-dependent contribution \(f_u(x,u)\) (which satisfies \(f_u(x,0)=0\)).

\begin{lemma}[Norm preservation in the control argument]
\label{lem:norm_preserve_u}
Let \(g:\mathbb{R}^n\times\mathbb{R}^p\to\mathbb{R}^m\) satisfy \(g(x,0)=0\) for all \(x\) and
\begin{equation}
\label{eq:finite_u_gain}
\|g\|_{u} \triangleq
\sup_{x\in\mathbb{R}^n,\,u\in\mathbb{R}^p\setminus\{0\}}
\frac{\|g(x,u)\|_{2}}{\|u\|_{2}} < \infty.
\end{equation}
Set \(l\triangleq p+m\).
Then there exist an orthogonal matrix \(U\in\mathbb{R}^{m\times m}\), a diagonal \(\Sigma=\begin{bmatrix}D & 0_{m\times p}\end{bmatrix}\in\mathbb{R}^{m\times l}\) with \(D\succ 0\), and a mapping \(v:\mathbb{R}^n\times\mathbb{R}^p\to\mathbb{R}^l\) such that
\begin{equation}
\label{eq:norm_preserve_u}
\|v(x,u)\|_2=\|u\|_2 \quad \forall (x,u),
\end{equation}
and
\begin{equation}
\label{eq:factorization_u}
g(x,u)=U\Sigma v(x,u)\quad \forall (x,u).
\end{equation}
\end{lemma}

\begin{proof}

\emph{Construction.}
Fix any orthogonal matrix $U\in\mathbb{R}^{m\times m}$.
For proof of existence, any orthogonal $U$ will do, though for many applications a U defined as in Corollary \ref{cor:extremal_direction_coordinates} would be maximally useful.
Define the \emph{directional induced gains in the $u$-argument} (at fixed output coordinates $U$) by
\[
c_i^{(u)}(U)\;\triangleq\;
\sup_{x\in\mathbb{R}^n,\;u\in\mathbb{R}^p\setminus\{0\}}
\frac{\big|e_i^\top U^\top g(x,u)\big|}{\|u\|_2},
\qquad i=1,\dots,m,
\]
and define a strictly positive diagonal cage
\begin{equation}
\begin{aligned}
\widetilde D_U \;&\triangleq\; \mathrm{diag}(\widetilde c_1,\dots,\widetilde c_m)\succ 0,
\\
\widetilde c_i \;&\triangleq\;
\begin{cases}
c_i^{(u)}(U), & c_i^{(u)}(U)>0,\\
c_\star, & c_i^{(u)}(U)=0,
\end{cases}
\end{aligned}
\end{equation}
for any fixed constant $c_\star>0$.
Now define the corresponding aggregation constant
\[
\kappa_u(U)\;\triangleq\;
\sup_{x\in\mathbb{R}^n,\;u\neq 0}
\frac{\|\widetilde D_U^{-1}U^\top g(x,u)\|_2}{\|u\|_2}
\;<\;\infty,
\]
and fix any $\gamma>\kappa_u(U)$.
Set
\[
D \;\triangleq\; \gamma\,\widetilde D_U \;\succ\; 0,
\qquad
\Sigma \;\triangleq\;
\begin{bmatrix}D & 0_{m\times p}\end{bmatrix}\in\mathbb{R}^{m\times(m+p)}.
\]
Since $\Sigma=[D\ \ 0]$ with $D\succ 0$, its Moore--Penrose pseudoinverse is
\[
\Sigma^\dagger=\begin{bmatrix}D^{-1}\\[2pt]0_{p\times m}\end{bmatrix},
\qquad\text{and hence}\qquad
\Sigma\Sigma^\dagger=I_m.
\]

Define the support component
\begin{equation}
\begin{aligned}
v_{\mathrm{support}}(x,u)&\triangleq \Sigma^\dagger U^\top g(x,u) \\
&=
\begin{bmatrix}
D^{-1}U^\top g(x,u)\\
0_p
\end{bmatrix}\in\mathbb{R}^{m+p}.
\end{aligned}
\end{equation}

For $u\neq 0$, define
\begin{equation}
\begin{aligned}
\alpha(x,u) &\triangleq \sqrt{1-\frac{\|v_{\mathrm{support}}(x,u)\|_2^2}{\|u\|_2^2}},
\\
v_{\mathrm{kernel}}(x,u)&\triangleq
\begin{bmatrix}
0_m\\
\alpha(x,u)\,u
\end{bmatrix}\in\mathbb{R}^{m+p}.
\end{aligned}
\end{equation}
Finally set $v(x,u)\triangleq v_{\mathrm{support}}(x,u)+v_{\mathrm{kernel}}(x,u)$ for $u\neq 0$, and define $v(x,0)\triangleq 0$.

\emph{Real-valuedness (radicand positivity).}
For any $x$ and $u\neq 0$,
\begin{equation}
\begin{aligned}
\|v_{\mathrm{support}}(x,u)\|_2
&=\|D^{-1}U^\top g(x,u)\|_2 \\
&=\frac{1}{\gamma}\,\|\widetilde D_U^{-1}U^\top g(x,u)\|_2 \\
&\le \frac{\kappa_u(U)}{\gamma}\,\|u\|_2.
\end{aligned}
\end{equation}
Let $\beta\triangleq \kappa_u(U)/\gamma\in[0,1)$.
Then
\[
1-\frac{\|v_{\mathrm{support}}(x,u)\|_2^2}{\|u\|_2^2}\;\ge\;1-\beta^2\;>\;0,
\]
so $\alpha(x,u)$ is well-defined and real-valued.

\emph{Norm preservation in the \(u\) argument.}
The vectors \(v_{\mathrm{support}}(x,u)\) and \(v_{\mathrm{kernel}}(x,u)\) have disjoint support (first \(m\) versus last \(p\) coordinates) and are therefore orthogonal.
Hence for \(u\neq 0\),
\begin{equation}
\begin{aligned}
\|v(x,u)\|_2^2
&=
\|v_{\mathrm{support}}(x,u)\|_2^2+\|v_{\mathrm{kernel}}(x,u)\|_2^2 \\
&=
\|v_{\mathrm{support}}(x,u)\|_2^2+\alpha(x,u)^2\|u\|_2^2 \\
&=
\|u\|_2^2,
\end{aligned}
\end{equation}
by the definition of \(\alpha(x,u)\). Also \(\|v(x,0)\|_2=0=\|0\|_2\).

\emph{Reconstruction.}
For any \(x\) and \(u\),
\begin{align}
\Sigma v_{\mathrm{support}}(x,u)
&=
\Sigma\Sigma^\dagger U^\top g(x,u)
=
U^\top g(x,u), \\
\Sigma v_{\mathrm{kernel}}(x,u)&=0
\end{align}
(the latter since the last \(p\) columns of \(\Sigma\) are zero).
Thus \(\Sigma v(x,u)=U^\top g(x,u)\), and multiplying by \(U\) yields
\[
U\Sigma v(x,u)=g(x,u)
\quad \forall (x,u).
\]
At \(u=0\), this holds because \(g(x,0)=0\) and \(v(x,0)=0\).
\end{proof}
\begin{remark}[Injectivity in the control argument]
Since \(\beta<1\) implies \(\alpha(x,u)>0\) for all \(u\neq 0\), the last \(p\) coordinates of \(v(x,u)\) equal a strictly positive scaling of \(u\), so \(u\) can be uniquely recovered from \(v(x,u)\); hence the lift is injective in the control argument.
\end{remark}

\begin{remark}[Implementability]
Norm-preserving lifts can be implemented by composing any learned vector map with a final renormalization step that rescales the output to match the input norm (with an $\varepsilon$-safeguard at the origin). This enforces pointwise norm preservation by construction while remaining compatible with backpropagation.
\end{remark}

\begin{quote}
\noindent\textbf{Key point.}
Any control-dependent contribution \(g(x,u)\) with finite induced gain in \(u\) can be written as a \emph{constant} matrix
\(B \triangleq U\Sigma\) times an instantaneous lifted input \(v(x,u)\) satisfying \(\|v(x,u)\|_{2}=\|u\|_{2}\) pointwise.
This is the mechanism that later restores metric fidelity for Hankel/BT-based certificates.
\end{quote}

%% file: subsections/results_gramians_strong_assumptions.tex
\subsection{Main Results Part 1: Nonlinear Gramians under Strong Assumptions}

In this subsection we translate the GSVD-based input calibration from Section~III.A into a certified reduction bound for a class of nonlinear control systems.
The core objective is to construct an LTI surrogate whose Hankel singular values (HSVs) yield a valid $H_\infty$ truncation certificate in the \emph{physical} input metric.

Throughout Part~1 we assume the autonomous dynamics admit an \emph{exact}, finite-dimensional Koopman generator:
$\dot{\varphi}_0(x)=A\varphi(x)$ for some finite $q$.
This assumption is intentionally idealized; Part~2 relaxes it by introducing an explicit closure residual.

In the following we (i) define the idealized class $\mathcal{J}$, (ii) show that every $G\in\mathcal{J}$ admits an LTI-like lifted representation with a \emph{pointwise input-energy calibrated} lifted input $v(x,u)$ (Theorem~\ref{theorem:norm_preserving}), (iii) define the associated LTI system $G^L$ that is linear in the calibrated input channel and invoke balanced truncation on $G^L$, and (iv) derive a non-feedback induced $H_\infty$ reduction bound by decomposing the total error into a calibration-controlled input-mismatch term plus the classical balanced truncation term for $G^L$, yielding the final certificate (Theorem~\ref{theorem:final_bound}).

\begin{definition}[Nonlinear control systems with induced-norm regularity and exact finite-dimensional Koopman closure]
\label{def:def_of_j}
Let $\mathcal{J}$ be the set of all systems $G$ with state-space representation
\begin{align}
\dot{x} &= f(x,u),\\
y &= h(x),
\end{align}
such that the autonomous dynamics $\dot{x}=f(x,0)$ admit an exact finite-dimensional Koopman representation in a state-inclusive lifting $\varphi$.
Fix a forward-invariant compact set $\mathcal X\subset\mathbb R^n$ containing $0$ such that all trajectories of interest remain in $\mathcal X$.
Then, assume there exist $\varphi:\mathbb{R}^n\to\mathbb{R}^q$ (with finite $q$) and a constant matrix $A\in\mathbb{R}^{q\times q}$ such that
\begin{align}
D\varphi(x)\,f(x,0)=A\varphi(x), \qquad \forall x\in\mathcal X.
\end{align}
Equivalently, for any initial condition $x(0)\in\mathcal X$, the identity holds
along the corresponding autonomous trajectory that remains in $\mathcal X$.

The tuple $(f,h,\varphi,A)$ is required to satisfy:
\begin{itemize}
    
    \item \textbf{Baseline dynamics (regularity, equilibrium, and stability of the lifted generator).}
    \begin{itemize}
        \item $x=0$ is an asymptotically stable hyperbolic equilibrium of $\dot x=f(x,0)$, and we fix a forward-invariant compact set $\mathcal X$ containing $0$ such that all trajectories of interest remain in $\mathcal X$,
        \item $f$ is globally Lipschitz in the input $u$, uniformly over $x\in\mathcal X$, i.e., there exists $L_u<\infty$ such that
        \begin{equation}
        \begin{aligned}
        \|f(x,u_1)-f(x,u_2)\|_2
        &\le L_u \|u_1-u_2\|_2, \\
        &\qquad \forall x \in \mathcal X,\ 
        \forall u_1,u_2 \in \mathbb R^p .
        \end{aligned}
        \end{equation}
        and $f$ is (at least) locally Lipschitz in $x$ on $\mathcal X$.
        \item $f(0,0)=0$,
        \item the Koopman generator $A$ is Hurwitz.
    \end{itemize}
    \item \textbf{Koopman lifting (finite-dimensional, state-inclusive, smooth).}
    \begin{itemize}
        \item $\varphi:\mathbb{R}^n\to\mathbb{R}^q$ is finite dimensional ($q<\infty$) and continuously differentiable on $\mathcal X$ (in particular, $D\varphi(x)$ exists for all $x\in\mathcal X$),
        \item $\varphi$ is state-inclusive: $\varphi(x)=\begin{bmatrix}x\\ \varphi_{\mathrm{lift}}(x)\end{bmatrix}$ for some $\varphi_{\mathrm{lift}}:\mathbb{R}^n\to\mathbb{R}^{q-n}$,
        \item $\varphi(0)=0$,
        \item fix an invariant compact set $\mathcal X$ of interest and define
$M_{\varphi}\triangleq \sup_{x\in\mathcal X}\|D\varphi(x)\|_2 < \infty$.
        \item the nontrivial lifted coordinates satisfy $D\varphi_{\mathrm{lift}}(0)=0_{(q-n)\times n}$,
    \end{itemize}

    \item \textbf{Output map (compatible with the lift).}
    \begin{itemize}
        \item $h$ lies in the span of $\varphi$.
    \end{itemize}
\end{itemize}

\noindent As a convention, we refer to the dimensions as:
\begin{itemize}
\item $x \in \mathbb{R}^n$,
\item $u \in \mathbb{R}^p$,
\item $\varphi(x) \in \mathbb{R}^q$,
\item $y \in \mathbb{R}^m$.
\end{itemize}
\end{definition}

\begin{remark}
Definition~\ref{def:def_of_j} is stated at the level of structural properties rather than a specific parameterization.
The assumptions on $\varphi$ are compatible with standard data-driven constructions; for example, $\varphi$ may be realized by a neural network whose final layer is linear and whose preceding activations are normalized to preserve the input $2$-norm, ensuring $\varphi(0)=0$.
Importantly, no linearity or affine structure is assumed in the state evolution itself: the Koopman closure assumption applies only to the autonomous dynamics in the lifted coordinates.
\end{remark}

We first show that, under the exact closure assumption in Definition~\ref{def:def_of_j}, the control influence can be written in an \emph{affine-like} lifted form with a pointwise input-energy calibrated lifted input.

\begin{theorem}[Pointwise norm-preserving, affine-like control inputs in non-affine systems]
\label{theorem:norm_preserving}
For every system $G\in\mathcal{J}$, there exist non-unique matrices
$A\in\mathbb{R}^{q\times q}$ (Hurwitz), $B\in\mathbb{R}^{q\times(p+q)}$, and
$C\in\mathbb{R}^{m\times q}$, and a mapping
$v:\mathcal X\times\mathbb{R}^p\to\mathbb{R}^{p+q}$ satisfying
\begin{align}
\|v(x,u)\|_2=\|u\|_2,\quad \forall x\in \mathcal{X},\ u\in\mathbb{R}^p,
\end{align}
such that $G$ admits the lifted representation
\begin{equation}
\begin{aligned}
\label{eq:lifted_affine_like}
D\varphi(x)\,f(x,u) &= A\varphi(x)+Bv(x,u),\\
y &= C\varphi(x).
\end{aligned}
\end{equation}
\end{theorem}

\begin{proof}
We (i) split $f$ into autonomous and control-induced parts,
(ii) isolate the corresponding control-induced term in $\varphi$-coordinates,
and (iii) apply Lemma~\ref{lem:norm_preserve_u} to obtain a \emph{constant} input matrix $B$
and a \emph{pointwise $u$-norm-preserving} lifted input $v(x,u)$.

\medskip
\noindent\emph{Step 1: Split the dynamics.}
Define
\begin{equation}
\label{eq:fu_def}
f_0(x)\triangleq f(x,0),
\qquad
f_u(x,u)\triangleq f(x,u)-f_0(x),
\end{equation}
so that $f(x,u)=f_0(x)+f_u(x,u)$.

\medskip
\noindent\emph{Step 2: Lift the autonomous dynamics and output.}
By Definition~\ref{def:def_of_j}, the autonomous dynamics close exactly:
\begin{equation}
\label{eq:autonomous_closure}
D\varphi(x)\,f_0(x)=A\varphi(x).
\end{equation}
Moreover, since $h$ lies in the span of $\varphi$, there exists $C\in\mathbb{R}^{m\times q}$ such that
\begin{equation}
\label{eq:output_in_lift}
y=C\varphi(x).
\end{equation}

\medskip
\noindent\emph{Step 3: Identify the lifted control-induced contribution.}
Define the lifted control-induced term
\begin{equation}
\label{eq:lifted_control_term_def}
g(x,u)\triangleq D\varphi(x)\,f_u(x,u)
= D\varphi(x)\big(f(x,u)-f(x,0)\big).
\end{equation}
Then, using $f=f_0+f_u$ and \eqref{eq:autonomous_closure},
\begin{equation}
\begin{aligned}
\label{eq:lifted_split_no_bars}
D\varphi(x)\,f(x,u)
&= D\varphi(x)\,f_0(x) + D\varphi(x)\,f_u(x,u) \\
&= A\varphi(x) + g(x,u).
\end{aligned}
\end{equation}
(Equivalently, along trajectories $\dot x=f(x,u)$, the chain rule gives
$\frac{d}{dt}\varphi(x(t)) = D\varphi(x)\,f(x,u)$.)

\medskip
\noindent\emph{Step 4: Verify finite induced gain in $u$ and apply Lemma~\ref{lem:norm_preserve_u}.}
By Definition~\ref{def:def_of_j}, $f$ is globally Lipschitz in $u$ uniformly over $x\in\mathcal X$, hence
$\|f_u(x,u)\|_2=\|f(x,u)-f(x,0)\|_2 \le L_u\|u\|_2$ for all $x\in\mathcal X$ and $u\in\mathbb R^p$.
Also $D\varphi$ is bounded on $\mathcal X$; let
$M_\varphi \triangleq \sup_{x\in\mathcal X}\|D\varphi(x)\|_2 < \infty$.
Therefore, for all $x\in\mathcal X$ and $u\in\mathbb R^p$,
\begin{equation}
\label{eq:g_gain_bound}
\|g(x,u)\|_2
\le \|D\varphi(x)\|_2\,\|f_u(x,u)\|_2
\le M_\varphi L_u \|u\|_2,
\end{equation}
so $\sup_{x\in\mathcal X,\,u\neq 0}\|g(x,u)\|_2/\|u\|_2 < \infty$.
Hence Lemma~\ref{lem:norm_preserve_u} applies to $(x,u)\mapsto g(x,u)$ on $\mathcal X\times\mathbb R^p$,
yielding a unitary matrix $U\in\mathbb R^{q\times q}$, a diagonal $\Sigma\in\mathbb R^{q\times(p+q)}$,
and a mapping $v:\mathcal X\times\mathbb R^p\to\mathbb R^{p+q}$ such that
\begin{equation}
\label{eq:g_factorization}
g(x,u)=U\Sigma\,v(x,u),
\qquad
\|v(x,u)\|_2=\|u\|_2,
\end{equation}
for all $x\in\mathcal X$ and $u\in\mathbb R^p$.
Let $B\triangleq U\Sigma$. Substituting \eqref{eq:g_factorization} into \eqref{eq:lifted_split_no_bars}
gives \eqref{eq:lifted_affine_like}; \eqref{eq:output_in_lift} provides the output equation.
\end{proof}

Theorem~\ref{theorem:norm_preserving} establishes two structural facts that will be used throughout the remainder of Part~1:
\begin{enumerate}
\item \textbf{Input-energy calibration (metric fidelity).}
The lifted input satisfies $\|v(x,u)\|_2=\|u\|_2$ pointwise, so gain cannot be hidden inside $v$; it must be carried by the constant input channel $B$.
This calibration is what later restores the interpretability of HSV-based $H_\infty$ truncation bounds in the \emph{physical} input metric.
\item \textbf{Affine-like actuation without control-affine assumptions.}
Even when the original system is not of the form $\dot{x}=f(x)+g(x)u$, the lifted dynamics can be written as an LTI-like system driven by an instantaneous input $v(x,u)$ through a \emph{constant} matrix $B$.
And though $v$ is a function of both $x$ and $u$, it's norm-equivalence to $u$ will be sufficient to isolate input-output gains in the next result.
\end{enumerate}

\subsubsection{A Finite-dimensional Example}

We illustrate the construction in Theorem~\ref{theorem:norm_preserving} on a two-dimensional system whose autonomous dynamics admit an exact finite-dimensional Koopman lift, and whose control contribution becomes non-affine under a simple modification.
The point of this example is to make the objects $(\varphi,A)$ and the calibrated input channel $Bv(x,u)$ fully explicit.

\noindent\emph{Reader guide.} This example follows the proof logic of Theorem~\ref{theorem:norm_preserving} verbatim. Step~1 specifies an exact Koopman lift $(\varphi,A)$ for the autonomous dynamics. Step~2 isolates the lifted control-induced term $g(x,u)$ via $g(x,u)=D\varphi(x)\big(f(x,u)-f(x,0)\big)$. Step~3 computes worst-case (uniform) coordinate gains of $g$ over a compact region $\mathcal X_R$ to obtain a single ordering and diagonal scaling that work for all admissible $(x,u)$. Step~4 then applies Lemma~\ref{lem:norm_preserve_u} with that uniform scaling to produce a constant matrix $B$ and a pointwise $u$-norm-preserving lift $v(x,u)$ satisfying $Bv(x,u)=g(x,u)$.

\paragraph{Step 1: Exact Koopman closure of the autonomous dynamics.}
We begin with the two-dimensional nonlinear system from~\cite{brunton2022review} and augment it with a scalar input channel on $x_1$:
\begin{equation}
\label{eq:brunton_controlled}
\begin{aligned}
\dot{x}_1 &= \mu x_1 + u, \\
\dot{x}_2 &= \lambda \big(x_2 - x_1^2\big).
\end{aligned}
\end{equation}
The autonomous dynamics ($u\equiv 0$) admit an exact finite-dimensional lifted linear representation with
\begin{equation}
\label{eq:phi0_def}
\varphi(x) \triangleq
\begin{bmatrix}
x_1 \\[2pt] x_2 \\[2pt] x_1^2
\end{bmatrix},
\qquad
A \triangleq
\begin{bmatrix}
\mu & 0 & 0 \\
0 & \lambda & -\lambda \\
0 & 0 & 2\mu
\end{bmatrix},
\end{equation}
so that $\dot{\varphi}_0(x)=A\varphi(x)$.

The Jacobian of the lifting is
\begin{equation}
\label{eq:Dphi0_example}
D\varphi(x)=
\begin{bmatrix}
1 & 0 \\[2pt]
0 & 1 \\[2pt]
2x_1 & 0
\end{bmatrix}.
\end{equation}

\paragraph{Step 2: Lifted control-induced contribution.}
Using the decomposition $f(x,u)=f(x,0)+f_u(x,u)$ with
$f_u(x,u)\triangleq f(x,u)-f(x,0)$, define the lifted control-induced term
\begin{equation}
\label{eq:g_def_example}
g(x,u)\triangleq D\varphi(x)\,f_u(x,u)
= D\varphi(x)\big(f(x,u)-f(x,0)\big).
\end{equation}
A direct computation shows the autonomous lifted dynamics close exactly:
\begin{equation}
\begin{aligned}
\label{eq:brunton_A_expand}
D\varphi(x)\,f(x,0)
&=
\begin{bmatrix}
1 & 0 \\
0 & 1 \\
2x_1 & 0
\end{bmatrix}
\begin{bmatrix}
\mu x_1 \\
\lambda(x_2-x_1^2)
\end{bmatrix} \\
&=
\begin{bmatrix}
\mu x_1\\
\lambda(x_2-x_1^2)\\
2\mu x_1^2
\end{bmatrix} \\
&=
A\varphi(x).
\end{aligned}
\end{equation}
Therefore,
\begin{equation}
\label{eq:lifted_split_example}
D\varphi(x)\,f(x,u)=A\varphi(x)+g(x,u).
\end{equation}

For the affine input channel in \eqref{eq:brunton_controlled}, the control-induced term is
\begin{equation}
\label{eq:delta_f_affine}
f_u(x,u)=f(x,u)-f(x,0)=
\begin{bmatrix}
u\\[2pt]
0
\end{bmatrix}.
\end{equation}
Therefore, the lifted control-induced term evaluates to
\begin{equation}
\begin{aligned}
\label{eq:g_affine}
g(x,u)
&=D\varphi(x)\,f_u(x,u)
=
\begin{bmatrix}
1 & 0\\
0 & 1\\
2x_1 & 0
\end{bmatrix}
\begin{bmatrix}
u\\
0
\end{bmatrix}
=
\begin{bmatrix}
u\\
0\\
2x_1u
\end{bmatrix}.
\end{aligned}
\end{equation}

\paragraph{Step 3: Coordinate gains and a uniform ordering on a compact set.}
For each fixed state $x$, define the \emph{coordinate-wise induced gain in the scalar input $u$} by
\begin{equation}
\label{eq:coord_gain_def_example}
\|\,(g(x,\cdot))_i\,\|_{(2\to2)_u}\;\triangleq\;
\sup_{u\neq 0}\frac{|(g(x,u))_i|}{|u|},
\qquad i\in\{1,2,3\}.
\end{equation}

For \eqref{eq:g_affine}, this gives
\begin{equation}
\begin{aligned}
\|\,(g(x,\cdot))_1\,\|_{(2\to2)_u}&=1, \\
\|\,(g(x,\cdot))_2\,\|_{(2\to2)_u}&=0, \\
\|\,(g(x,\cdot))_3\,\|_{(2\to2)_u}&=2|x_1|.
\end{aligned}
\end{equation}

To select a \emph{single} permutation and diagonal scaling that is valid uniformly over a compact region, we pass to
worst-case (uniform) gains over the ball $\mathcal X_R\triangleq\{x:\|x\|_2\le R\}$:
\begin{equation}
\label{eq:coord_gain_uniform_example}
g_i(R)\;\triangleq\;\sup_{x\in \mathcal X_R}\|\,(g(x,\cdot))_i\,\|_{(2\to2)_u}.
\end{equation}
Since $\sup_{\|x\|_2\le R}|x_1|=R$, we obtain
\begin{equation}
g_1(R)=1,\qquad g_2(R)=0,\qquad g_3(R)=2R.
\end{equation}
Hence, for $R\ge \tfrac12$, the \emph{uniform} ordering
\[
g_3(R)\;\ge\; g_1(R)\;\ge\; g_2(R)
\]
holds on $\mathcal X_R$.
This ordering is intentionally worst-case: at particular states (e.g., $x_1=0$) the instantaneous ordering of $\|\,(g(x,\cdot))_i\,\|_{(2\to2)_u}$ a single diagonal scaling in $\Sigma$ can dominate all admissible $(x,u)$ in $\mathcal X_R$.

\paragraph{Step 4: Apply Lemma~\ref{lem:norm_preserve_u} to construct $B$ and $v(x,u)$.}
Lemma~\ref{lem:norm_preserve_u} requires choosing $\Sigma$ large enough that the kernel term is real-valued uniformly
over the admissible region.
Using the uniform gains from \eqref{eq:coord_gain_uniform_example}, a simple choice is to scale the nonzero gains by a
common factor $c>\sqrt2$ (and include a small $\varepsilon>0$ for the zero-gain coordinate):
\begin{equation}
\label{eq:Sigma_example}
\Sigma=
\begin{bmatrix}
2cR & 0 & 0 & 0\\
0 & c & 0 & 0\\
0 & 0 & \varepsilon & 0
\end{bmatrix},\qquad \varepsilon>0.
\end{equation}
Choose the permutation matrix $U\in\mathbb R^{3\times 3}$ so that $U^T$ reorders coordinates according to the
uniform gain ordering $g_3(R)\ge g_1(R)\ge g_2(R)$, i.e.,
\begin{equation}
U^T
\begin{bmatrix}
(g(x,u))_1\\ (g(x,u))_2\\ (g(x,u))_3
\end{bmatrix}
=
\begin{bmatrix}
(g(x,u))_3\\ (g(x,u))_1\\ (g(x,u))_2
\end{bmatrix},
\;\text{for example}\;
U=
\begin{bmatrix}
0 & 1 & 0\\
0 & 0 & 1\\
1 & 0 & 0
\end{bmatrix}.
\end{equation}
Then define $B\triangleq U\Sigma$:
\begin{equation}
\label{eq:B_example}
B=
\begin{bmatrix}
0 & c & 0 & 0\\
0 & 0 & \varepsilon & 0\\
2cR & 0 & 0 & 0
\end{bmatrix}.
\end{equation}
Then, per the construction in Lemma~\ref{lem:norm_preserve_u}, we obtain
\begin{align}
\label{eq:v_support_example}
v_{\mathrm{support}}(x,u)
&=
\begin{bmatrix}
\frac{x_1}{cR}u\\[2pt]
\frac{1}{c}u\\[2pt]
0\\[2pt]
0
\end{bmatrix},
\\
\label{eq:v_kernel_example}
v_{\mathrm{kernel}}(x,u)
&=
\begin{bmatrix}
0\\0\\0\\u
\end{bmatrix}
\sqrt{1-\frac{\|v_{\mathrm{support}}(x,u)\|^2}{\|u\|^2}},
\\
v(x,u)&=v_{\mathrm{support}}(x,u)+v_{\mathrm{kernel}}(x,u),
\end{align}
with the convention $v_{\mathrm{kernel}}(x,0)=0$ (and hence $v(x,0)=0$),
so that $\|v(x,u)\|=\|u\|$ pointwise and $g(x,u)=U\Sigma v(x,u)$.
Indeed, $v_{\mathrm{kernel}} \in \mathrm{ker(B)}$, hence
\begin{align}
Bv(x,u)
&= Bv_{\mathrm{support}}(x,u) \nonumber\\
&= \begin{bmatrix}
c\cdot\frac{1}{c}u\\[2pt]
0\\[2pt]
2cR\cdot\frac{x_1}{cR}u
\end{bmatrix} \nonumber\\
&= \begin{bmatrix}
u\\[2pt]
0\\[2pt]
2x_1u
\end{bmatrix} \nonumber\\
&= g(x,u).
\end{align}
which matches \eqref{eq:g_affine} (i.e., $g(x,u)=[u,\,0,\,2x_1u]^\top$ in this affine case).
Thus the lifted identity \eqref{eq:lifted_split_example} takes the affine-like form
$D\varphi(x)\,f(x,u)=A\varphi(x)+Bv(x,u)$ with a pointwise $u$-norm-preserving lifted input.

\medskip
\noindent\textbf{Takeaway.}
In this system the autonomous Koopman closure fixes $(\varphi,A)$, and the control enters the lifted coordinates only through the explicitly computable term $g(x,u)$.
Once $g$ is known, Steps~3--4 follow directly from the construction: a \emph{uniform} ordering and scaling over $\mathcal X_R$ are what make the support/kernel construction real-valued for all admissible states, yielding a single constant input channel $B$.
The resulting factorization $g(x,u)=Bv(x,u)$ makes the gain calibration visible: \emph{all actuation gain is carried by $B$}, while $v(x,u)$ preserves the physical input energy pointwise, $\|v(x,u)\|_2=\|u\|_2$.

\subsubsection{Non-affine input variant (replacing $u$ by $\sin(u)$)}

Now replace the input channel in \eqref{eq:brunton_controlled} by $\sin(u)$:
\begin{equation}
\label{eq:brunton_controlled_sin}
\begin{aligned}
\dot{x}_1 &= \mu x_1 + \sin(u), \\
\dot{x}_2 &= \lambda \big(x_2 - x_1^2\big).
\end{aligned}
\end{equation}
The autonomous lifted dynamics remain unchanged, while the lifted control-induced term becomes
\begin{equation}
\label{eq:g_sin}
g(x,u)=
\begin{bmatrix}
\sin(u)\\
0\\
2x_1\sin(u)
\end{bmatrix}.
\end{equation}
The induced gain in $u$ is unchanged because $|\sin(u)|\le |u|$ for all $u$, hence
\begin{equation}
\sup_{u\neq 0}\frac{|\sin(u)|}{|u|}=1.
\end{equation}
Therefore, the same ordering, the same radius threshold $R\ge \tfrac12$, and the same choice of $c>\sqrt2$ apply.
Consequently, we may use the same matrices $U$ and $\Sigma$ as in \eqref{eq:Sigma_example}--\eqref{eq:B_example};
the only change is that $v_{\mathrm{support}}$ is evaluated on $\sin(u)$:
\begin{equation}
v_{\mathrm{support}}(x,u)=
\begin{bmatrix}
\frac{x_1}{cR}\sin(u)\\[2pt]
\frac{1}{c}\sin(u)\\[2pt]
0\\[2pt]
0
\end{bmatrix},
\end{equation}
with $v_{\mathrm{kernel}}$ updated accordingly to enforce $\|v(x,u)\|=\|u\|$.
The lifted identity therefore retains the same affine-like structure
\[
D\varphi(x)\,f(x,u)=A\varphi(x)+Bv(x,u),
\]
and in this non-affine variant the control-induced term satisfies
\[
Bv(x,u)=g(x,u)=[\sin(u),\,0,\,2x_1\sin(u)]^\top.
\]
Relative to the affine case, the saturation in $\sin(u)$ causes the support component to occupy a smaller fraction of the available input-energy budget, so the kernel component (which lies in $\ker(B)$) can become comparatively larger; equivalently, the lift $v(x,u)$ rotates further into the null space of $B$ while preserving $\|v\|=\|u\|$ pointwise.

\medskip
\noindent\textbf{Takeaway.}
Replacing $u$ by $\sin(u)$ changes the lifted control-induced term from $[u,\,0,\,2x_1u]^\top$ to $[\sin(u),\,0,\,2x_1\sin(u)]^\top$, but it does \emph{not} change the induced gain in $u$ because $|\sin(u)|\le |u|$.
Hence the same uniform ordering and the same $(U,\Sigma)$ remain valid on $\mathcal X_R$, and only the support component depends on the modified channel.
Geometrically, saturation reduces the fraction of the input-energy budget used by the support component, so the kernel component (in $\ker(B)$) can occupy a larger share while still enforcing $\|v(x,u)\|_2=\|u\|_2$.

Theorem~\ref{theorem:norm_preserving} isolates all actuation gain into a \emph{constant} input matrix $B$ and enforces pointwise input-energy calibration $\|v(x,u)\|_2=\|u\|_2$.
This motivates introducing an associated LTI system whose input is an exogenous signal of the same dimension as $v$; the original nonlinear system is recovered by feeding this LTI system with the particular signal $v(x(t),u(t))$ generated along trajectories.

\begin{definition}[Associated LTI system linear in the calibrated input]
\label{def:linear_in_v}
Let $G\in\mathcal{J}$ admit a lifted representation of the form in Theorem~\ref{theorem:norm_preserving} with matrices $(A,B,C)$ and calibrated input map $v(x,u)$.
Define the \emph{associated LTI system} $G^L$ as the LTI input--output operator with state $\varphi\in\mathbb{R}^q$ and exogenous input $w\in\mathbb{R}^{p+q}$:
\begin{equation}
\begin{aligned}
\dot{\varphi}_0 &= A\varphi + Bw,\\
y &= C\varphi.
\end{aligned}
\end{equation}
When representing the original nonlinear system, we take $w(t)=v(x(t),u(t))$.
\end{definition}

\begin{remark}[Why introduce $G^L$] \label{rem:traj_dependent_w} Definition~\ref{def:linear_in_v} introduces $G^L$ so that classical Hankel/Gramian tools apply to the calibrated input channel. For the original nonlinear system, however, the exogenous signal is not free: it is constrained by the trajectory-dependent identification \[ w(t)=v(x(t),u(t)). \] This motivates the results that follow. First, since $G^L$ is an LTI realization, it admits a standard balancing transform whenever it is minimal, enabling certified LTI reduction on the surrogate. Second, after reduction we must compare the surrogate driven by the \emph{true} calibrated signal $w(t)$ to the truncated surrogate driven by a \emph{reduced-state evaluation} of the calibrated map; this creates an input mismatch in $w$. The final bound (Theorem~\ref{theorem:final_bound}) is obtained by decomposing the output error into an input-mismatch term and the classical balanced-truncation term for $G^L$. \end{remark}

When $G^L$ is minimal, it admits a balancing transform $z=T\varphi$ (standard).
Define $\tilde{A}=TAT^{-1}$, $\tilde{B}=TB$, and $\tilde{C}=CT^{-1}$, so that $y=\tilde{C}z$.
Because $\varphi$ is state-inclusive, set $S\triangleq \begin{bmatrix}I_{n\times n}&0\end{bmatrix}$ and $R\triangleq ST^{-1}$ so that $x=Rz$ whenever $z=T\varphi(x)$.

In balanced coordinates, the lifted dynamics become
\[
\dot z=\tilde A z+\tilde B\,v(Rz,u),\qquad y=\tilde Cz.
\]
To define a reduced model, we need an \emph{implementable} calibrated input map expressed directly in balanced coordinates, so that it can be evaluated using only reduced-state information.
We therefore introduce a balanced-coordinate calibrated input map $\tilde v(z,u)$ satisfying
\[
\tilde B\,\tilde v(z,u)=\tilde B\,v(Rz,u),
\qquad
\|\tilde v(z,u)\|_2=\|u\|_2.
\]

After truncation, the reduced model evolves only the retained balanced coordinates.
When the calibrated map must be evaluated using reduced-state information, we use the canonical zero-padding interpretation (made explicit in the definition of the reduced model and the induced signals $w$ and $w_r$ in Lemma~4).
This is exactly the $w$--$w_r$ input mismatch term that appears in the reduction bound developed in Theorem~\ref{theorem:final_bound}.

The next lemma formalizes a (deliberately) mundane bookkeeping step that is nevertheless essential for the bound:
it shows how to express the calibrated input in balanced coordinates so that the injected term driving the balanced surrogate is unchanged, while the pointwise input-energy normalization is preserved.
This gives a canonical, implementable way to define the surrogate inputs used in Lemma~4 (and hence the $w$--$w_r$ decomposition), by separating the part of the calibrated signal that actually drives the balanced dynamics from the remaining degrees of freedom that do not affect the state equation.

\begin{lemma}[Balanced-coordinate calibrated input map]
\label{lemma:same_b}
Let $G\in\mathcal{J}$ and let $(A,B,C,v)$ be as in Theorem~\ref{theorem:norm_preserving}, so that the associated LTI surrogate $G^L$ (Definition~\ref{def:linear_in_v}) has realization $(A,B,C)$.
Let $z=T\varphi(x)$ be a balancing transform for $G^L$ and define
\[
\tilde A \triangleq TAT^{-1},\qquad \tilde B \triangleq TB,\qquad \tilde C \triangleq CT^{-1}.
\]
Assume the factorization $B=U\Sigma$ from Lemma~\ref{lem:norm_preserve_u} is chosen with $\Sigma=[\Sigma_0\ \ 0]$,
where $\Sigma_0\in\mathbb R^{q\times q}$ is diagonal with strictly positive entries (so that $\Sigma\Sigma^\dagger=I_q$).
Then there exists a mapping $\tilde v:\mathbb{R}^q\times\mathbb{R}^p\to\mathbb{R}^{p+q}$ satisfying
\begin{equation}
\begin{aligned}
\|\tilde v(z,u)\|_2 &=\|u\|_2,\quad \forall z\in\mathcal Z,\ \forall u\in\mathbb{R}^p,
\\
\mathcal Z &\triangleq \{\,z\in\mathbb R^q : Rz\in\mathcal X\,\},
\end{aligned}
\end{equation}
such that, under the identification $z=T\varphi(x)$, the balanced-coordinate model
\begin{equation}
\begin{aligned}
\label{eq:balanced_LTI_like}
\dot{z} &= \tilde{A}z + \tilde{B}\tilde{v}(z,u),\\
y &= \tilde{C}z,
\end{aligned}
\end{equation}
reproduces the same $(z(t),y(t))$ trajectories as the balanced realization driven by the trajectory-evaluated input,
\[
\dot z=\tilde A z+\tilde B\,v(Rz,u),\qquad y=\tilde C z,
\]
whenever $Rz(t)\in\mathcal X$.
\end{lemma}

\begin{proof}
\noindent\emph{Step 1: Write the lifted dynamics in balanced coordinates and isolate the control-induced term.}
From Theorem~\ref{theorem:norm_preserving}, the system admits the lifted representation, such that along trajectories of $\dot x=f(x,u)$,
\[
\frac{d}{dt}\varphi(x(t)) = A\varphi(x(t)) + Bv(x(t),u(t)),\;\;
y(t)=C\varphi(x(t)).
\]
Apply the balancing transform $z=T\varphi(x)$ to obtain
\begin{equation}
\begin{aligned}
\dot z = T\dot{\varphi}_0 = T(A\varphi + Bv(x,u)) &= \tilde A z + \tilde B\,v(x,u), \\ 
y &= CT^{-1}z = \tilde C z,
\end{aligned}
\end{equation}
where $\tilde A=TAT^{-1}$, $\tilde B=TB$, and $\tilde C=CT^{-1}$.

Recalling the definition $R \triangleq ST^{-1}$ from the balanced-coordinate construction, we have $x=Rz$ whenever $z=T\varphi(x)$, and therefore
\[
\dot z = \tilde A z + \tilde B\,v(Rz,u).
\]
The control-induced injected term in balanced coordinates is $\tilde B\,v(Rz,u)$.
(Note that $v(Rz,0)=0$ by pointwise norm preservation.)

\medskip
\noindent\emph{Step 2: Construct a norm-preserving $\tilde v(z,u)$ such that $\tilde B\,\tilde v(z,u)=\tilde B\,v(Rz,u)$.}

Write $B=U\Sigma$ as in Theorem~\ref{theorem:norm_preserving} (via Lemma~\ref{lem:norm_preserve_u}), so $\tilde B = TU\Sigma$.
Define the support component
\[
\tilde v_{\mathrm{support}}(z,u)\triangleq \Sigma^\dagger \Sigma\,v(Rz,u),
\]
Then, by construction,
\begin{equation}
\begin{aligned}
\tilde B\,\tilde v_{\mathrm{support}}(z,u)
&= TU\Sigma\,(\Sigma^\dagger\Sigma)\,v(Rz,u) \\
&= TU\Sigma\,v(Rz,u) \\
&= \tilde B\,v(Rz,u).
\end{aligned}
\end{equation}
Define the kernel component (as in Lemma~\ref{lem:norm_preserve_u})
\[
\tilde v_{\mathrm{kernel}}(z,u)\triangleq
\begin{bmatrix}0_q\\ u\end{bmatrix}
\sqrt{\frac{\|u\|_2^2-\|\tilde v_{\mathrm{support}}(z,u)\|_2^2}{\|u\|_2^2}},
\]
with the convention $\tilde v_{\mathrm{kernel}}(z,0)=0$, and set
\[
\tilde v(z,u)\triangleq \tilde v_{\mathrm{support}}(z,u)+\tilde v_{\mathrm{kernel}}(z,u).
\]

Because $\Sigma=[\Sigma_0\ \ 0]$, we have $\Sigma\tilde v_{\mathrm{kernel}}=0$,
Moreover, by the support/kernel computation, $\|\tilde v(z,u)\|_2=\|u\|_2$ pointwise.

Since $\Sigma=[\Sigma_0\ \ 0]$, the matrix $\Sigma^\dagger\Sigma$ is an orthogonal projector, so
$\|\tilde v_{\mathrm{support}}(z,u)\|_2 \le \|v(Rz,u)\|_2=\|u\|_2$ and the radicand is nonnegative.

\medskip
\noindent\emph{Conclusion.}
Using $\tilde B\,\tilde v(z,u)=\tilde B\,v(Rz,u)$ in the balanced-coordinate dynamics yields
\[
\dot z = \tilde A z + \tilde B\tilde v(z,u),\qquad y=\tilde C z,
\]
which is \eqref{eq:balanced_LTI_like}.
\end{proof}

\begin{remark}[Interpretation of Lemma~\ref{lemma:same_b}]
\label{rem:gauge_tildev}
In balanced coordinates, the state equation depends on the calibrated input only through the injected term $\tilde B\,w$.
Accordingly, any two signals $w,w'\in\mathbb R^{p+q}$ satisfying $\tilde B w=\tilde B w'$ are dynamically indistinguishable, since their difference lies in $\ker(\tilde B)$.
Lemma~\ref{lemma:same_b} fixes a convenient representative of this equivalence class: the support component
\[
\tilde v_{\mathrm{support}}(z,u)=\Sigma^\dagger\Sigma\,v(Rz,u)
\]
is the orthogonal projection of $v(Rz,u)$ onto the $\tilde B$-visible subspace and satisfies $\tilde B\,\tilde v_{\mathrm{support}}(z,u)=\tilde B\,v(Rz,u)$.
The remaining degrees of freedom in $\ker(\tilde B)$ are then used (via $\tilde v_{\mathrm{kernel}}$) to enforce the pointwise calibration $\|\tilde v(z,u)\|_2=\|u\|_2$ without altering the injected term.
This bookkeeping is what makes the surrogate inputs $w$ and $w_r$ in Lemma~4 well-defined and enables the $w-w_r$ mismatch decomposition used in the reduction bound.
Moreover, since $\Sigma^\dagger\Sigma$ is an orthogonal projector and $\|v(Rz,u)\|_2=\|u\|_2$, we have
\begin{equation}
\label{eq:support_bound}
\|\tilde v_{\mathrm{support}}(z,u)\|_2 \le \|u\|_2,
\qquad \forall (z,u)\ \text{such that}\ Rz\in\mathcal X.
\end{equation}
\end{remark}

\begin{remark}
For a system $G\in\mathcal{J}$, the model \eqref{eq:balanced_LTI_like} is a similarity transform of $G^L$, so we will use $G^L$ to refer to either the original or balanced coordinates when the meaning is clear from context.
\end{remark}

We now combine the calibration-based input-mismatch estimate with the classical balanced truncation certificate for the associated LTI surrogate $G^L$.
Lemma~4 isolates the nonlinear difficulty into a mismatch term proportional to $\|G^L\|_{H_\infty}$ and reduces the remainder to the purely LTI quantity $\|G^L-G_r^L\|_{H_\infty}$; Theorem~2 then closes the argument by upper bounding this latter term by the Hankel singular value tail of $G^L$.

Before stating the next bound, we fix the induced-norm convention and the admissible signal/trajectory class so that the operators below are unambiguously defined.
Throughout this section, all induced norms (e.g., $\|G\|_{H_\infty(\mathcal U)}$ and $\|G^L\|_{H_\infty}$) are interpreted under a fixed equilibrium initial condition, i.e., $x(0)=0$ and hence $\varphi(0)=0$ (equivalently $z(0)=0$ in balanced coordinates).
Under this convention, $G:u(\cdot)\mapsto y(\cdot)$ and $G^L:w(\cdot)\mapsto y(\cdot)$ are well-defined causal operators on $L^2$ over the admissible trajectory class (in particular, trajectories remaining in $\mathcal X$).
Since $A$ is Hurwitz and the realization has no direct term, $G^L$ is stable and strictly proper, and $\|G^L\|_{H_\infty}<\infty$.

To make the truncation construction fully explicit, we now introduce the canonical projection/embedding pair between the full balanced state space $\mathbb R^q$ and its retained $r$-dimensional subspace.
Fix an order $r<q$. Let $\Pi_r:\mathbb R^{q}\to\mathbb R^{r}$ denote the coordinate projection
\[
\Pi_r z \triangleq \begin{bmatrix} I_r & 0_{r\times(q-r)}\end{bmatrix} z,
\]
and let $E:\mathbb R^{r}\to\mathbb R^{q}$ denote the canonical zero-padding embedding (a right-inverse of $\Pi_r$),
\begin{equation}\label{eq:zero_padding_map}
E z_r \triangleq \begin{bmatrix} z_r \\ 0_{q-r}\end{bmatrix},
\qquad
E \;=\;\begin{bmatrix} I_r \\ 0_{(q-r)\times r}\end{bmatrix}.
\end{equation}
Then $\Pi_r E = I_r$ and $E\Pi_r$ is the projection onto the first $r$ balanced coordinates.

\begin{lemma}[Calibration-based bound over an admissible input class]
\label{lem:closed_loop_BT}

Let $G\in\mathcal J$ admit the balanced-coordinate lifted representation \eqref{eq:balanced_LTI_like} with associated stable LTI surrogate $G^L$.
Assume $G^L$ is realized in minimal form, and let $G_r^L$ denote the order-$r$ balanced truncation of this minimal realization with realization $(\tilde A_r,\tilde B_r,\tilde C_r)$ and state $z_r\in\mathbb R^r$.
Define the implementable reduced nonlinear model $G_r$ by
\begin{equation}
\label{eq:Gr_nonfeedback}
\dot z_r = \tilde A_r z_r + \tilde B_r\,\tilde v_{\mathrm{support}}(E z_r,u),\qquad y_r = \tilde C_r z_r.
\end{equation}
Fix an admissible input class $\mathcal U\subset L^2$ such that for every $u(\cdot)\in\mathcal U$,
the trajectories satisfy $z(t)\in\mathcal Z$ and $Ez_r(t)\in\mathcal Z$ for all $t\ge 0$,
where $\mathcal Z\triangleq\{z\in\mathbb R^q:\;Rz\in\mathcal X\}$.
Then the induced input--output error satisfies
\begin{equation}
\label{eq:closed_loop_bound}
\|G-G_r\|_{H_\infty(\mathcal U)}
\le 2\|G^L\|_{H_\infty} + \|G^L-G_r^L\|_{H_\infty}.
\end{equation}
\end{lemma}

\begin{proof}
Fix $u(\cdot)\in\mathcal U$ and let $z(\cdot)$ and $z_r(\cdot)$ denote the trajectories of $G$ and $G_r$ driven by $u(\cdot)$.
Define the calibrated support signals
\[
w(t)\triangleq \tilde v_{\mathrm{support}}(z(t),u(t)),\quad
w_r(t)\triangleq \tilde v_{\mathrm{support}}(E z_r(t),u(t)).
\]
By Lemma~\ref{lemma:same_b}, the support/kernel split yields $\tilde v=\tilde v_{\mathrm{support}}+\tilde v_{\mathrm{kernel}}$ with
$\tilde B,\tilde v_{\mathrm{kernel}}(z,u)=0$ for all $(z,u)$; hence the injected term depends only on the support component.
Fix $u(\cdot)\in\mathcal U$, let $z(\cdot)$ and $z_r(\cdot)$ be the resulting trajectories of $G$ and $G_r$, and define the (endogenous) support signals
\[
w(t)\triangleq \tilde v_{\mathrm{support}}(z(t),u(t)),\quad
w_r(t)\triangleq \tilde v_{\mathrm{support}}(E z_r(t),u(t)).
\]
(Note that there is no circularity: $z(\cdot)$ and $z_r(\cdot)$ are defined by the original dynamics driven by $u(\cdot)$; $w(\cdot)$ and $w_r(\cdot)$ are then \emph{derived} signals. We may subsequently view them as inputs to the LTI surrogates.)

Consequently, the corresponding outputs satisfy
\[
y = G^L(w),\qquad y_r = G_r^L(w_r).
\]
Using linearity of $G^L$ and adding/subtracting $G^L(w_r)$ gives
\begin{equation}
\begin{aligned}
y-y_r
&= G^L(w)-G_r^L(w_r) \\
&= \big(G^L(w)-G^L(w_r)\big) + \big(G^L(w_r)-G_r^L(w_r)\big) \\
&= G^L(w-w_r) + (G^L-G_r^L)(w_r).
\end{aligned}
\end{equation}
Taking $L^2$ norms and using induced gains yields
\[
\|y-y_r\|_{L^2}
\le \|G^L\|_{H_\infty}\,\|w-w_r\|_{L^2}
   + \|G^L-G_r^L\|_{H_\infty}\,\|w_r\|_{L^2}.
\]
By the calibration/support bound (Remark~\ref{rem:gauge_tildev}), we have $\|w(t)\|_2\le \|u(t)\|_2$ and $\|w_r(t)\|_2\le \|u(t)\|_2$ for all $t$, hence
\[
\|w-w_r\|_{L^2}\le \|w\|_{L^2}+\|w_r\|_{L^2}\le 2\|u\|_{L^2},
\;\;
\|w_r\|_{L^2}\le \|u\|_{L^2}.
\]
Substituting gives
\[
\|y-y_r\|_{L^2}
\le \Big(2\|G^L\|_{H_\infty}+\|G^L-G_r^L\|_{H_\infty}\Big)\,\|u\|_{L^2}.
\]
Dividing by $\|u\|_{L^2}$ and taking the supremum over $u\in\mathcal U\setminus\{0\}$ yields \eqref{eq:closed_loop_bound}.
\end{proof}

\begin{theorem}[Non-feedback implementable error bounds for reduced nonlinear control systems]
\label{theorem:final_bound}
Assume $G\in\mathcal J$ and let $G^L$ be the associated stable LTI surrogate, realized in minimal form, with Hankel singular values $\{\nu_i\}_{i=1}^q$.
Let $G_r$ be the implementable reduced model defined in Lemma~\ref{lem:closed_loop_BT}.
Then
\begin{equation}
\label{eq:bounds_impl}
\|G-G_r\|_{H_\infty(\mathcal U)}
\le
2\|G^L\|_{H_\infty}
+ 2\sum_{i=r+1}^{q}\nu_i .
\end{equation}
\end{theorem}

\begin{proof}
By Lemma~\ref{lem:closed_loop_BT},
\[
\|G-G_r\|_{H_\infty(\mathcal U)}
\le 2\|G^L\|_{H_\infty} + \|G^L-G_r^L\|_{H_\infty}.
\]
Classical balanced truncation for $G^L$ yields
\[
\|G^L-G_r^L\|_{H_\infty} \le 2\sum_{i=r+1}^{q}\nu_i.
\]
Combining the two inequalities gives \eqref{eq:bounds_impl}.
\end{proof}

\begin{remark}
The term $2\|G^L\|_{H_\infty}$ is independent of $r$: without additional regularity linking the reduced state to the calibrated input map, the mismatch signal $w-w_r$ is only bounded in norm by $2\|u\|_{L^2}$ and need not shrink as $r$ increases.
For LTI systems the calibrated input map is state-independent, so $w\equiv w_r$ and this mismatch term vanishes, recovering the classical balanced truncation estimate.
\end{remark}

\begin{remark}
The tightness of the bound depends on the gain allocation between the static matrix $\Sigma$ (hence $B$) and the calibrated lift $v$.
If $\Sigma$ is scaled conservatively to guarantee real-valuedness of the kernel term uniformly, then $\|G^L\|_{H_\infty}$ and the HSV tail $\sum_{i=r+1}^q \nu_i$ may become more conservative.
\end{remark}

\begin{remark}
An important assumption in Theorem~\ref{theorem:final_bound} is the existence of an exact finite-dimensional Koopman closure for the autonomous dynamics with a lifting whose coordinate functions have finite $2$-induced norm.
The next section accounts for deviations from this closure assumption.
\end{remark}

%% file: subsections/appendix_summary.tex
\subsection{Main Results Part 2: Extension to Systems with Approximate Koopman Representations}
\label{sec:approx_koopman}

Part~1 assumed exact finite-dimensional Koopman closure of the lifted autonomous dynamics. When closure holds only approximately, the lifted dynamics acquire an additional residual channel. This residual can be isolated as a feedback uncertainty acting on an otherwise LTI-like lifted system, without destroying the input-energy calibration established in Part~1. The definitions and intermediate constructions needed to interpret the resulting bound (approximate closure, feedback sister systems, and the small-gain gap estimate) are collected in Appendix~A; the main consequences are summarized below.

The resulting analysis is organized around a symmetric pair of robustness operations. 
First, the closure-error channel is isolated as a memoryless feedback block $G^E$ acting on the lifted state and is \emph{peeled off} at full order, producing a nominal plant $G^P$ whose input--output behavior differs from the original system by a small-gain gap quantified by $\xi(G^P,G^E)$.
This nominal plant falls directly within the scope of Part~1, so balanced truncation yields a certified reduction $G^P\to G_r^P$ with a Hankel singular value bound.

Second, the same closure-error feedback is \emph{reintroduced} after truncation. This mirrors the first step at reduced order and incurs an additional small-gain gap $\xi(G_r^P,G_r^E)$. Thus, the robustness penalty associated with approximate Koopman closure appears twice: once before truncation and once after, reflecting the symmetry between removing and restoring the feedback channel.

Because the feedback representation is stated most cleanly in identity-output form, the comparison between the physical output map and the identity output is handled separately using a resolvent-type estimate of the form $\|C_1-C_2\|_{2\to2}\,\|\Phi\|_{H_\infty}$ (Observation~\ref{observation:different_c}). Together, these steps yield the final bound of Theorem~\ref{theorem:bound_with_error}, in which the total error decomposes into paired robustness terms, a certified truncation term, and two output-map mismatch terms.

\begin{theorem}[Error bound for reduced control systems with approximate Koopman closure]
\label{theorem:bound_with_error}
Let $G\in\mathcal{J}^+$ and fix an admissible input class $\mathcal{U}\subset L^2$ (all induced gains evaluated at the equilibrium initial condition).
Let $G^I$, $G^P$, and $G^E$ denote the identity-output and feedback sister systems associated with $G$ (Definition~\ref{def:embedded_systems}),
and let $G_r^P$, $G_r^E$, and $G_r^I$ denote their order-$r$ reduced counterparts obtained by truncating the nominal LTI surrogate and re-forming the same feedback interconnection at reduced order.

Assume the small-gain conditions hold for the full- and reduced-order feedback interconnections:
\[
\|G^P G^E\|_{H_\infty} < 1,
\qquad
\|G_r^P G_r^E\|_{H_\infty} < 1.
\]
Assume further that for every $u(\cdot)\in\mathcal{U}$, all trajectories compared below exist for all $t\ge 0$ and remain in the prescribed compact set $\mathcal{X}$.
Equivalently, in balanced lifted coordinates $z\in\mathbb{R}^q$ and reduced coordinates $z_r\in\mathbb{R}^r$, we have
\[
R z(t)\in\mathcal{X},
\qquad
R E z_r(t)\in\mathcal{X}
\quad\forall t\ge 0,
\]
where $R$ is the state readout map in balanced coordinates and
\[
E:\mathbb{R}^r\to\mathbb{R}^q,\qquad
E z_r \triangleq \begin{bmatrix} z_r \\ 0_{q-r}\end{bmatrix}
\]
is the canonical zero-padding embedding.

Let $(G^P)^L$ denote the associated stable LTI surrogate obtained by treating the calibrated actuation signal as an exogenous input $w$:
\[
\dot z = \tilde A z + \tilde B w,\qquad y = I z,
\]
and let $\{\nu_i\}_{i=1}^q$ be the Hankel singular values of $(G^P)^L$.
Let $\Phi$ and $\Phi_r$ denote the resolvent input maps associated with the full- and order-$r$ truncated surrogates (Definition~\ref{def:def_of_phi}).

Then the induced input--output error between $G$ and its order-$r$ reduced model $G_r$ satisfies
\begin{equation}
\label{eq:bound_with_error_clean_summary}
\begin{aligned}
\|G - G_r\|_{H_\infty(\mathcal{U})}
\;&\le\;
\|\tilde C_0 - I\|_{2\to2}\,\|\Phi\|_{H_\infty}
\\
&+\;\xi(G^P,G^E)
\\
&+\; 2\Big(\|(G^P)^L\|_{H_\infty} \;+\; \sum_{i=r+1}^{q}\nu_i\Big)
\\
&+\;\xi(G_r^P,G_r^E)
\\
&+\;\|\tilde C_{r,0} - I_{r\times r}\|_{2\to2}\,\|\Phi_r\|_{H_\infty},
\end{aligned}
\end{equation}
where $\xi(\cdot,\cdot)$ is the small-gain gap bound (Definition~\ref{def:def_of_xi}), and $\tilde C_0$ and $\tilde C_{r,0}$ are the zero-padded embeddings of the full and reduced output maps (Definition~\ref{def:embedded_c}).
\end{theorem}
The bound in Theorem~\ref{theorem:bound_with_error} separates into three qualitatively different contributions.
The output-map mismatch terms, which arise from temporarily replacing the physical output map by the identity to expose the feedback structure, are independent of the Koopman closure accuracy.
In contrast, the robustness penalties $\xi(G^P,G^E)$ and $\xi(G_r^P,G_r^E)$ quantify the sensitivity of the reduction to closure error and grow unbounded as the corresponding small-gain margins approach zero.

The small-gain condition used to control the closure residual is sufficient but not necessary; accordingly, divergence of the bound does not imply divergence of the actual input--output error.
Finally, the construction relies only on state-inclusivity of the lifting: exact Koopman structure is not required, and even identity or Jacobian-based coordinates fit within the same framework when the residual satisfies the small-gain condition.

%% file: subsections/numeric_results.tex
\section{Example and Experimental Results}

\subsection{Algebraic Example, Overview}

In this section we instantiate the reduction-and-certification pipeline implied by Theorem~2 on a concrete nonlinear system with state-dependent, non-affine actuation (Appendix~B). The goal is twofold. First, we isolate the metric-mismatch failure mode emphasized in the Introduction: if the lifted input $v(x,u)$ is not calibrated to the physical control energy, then Hankel/BT quantities computed on the lifted surrogate can be valid only with respect to a different lifted-input norm and may cease to certify the original input--output operator induced by $u$. Second, we show that when the pointwise calibration constraint $\|v(x,u)\|_2=\|u\|_2$ is enforced, the resulting lifted surrogate admits a constant input channel and the associated Hankel singular values regain their standard interpretation as energy-coupled input--output modes in the physical input metric.

We evaluate these effects numerically on a $25$-state nonlinear network model (a five-neuron Hodgkin--Huxley network with saturating optogenetic control; details in Appendix~B). Figures~\ref{fig:hh_certification_spectral}--\ref{fig:hh_rollout_time_domain} assess certificate validity by comparing predicted Hankel/BT bounds to measured rollout error under identical bounded inputs, with and without input-energy calibration. The remainder of the section then derives the explicit one-mode reduced dynamics obtained after balancing, showing how the reduced nonlinear vector field can be written as a weighted superposition of the original coordinate functionals $\{f_i\}$ with the reduced coordinate substituted back into each $f_i$.

\subsubsection{Numerical protocol}
We evaluate the proposed certification pipeline on the discrete-time Hodgkin--Huxley network described in Appendix~B. 
From simulated trajectories $\{(x_k,u_k)\}$ we fit a lifted LTI surrogate of the form
\[
z_{k+1}=Az_k + B\,v(x_k,u_k), \qquad y_k=Cz_k,
\]
and compute (i) the spectrum of $A$, (ii) the singular values of $B$, and (iii) the Hankel singular values (HSVs) of the associated LTI system $G^L$.
We then perform balanced truncation at multiple reduced orders $r$ and compare the predicted reduction bound
\(
\|G^L-G^L_r\|_{\mathcal H_\infty}\le 2\sum_{i=r+1}^q \nu_i
\)
to the empirically measured rollout error of the reduced nonlinear surrogate.
To isolate the role of input-energy calibration, we train two variants: (a) a norm-preserving lifting constrained to satisfy $\|v(x_k,u_k)\|_2=\|u_k\|_2$ pointwise, and (b) an unconstrained lifting.
Figures~\ref{fig:hh_certification_spectral}--\ref{fig:hh_rollout_time_domain} summarize the resulting differences in identifiability, reduction spectra, and certificate validity.

\paragraph{Discrete-time implementation}
All experiments are conducted from sampled trajectories and are intended for digital control, so we work with the discrete-time input--output operator induced by the sampled system. Accordingly, the lifted surrogate is identified in discrete time, and all induced-gain quantities (including $\mathcal H_\infty$ norms and Hankel singular values) are computed for the resulting discrete-time LTI surrogate. This is the discrete-time analogue of the operator-norm framework used in Theorem~2, with $L_2$ replaced by $\ell_2$.

We implement the learned lifted surrogate in the form
\[
z_{k+1}=Az_k+B\,v(x_k,u_k), \qquad y_k=Cz_k,
\]
where discrete-time stability corresponds to $\rho(A)<1$ (eigenvalues inside the unit disk).
All induced-gain quantities are computed in the discrete-time $\mathcal{H}_\infty$ sense,
\[
\|G\|_{\mathcal{H}_\infty}=\sup_{\omega\in[0,2\pi)} \sigma_{\max}\!\bigl(G(e^{j\omega})\bigr),
\]
and controllability/observability Gramians are obtained from the Stein equations
\begin{align}
P &= A P A^\top + B B^\top, \label{eq:stein_P}\\
Q &= A^\top Q A + C^\top C, \label{eq:stein_Q}\\
\intertext{equivalently,}
0 &= A P A^\top - P + B B^\top, \label{eq:stein_P_residual}\\
0 &= A^\top Q A - Q + C^\top C. \label{eq:stein_Q_residual}
\end{align}
The pointwise calibration constraint is enforced as $\|v(x_k,u_k)\|_2=\|u_k\|_2$, so that $\|v\|_{\ell_2}=\|u\|_{\ell_2}$ along trajectories.
In implementation, this is enforced by normalizing the raw lifted-input output $\hat v(x_k,u_k)$ to unit Euclidean norm and then scaling by the input magnitude:
\[
v(x_k,u_k) \;=\; \|u_k\|_2 \,\frac{\hat v(x_k,u_k)}{\|\hat v(x_k,u_k)\|_2+\varepsilon},
\]
with a small $\varepsilon>0$ to avoid division by zero when $\hat v$ is (near) zero.

The calibration constraint places the lifted surrogate in the correct input-energy metric, so that Hankel/BT certificates are interpreted with respect to the original input $u$. The balanced truncation tail bound governs only the truncation error of the stable LTI surrogate,
\[
\|G^L-G_r^L\|_{H_\infty} \le 2\sum_{i=r+1}^{q}\nu_i,
\]
while the overall nonlinear reduction error contains an additional state-dependent term arising from recomputation of the calibrated input channel. Theorem~\ref{theorem:final_bound} combines these contributions into a single induced-gain bound for the nonlinear operator.

\subsubsection*{Summary of numerical findings}
Figure~\ref{fig:hh_certification_spectral} shows that constraining $\|v\|_2=\|u\|_2$ resolves a scale/gauge ambiguity between $(B,v)$: the norm-preserving model concentrates gain in $B$ and yields HSV spectra whose truncation bounds upper-bound the measured error, whereas the unconstrained model can hide effective gain inside $v(x,u)$ and thereby produce non-certifying bounds despite stable $A$.
Figure~\ref{fig:hh_rollout_time_domain} visualizes the same effect in time domain: under identical bounded inputs, the unconstrained lifting yields catastrophic divergence while the norm-preserving lifting produces stable rollouts and meaningful low-rank reductions.

We now shift focus from certificate validity to structure. Having established that input-energy calibration is necessary for Hankel/BT quantities to meaningfully certify reduction error in the physical input metric, we examine the form of the reduced nonlinear dynamics themselves. In particular, we show that after balancing and truncation to a single dominant mode, the reduced vector field admits an explicit algebraic representation as a weighted superposition of the original coordinate functionals, with the reduced coordinate substituted back into each nonlinearity. This derivation clarifies how the balanced mode selects and combines the underlying physical mechanisms of the full system.

\subsection{Setup and notation}
Let $x\in\mathbb{R}^n$ denote the full state and let $\varphi_0:\mathbb{R}^n\to\mathbb{R}^q$ be a state-inclusive lifting. For the algebra below, it is convenient to write the dynamics componentwise as,
\[
\dot{x}_i = f_i([I_n 0]\varphi_0(x),u), \qquad i=1,\dots,n,
\]
so that \(f=[f_1,\dots,f_n]^\top\) is the original right-hand side. Let \(z = T \varphi_0(x)\) be the balanced coordinates and \(x=[I_n 0] T^{-1}z\). 
When we retain only the first $r$ balanced modes, we consider $T^{-1}_{:w}$, the first $w$ columns of $T^{-1}$, such that the projection for the kept coordinates is
\[
R = [I\;0]\,T^{-1}_{:w} \;\in\; \mathbb{R}^{n\times 1}.
\]
In the special case of $w=1$, we define
\[
r \;=\; \begin{bmatrix} r_1 \\ r_2 \\ \vdots \\ r_n \end{bmatrix} \in \mathbb{R}^n,
\quad\text{so that}\quad
R=r,\qquad x \approx R z_1 = r\,z_1.
\]
Thus \(r_i\) is the weight attached to the \(i\)-th balanced coordinate and equation in the dominant balanced mode.

The one-mode dynamics for the first balanced coordinate follow from the chain rule. Since
\[
z = T\varphi_0(x),\qquad z_1 = e_1^\top z = e_1^\top T\varphi_0(x),
\]
Along trajectories $\dot x = f(x,u)$,
\[
\dot z_1
= \frac{d}{dt}\big(e_1^\top T\varphi_0(x)\big)
= e_1^\top T\,D\varphi_0(x)\,f(x,u).
\]
Under the one-mode embedding $x \approx r z_1$, this yields the reduced scalar dynamics
\[
\dot z_1 \approx e_1^\top T\,D\varphi_0(r z_1)\,f(r z_1,u).
\]

\subsubsection*{Step 1: Linear combination of coordinate functionals}
Using the Moore–Penrose pseudoinverse for a full-column vector \(R=r\),
\[
R^\dagger = \frac{r^\top}{r^\top r},
\]
the reduced right-hand side is
\begin{equation}
\label{eq:one_mode_rhs}
\dot z_1
= R^\dagger f(R z_1,u)
= \frac{r^\top f(r z_1,u)}{r^\top r}
= \frac{\sum_{i=1}^n r_i\, f_i(r z_1,u)}{\sum_{i=1}^n r_i^2}.
\end{equation}
It is helpful to first emphasize just the \emph{additive structure}, before worrying about the substitution \(x=r z_1\). Writing the full system schematically as
\begin{equation}
\begin{aligned}
\dot{x}_1 &= f_1(x,u),\\
\dot{x}_2 &= f_2(x,u),\\
&\;\vdots\\
\dot{x}_n &= f_n(x,u),
\end{aligned}
\qquad\Longrightarrow\qquad
\begin{aligned}
\dot{z}_1
=\; &r_1\, f_1(\,\cdot\,)\, + \\ 
&r_2\, f_2(\,\cdot\,)\, + \\ 
&\;\;\; \vdots \\
&r_n\, f_n(\,\cdot\,)
\end{aligned}
\end{equation}
Equation~\eqref{eq:one_mode_rhs} expresses $\dot z_1$ as a linear combination of the coordinate functionals $f_i$, weighted by the dominant balanced mode $r$ (the first column of $T^{-1}$).

\subsubsection*{Step 2: Substitution of the reduced coordinate}
Next we make explicit how \(z_1\) enters those functionals:
\begin{equation*}
\begin{aligned}
x = r\, z_1 \Longrightarrow
f_i(x,u) &\;\mapsto\; f_i(r z_1,u) \\
&\;=\; f_i\!\big([\,r_1 z_1,\; r_2 z_1,\;\dots,\; r_n z_1\,]^\top,\,u\big).
\end{aligned}
\end{equation*}
This shows that every occurrence of an original variable \(x_j\) (e.g., a membrane voltage, a gate, or a synaptic state) is replaced by \(r_j z_1\). \emph{Thus \(r_j\) measures how much the retained mode \(z_1\) “behaves like” the original coordinate \(x_j\)} inside each nonlinearity.

Combining Steps 1 and 2 yields the explicit one-mode reduced dynamics:
\begin{equation}
\begin{aligned}
\dot{z}_1
= \frac{\sum_{i=1}^n r_i\, f_i(r z_1,u)}{\sum_{i=1}^n r_i^2}.
\end{aligned}
\end{equation}

\paragraph{What the \(f_i\) look like in HH (five representatives)}
We record five representative coordinate functionals from the Hodgkin--Huxley network (Appendix~B): the membrane-voltage equation for one neuron, the corresponding gating-logit equations, and an excitatory synaptic-gate equation:
\begin{align*}
\text{(Voltage)}\quad
&f_{\hat V_1}(x,u) \;=\; -\alpha\big(I_{\mathrm{Na},1}+I_{\mathrm{K},1}+I_{\mathrm{L},1} \\ &\quad\quad\quad\quad\quad\quad\quad\;\;\; +I_{\mathrm{syn},1}+I_{\mathrm{ChR},1}\big),\\
\text{(Gate $m$)}\quad
&f_{z_{m,1}}(x,u) \;=\; \frac{a_m(V_1)}{m_1} - \frac{b_m(V_1)}{1-m_1},\\
\text{(Gate $h$)}\quad
&f_{z_{h,1}}(x,u) \;=\; \frac{a_h(V_1)}{h_1} - \frac{b_h(V_1)}{1-h_1},\\
\text{(Gate $n$)}\quad
&f_{z_{n,1}}(x,u) \;=\; \frac{a_n(V_1)}{n_1} - \frac{b_n(V_1)}{1-n_1},\\
\text{(Excit. synapse)}\quad
&f_{s_{E,1}}(x,u) \;=\; \frac{s_\infty(V_1;V_{\theta,E},k_E)-s_{E,1}}{\tau_E}.
\end{align*}

\paragraph{Putting it together}
Restricting \eqref{eq:one_mode_rhs} to the five representative coordinates listed above yields
\begin{equation}
\begin{aligned}
\dot z_1 &\;=\; r_1 f_{\hat V_1}(x,u) \\
&+ r_2 f_{z_{m,1}}(x,u) \\ 
&+ r_3 f_{z_{h,1}}(x,u) \\
&+ r_4 f_{z_{n,1}}(x,u)\\
&+ r_5 f_{s_{E,1}}(x,u) \\
&+ \sum_{i\in\mathcal I_{\mathrm{rest}}} r_i f_i(x,u),
\end{aligned}
\end{equation}
where \(\mathcal I_{\mathrm{rest}}\) indexes the remaining coordinates in the full network.

Under the one-mode embedding \(x \approx r z_1\), each occurrence of a state component in the coordinate functionals is evaluated at its reduced form. For example,
\begin{equation}
\begin{aligned}
V_1 &= E_\mathrm{L} + V_\mathrm{scale}\big(r_{\hat V_1}\, z_1 + \hat V_1^\star\big), \\
m_1 &= \sigma\!\big(r_{z_{m,1}}\, z_1 + z^\star_{m,1}\big), \\
h_1 &= \sigma\!\big(r_{z_{h,1}}\, z_1 + z^\star_{h,1}\big), \\
n_1 &= \sigma\!\big(r_{z_{n,1}}\, z_1 + z^\star_{n,1}\big),
\end{aligned}
\end{equation}
and similarly for all other coordinates. Substituting into the representation above gives
\begin{equation}
\begin{aligned}
\dot z_1 &\;=\; r_1 f_{\hat V_1}(r z_1,u) \\
&+ r_2 f_{z_{m,1}}(r z_1,u) \\
&+ r_3 f_{z_{h,1}}(r z_1,u) \\
&+ r_4 f_{z_{n,1}}(r z_1,u) \\
&+ r_5 f_{s_{E,1}}(r z_1,u) \\
&+ \sum_{i\in\mathcal I_{\mathrm{rest}}} r_i f_i(r z_1,u).
\end{aligned}
\label{eq:final_sum}
\end{equation}

Equations~\eqref{eq:one_mode_rhs}--\eqref{eq:final_sum} show that the one-mode reduction is obtained by evaluating the original coordinate functionals $f_i$ on the rank-one embedding $x \approx r z_1$ and aggregating their contributions with coefficients determined by the dominant balanced mode $r$. The weights $r_i$ quantify each coordinate's contribution to $\dot z_1$, while the substitution $x_j \mapsto r_j z_1$ specifies how the retained coordinate enters the nonlinearities.

\subsubsection*{Connection to certification}

The one-mode derivation above characterizes the reduced dynamics in mechanistic terms: $\dot z_1$ is given by a weighted superposition of the original coordinate functionals evaluated under the embedding $x \mapsto r z_1$.

The certification results, however, concern the induced input--output error operator associated with this reduction.

\begin{figure*}[t]
    \centering
    \includegraphics[
    width=\textwidth,
    height=0.85\textheight,
    keepaspectratio]{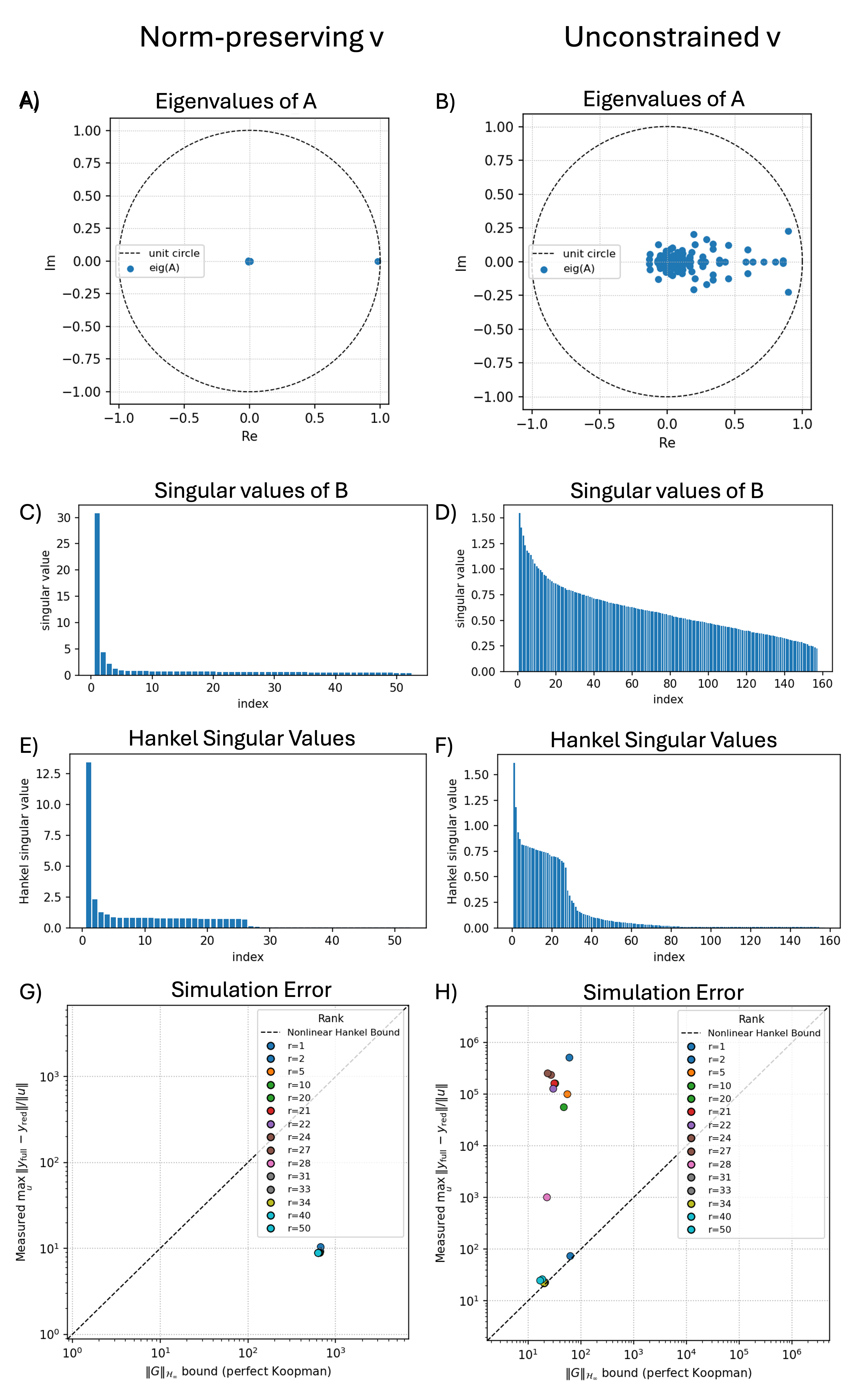}
    \caption{
    \textbf{Norm preservation calibrates the lifted input channel, restoring certified error bounds.}
    Left column: models trained with the pointwise constraint $\|v(x_k,u_k)\|_2=\|u_k\|_2$; right column: models trained with unconstrained $v$.
    (A--B) Eigenvalues of the learned discrete-time matrix $A$ (dashed unit circle) indicate stability in both settings, showing that stability of $A$ alone does not guarantee a meaningful certified reduction.
    (C--D) Singular values of the learned input matrix $B$ reveal a pronounced scale/identifiability difference: when $v$ is norm-preserving, the required actuation gain must be carried by $B$, whereas unconstrained training can hide effective gain inside $v(x,u)$.
    (E--F) Hankel singular values (HSVs) of the lifted LTI surrogate suggest apparent reducibility in both cases, but only the norm-preserving construction ensures that these HSVs correspond to the \emph{physical} input metric.
    (G--H) Certification check: measured simulation error (vertical axis) versus the predicted $\mathcal{H}_\infty$ reduction bound computed from the lifted surrogate (horizontal axis), across multiple reduced orders $r$.
    With norm-preserving $v$ (G), the empirical error remains below the bound (conservative but valid); with unconstrained $v$ (H), the bound can underestimate error by orders of magnitude, demonstrating that without input norm calibration the computed $\mathcal{H}_\infty$ bound is not a certificate with respect to the physical input $\|u\|$.
    }
    \label{fig:hh_certification_spectral}
\end{figure*}

\begin{figure*}[t]
    \centering
    \includegraphics[width=\textwidth]{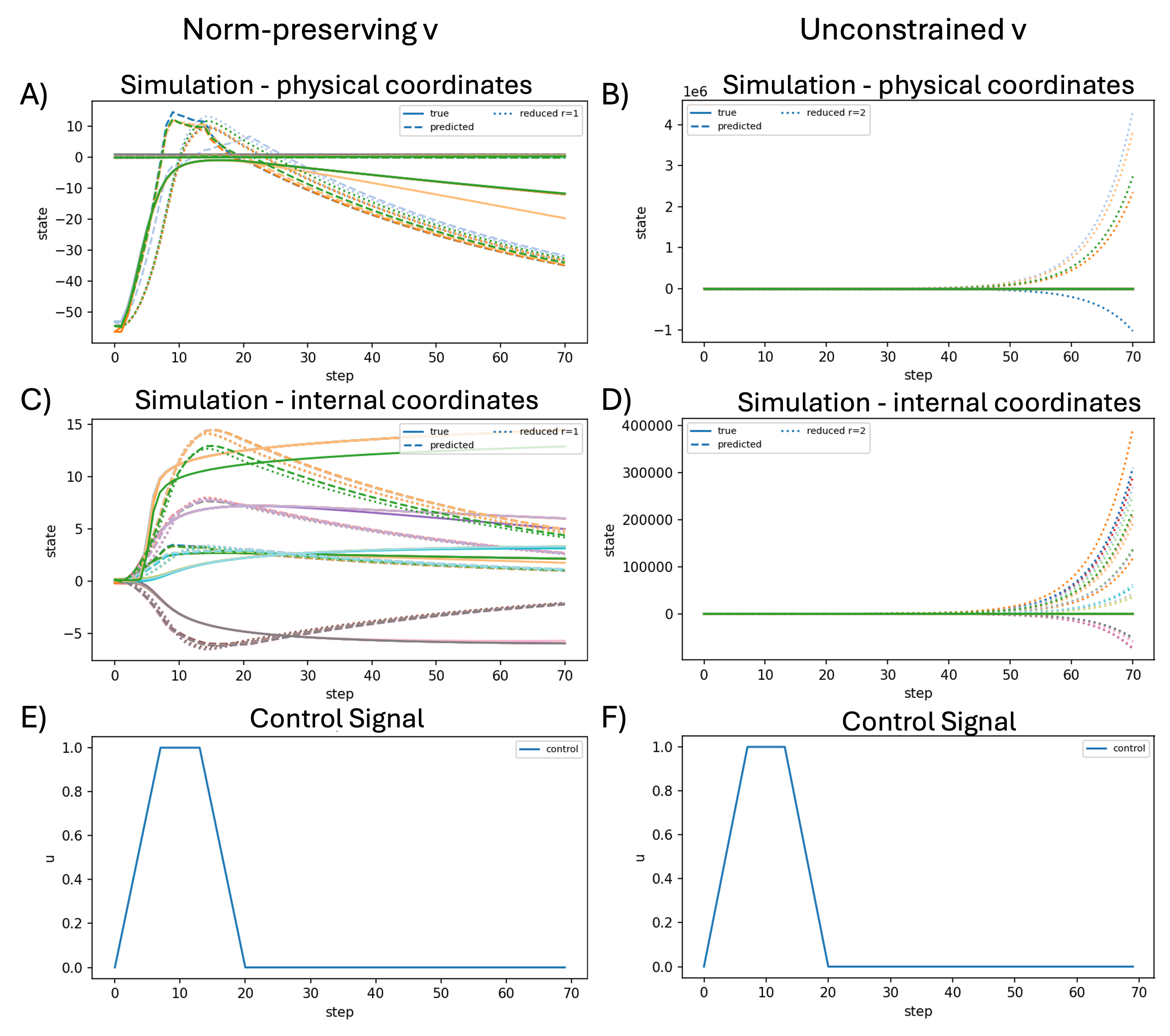}
    \caption{
    \textbf{Time-domain consequence of input norm calibration: stable prediction and meaningful reduction versus gain-induced blow-up.}
    Left column: norm-preserving $v$; right column: unconstrained $v$.
    (A--B) Rollout in physical coordinates under the same bounded control input $u_k$ (shown in (E--F)), comparing ground truth (true), the learned lifted surrogate (predicted), and a rank-$r{=}1$ or $2$ reduced surrogate (reduced $r{=}1$ or $2$).
    Under norm preservation (A), the surrogate and reduced model track the system over the horizon; under unconstrained lifting (B), trajectories diverge catastrophically despite bounded $u_k$, consistent with hidden gain amplification in $v(x,u)$.
    (C--D) Rollout in the internal (lifted) coordinates shows the same effect: norm-preserving lifting yields bounded internal dynamics, while unconstrained lifting drives the internal state to extreme magnitudes.
    (E--F) Control signal used in both experiments.
    Together with Fig.~\ref{fig:hh_certification_spectral}, these results illustrate why enforcing $\|v\|_2=\|u\|_2$ is essential: it fixes the input-energy gauge so that (i) the lifted LTI surrogate remains physically calibrated and (ii) Hankel/BT-based $\mathcal{H}_\infty$ bounds serve as genuine certificates for reduced-order models in the original input metric.
    }
    \label{fig:hh_rollout_time_domain}
\end{figure*}

%% file: subsections/conclusion.tex
\section{Conclusion}
We introduced a GSVD-based framework for certified reduction of nonlinear control systems that preserves the input-energy metric required for Hankel- and $\mathcal{H}_\infty$-based guarantees. The key step is an input-energy calibrated lifting $v(x,u)$ satisfying the pointwise constraint $\|v(x,u)\|_2=\|u\|_2$, together with a GSVD construction that represents general (including non-affine) input nonlinearities in an LTI-like lifted form $\dot z = Az + Bv(x,u)$ with constant $A,B$. In this representation, induced-gain quantities computed from the lifted surrogate are expressed in the physical input metric, restoring the interpretation of Hankel singular values as certificate-relevant input--output energy modes.

Under the assumptions of our main results, balanced truncation applied to the lifted surrogate yields an a priori $\mathcal{H}_\infty$ error bound for the reduced nonlinear input--output operator, consisting of the classical HSV tail term together with the additional state-dependent input-channel recomputation contribution quantified in Theorem~2. We validated the resulting certificate behavior on a $25$-state Hodgkin--Huxley network with saturating optogenetic inputs, where the calibrated construction admits stable reduced rollouts and empirically certifying bounds, while an unconstrained lifted input exhibits the metric-mismatch failure mode discussed in the Introduction.

Finally, we showed how approximate Koopman closure error can be treated as a feedback uncertainty, enabling learning-based parameterizations of the lifting while retaining conservative robustness guarantees under a small-gain condition. This connects data-driven lifting architectures to classical robust control tools for certified model reduction and analysis.

%% file: subsections/results_gramians_weak_assumptions.tex
\subsection{Appendix A: Approximate Koopman Representations (Definitions and Proofs)}
\label{appendix:koopman_with_error}

This appendix provides the definitions, intermediate constructions, and proofs supporting the Part~2 result stated in Theorem~\ref{theorem:bound_with_error}.
The section consists primarily of a generalization to systems whose Koopman representations are only approximate in finite dimensions. 

As in the main manuscript, we measure nonlinear input--output gains using the restricted induced $L^2\to L^2$ gain over an admissible input class $\mathcal{U}\subset L^2$:
\[
\|G\|_{H_\infty(\mathcal{U})} \triangleq \sup_{u\in\mathcal{U}\setminus\{0\}} \frac{\|G(u)\|_{L^2}}{\|u\|_{L^2}}.
\]

When $G$ is stable LTI, we use the classical $H_\infty$ norm (equivalently, take $\mathcal{U}=L^2$) and write $\|G\|_{H_\infty}$.
Unless stated otherwise, all gains are evaluated at the equilibrium initial condition $x(0)=0$ (equivalently $\varphi_0(0)=0$ and $z(0)=0$).

\subsubsection{Balanced Coordinates for Nonlinear Systems with Error in the Koopman Representation}

Approximate Koopman closure augments the lifted autonomous dynamics with an additional residual term. Under a state-inclusive lift on the forward-invariant set, this residual is a well-defined function of the lifted state and can be treated as an additional channel acting on the lifted dynamics. The GSVD factorization developed earlier provides a representation in which all gain associated with this residual is carried by a constant matrix, while the remaining nonlinearity preserves energy pointwise. When this channel is combined with the Part~1 calibrated actuation representation, the resulting lifted model has the same structural form needed to (i) interpret closure error as a feedback interconnection and (ii) apply classical balancing and truncation tools to a nominal LTI surrogate.

This subsubsection fixes the lifted and balanced-coordinate representations so that, in the next subsection, the closure residual can be written as a positive-feedback interconnection and controlled via a small-gain gap bound, after which the Part~1 truncation certificate applies directly to the nominal surrogate.

\begin{definition}
\label{def:j_plus}
Let $\mathcal{J}^+$ denote the class of systems satisfying the axioms of $\mathcal{J}$,
except that the Koopman closure holds only up to an additive residual.
That is, there exists a mapping $f_{\mathrm{error}}:\mathbb{R}^q\to\mathbb{R}^q$
with $f_{\mathrm{error}}(0)=0$ and $\|f_{\mathrm{error}}\|_{2\to2}<\infty$ such that
\begin{align}
D\varphi_0(x)\,f(x,0)=A\varphi_0(x)+f_{\mathrm{error}}(\varphi_0(x)),
\qquad \forall x\in\mathcal{X}.
\label{eq:approx_closure_def}
\end{align}
\end{definition}

\begin{remark}
Because the lifting is \emph{state-inclusive} on $\mathcal{X}$, i.e.
\[
\varphi_0(x)=\begin{bmatrix}x\\ \varphi_{\mathrm{lift}}(x)\end{bmatrix}\quad \forall x\in\mathcal{X},
\]
the map $\varphi_0$ is injective on $\mathcal{X}$ and admits a left-inverse
$x=S\phi$ on $\varphi_0(\mathcal{X})$, where $S=\begin{bmatrix}I_{n\times n}&0\end{bmatrix}$.
Consequently, for any $\phi\in\varphi_0(\mathcal{X})$ there is a unique state
$x=S\phi$, and the residual
\[
D\varphi_0(x)f(x,0)-A\varphi_0(x)
\]
can be defined unambiguously as a function of $\phi=\varphi_0(x)$ alone.
Hence $f_{\mathrm{error}}$ is well-defined on $\varphi_0(\mathcal{X})$
(and may be extended arbitrarily to all of $\mathbb{R}^q$ if desired).
\end{remark}

\begin{observation}[SVD-like factorization of the closure residual]
\label{obs:closure_residual_factorization}
By \eqref{eq:approx_closure_def}, the residual satisfies
\begin{align}
f_{\mathrm{error}}(\varphi_0(x)) \triangleq D\varphi_0(x)\, f(x,0) - A \varphi_0(x).
\label{eq:ferr_def}
\end{align}
Moreover, since $x(0)=0$ implies $\varphi_0(0)=0$ and $f(0,0)=0$, we have $f_{\mathrm{error}}(0)=0$.

Applying Corollary~\ref{cor:extremal_direction_coordinates} to the map
$f_{\mathrm{error}}:\mathbb{R}^q\to\mathbb{R}^q$ yields an orthogonal matrix
$U_{\mathrm{error}}\in\mathbb{R}^{q\times q}$, a rectangular diagonal matrix
$\Sigma_{\mathrm{error}}\in\mathbb{R}^{q\times 2q}$, and a mapping
$v_{\mathrm{error}}:\mathbb{R}^{q}\to\mathbb{R}^{2q}$ such that
\begin{align}
f_{\mathrm{error}}(\varphi_0) &= U_{\mathrm{error}} \Sigma_{\mathrm{error}}\, v_{\mathrm{error}}(\varphi_0),
\\
\|v_{\mathrm{error}}(\varphi_0)\|_2&=\|\varphi_0\|_2.
\end{align}
Defining $D_{\mathrm{error}}\triangleq U_{\mathrm{error}}\Sigma_{\mathrm{error}}$, we equivalently have
\begin{align}
D\varphi_0(x)\,f(x,0)=A\varphi_0(x)+D_{\mathrm{error}}v_{\mathrm{error}}(\varphi_0(x)).
\end{align}
\end{observation}

\begin{corollary}[LTI-like lifted representation under approximate Koopman closure]
\label{cor:lti_like_with_err}
For every system $G\in\mathcal{J}^+$ there exist matrices
$A\in\mathbb{R}^{q\times q}$ (Hurwitz), $B\in\mathbb{R}^{q\times(p+q)}$, $C\in\mathbb{R}^{m\times q}$,
and $D_{\mathrm{error}}\in\mathbb{R}^{q\times 2q}$, together with mappings
\[
v:\mathcal{X}\times\mathbb{R}^p\to\mathbb{R}^{p+q},
\qquad
v_{\mathrm{error}}:\mathbb{R}^q\to\mathbb{R}^{2q},
\]
satisfying
\begin{align}
\|v(x,u)\|_2&=\|u\|_2, \quad \forall x\in\mathcal{X},\ u\in\mathbb{R}^p, 
\label{eq:v_norm_preserve_u_appx}
\\
\|v_{\mathrm{error}}(\varphi_0(x))\|_2&=\|\varphi_0(x)\|_2, \quad \forall x\in\mathcal{X},
\label{eq:verr_norm_preserve_phi_appx}
\end{align}
such that $G$ admits the lifted representation
\begin{equation}
\begin{aligned}
D\varphi_0(x)\,f(x,u) &= A \varphi_0(x) + D_{\mathrm{error}}v_{\mathrm{error}}(\varphi_0(x)) + B\,v(x,u), \\
y &= C \varphi_0(x),
\end{aligned}
\label{eq:lifted_rep_with_err}
\end{equation}
for all $x\in\mathcal{X}$ and $u\in\mathbb{R}^p$.
Moreover, if $B$ is chosen full row rank, then defining
\begin{align}
v'(x,u)\triangleq v(x,u)+B^\dagger D_{\mathrm{error}}v_{\mathrm{error}}(\varphi_0(x))
\label{eq:vprime_def}
\end{align}
yields the equivalent form
\begin{equation}
\begin{aligned}
D\varphi_0(x)\,f(x,u)&=A\varphi_0(x)+Bv'(x,u),\\ y&=C\varphi_0(x).
\end{aligned}
\end{equation}
\end{corollary}

\begin{proof}
The approximate Koopman closure is \eqref{eq:approx_closure_def}. By
Observation~\ref{obs:closure_residual_factorization}, the residual admits the representation
\[
D\varphi_0(x)\,f(x,0)=A\varphi_0(x)+D_{\mathrm{error}}v_{\mathrm{error}}(\varphi_0(x)).
\]
The control-induced term is handled exactly as in the proof of Theorem~\ref{theorem:norm_preserving}:
define $f_u(x,u)\triangleq f(x,u)-f(x,0)$ and $g(x,u)\triangleq D\varphi_0(x)\,f_u(x,u)$, so that
\[
D\varphi_0(x)\,f(x,u)=D\varphi_0(x)\,f(x,0)+g(x,u).
\]
Since $f$ is Lipschitz in $u$ uniformly over $x\in\mathcal X$ and $D\varphi_0$ is bounded on $\mathcal X$,
the map $(x,u)\mapsto g(x,u)$ has finite induced gain in $u$ on $\mathcal X$.
Applying Lemma~\ref{lem:norm_preserve_u} yields a constant matrix $B$ and a pointwise norm-preserving map $v(x,u)$
such that $g(x,u)=Bv(x,u)$ and $\|v(x,u)\|_2=\|u\|_2$.
Substituting into the lifted split gives \eqref{eq:lifted_rep_with_err}.
If $B$ is full row rank, the absorption step \eqref{eq:vprime_def} yields the final equivalent form.
\end{proof}

\begin{remark}[Balanced lifted coordinates]
Let $T$ be a balancing transform for the associated LTI surrogate, and define
\begin{equation}
\begin{aligned}
z \triangleq T\varphi_0(x),\qquad
\tilde A \triangleq &TAT^{-1},\\
\tilde B \triangleq TB,\qquad
\tilde C \triangleq CT^{-1},\qquad
&\tilde D_{\mathrm{error}} \triangleq T D_{\mathrm{error}}.
\end{aligned}
\end{equation}
By state-inclusivity, there exists $S=\begin{bmatrix}I_{n\times n}&0\end{bmatrix}$ such that
$x=S\varphi_0(x)$, hence $x=Rz$ with $R\triangleq ST^{-1}$.
Define
\[
\tilde v(z,u)\triangleq v(Rz,u),
\qquad
\tilde v_{\mathrm{error}}(z)\triangleq v_{\mathrm{error}}(T^{-1}z).
\]
Then $\|\tilde v(z,u)\|_2=\|u\|_2$ pointwise, and since
$\|v_{\mathrm{error}}(\eta)\|_2=\|\eta\|_2$ for all $\eta\in\mathbb{R}^q$, we have
\[
\|\tilde v_{\mathrm{error}}(z)\|_2
=\|v_{\mathrm{error}}(T^{-1}z)\|_2
=\|T^{-1}z\|_2
\le \|T^{-1}\|_{2\to2}\,\|z\|_2.
\]
Along trajectories,
\[
\dot z=\tilde A z + \tilde D_{\mathrm{error}}\tilde v_{\mathrm{error}}(z)+\tilde B\,\tilde v(z,u),
\qquad
y=\tilde C z.
\]
We use this balanced-coordinate representation throughout Appendix~\ref{appendix:koopman_with_error}.
\end{remark}

\subsubsection{Modeling Error as Feedback}
The following results treat modeling error in the Koopman representation of the autonomous dynamics through a robust-control lens. The closure residual is represented as a memoryless uncertainty block interconnected with a nominal lifted plant in positive feedback. 
The analysis below uses small-gain conditions to guarantee the feedback interconnection is well-posed and finite-gain stable.

The small-gain condition is a standard sufficient condition for well-posedness and finite-gain stability of a (positive) feedback interconnection. In particular, for two causal operators $G_1$ and $G_2$ with finite induced $L^2\!\to\!L^2$ gains, the interconnection is well-posed whenever \begin{align} \|G_1G_2\|_{H_\infty} < 1, \end{align} in which case the inverse $(I-G_1G_2)^{-1}$ exists as a bounded causal operator (see, e.g., \cite{dullerud2000chapter4}).

A related canonical result bounds the induced-norm gap between the positive-feedback interconnection and the corresponding open-loop plant with feedback removed. This upper bound (reviewed below) is used to control the error introduced by removing and then reattaching the closure-error channel.
\begin{observation}[Small-gain gap bound under positive feedback]
\label{observation:feedback_bound}
Let $G_1$ be a stable linear (LTI) operator and let $G_2$ be a stable causal operator, both mapping $L^2\to L^2$ with finite induced gains (which we denote by $\|\cdot\|_{H_\infty}$), such that
\begin{equation}
\begin{aligned}
    y_1 &= G_1(u_1), \\
    y_2 &= G_2(u_2),
\end{aligned}
\end{equation}
and the small-gain condition is met:
\begin{align}
      \bigl\|\,G_1 G_2\,\bigr\|_{H_\infty} < 1.
\end{align}
Close the loop \emph{positively} by setting:
\begin{equation}
\begin{aligned}
    u_1 &= y_2 , \\
    u_2 &= y_1 ,
\end{aligned}
\end{equation}
and attach an external signal $r$ additively to $u_1$ 
(the particular injection point of $r$ is arbitrary and does not affect the bound).  

The closed-loop map from $r$ to $y_1$ is:
\begin{align}
    M = \bigl(I - G_1\,G_2\bigr)^{-1}\,G_1.
\end{align}
Then
\begin{align}
\|M - G_1 \|_{H_\infty}
\le
\frac{\|G_1G_2\|_{H_\infty}\,\|G_1\|_{H_\infty}}{1 - \|G_1 G_2\|_{H_\infty}}.
\end{align}
\end{observation}
\begin{proof}
Let $y_1$ denote the closed-loop output corresponding to input $r$.
Under the positive-feedback interconnection with additive injection at $u_1$, we have
\[
u_1 = r + y_2,\qquad u_2 = y_1,\qquad y_1 = G_1(u_1),\qquad y_2 = G_2(u_2),
\]
and therefore
\begin{align}
y_1
= G_1(r + y_2)
= G_1\bigl(r + G_2(y_1)\bigr).
\end{align}
Let $y_{1,0} \triangleq G_1(r)$ denote the output with the feedback removed, and define the difference
\begin{align}
e \triangleq y_1 - y_{1,0}.
\end{align}
Since $G_1$ is linear,
\begin{align}
e
= G_1\bigl(r + G_2(y_1)\bigr) - G_1(r)
= G_1\bigl(G_2(y_1)\bigr).
\end{align}
Hence, by the induced-gain definition,
\begin{align}
\|e\|_{L^2} \le \|G_1 G_2\|_{H_\infty}\, \|y_1\|_{L^2}.
\end{align}
Also,
\begin{align}
\|y_1\|_{L^2}
\le \|y_{1,0}\|_{L^2} + \|e\|_{L^2}
\le \|G_1\|_{H_\infty}\|r\|_{L^2} + \|e\|_{L^2}.
\end{align}
Combining gives
\begin{align}
\|e\|_{L^2}
\le \|G_1 G_2\|_{H_\infty}\bigl(\|G_1\|_{H_\infty}\|r\|_{L^2} + \|e\|_{L^2}\bigr),
\end{align}
so if $\|G_1 G_2\|_{H_\infty}<1$,
\begin{align}
\|e\|_{L^2}
\le
\frac{\|G_1 G_2\|_{H_\infty}\,\|G_1\|_{H_\infty}}{1-\|G_1 G_2\|_{H_\infty}}\,\|r\|_{L^2}.
\end{align}
Since $y_1 = Mr$ and $y_{1,0}=G_1 r$, we have $e=(M-G_1)r$, and the claimed induced-gain bound follows.
\end{proof}

The closure residual enters the lifted dynamics as a state-dependent additive channel and therefore admits a feedback interpretation. The next lemma makes this interconnection explicit, which allows Observation~\ref{observation:feedback_bound} to be applied before truncation and again after reduction. The lemma is stated with identity output; output-map differences are bounded separately later.

\begin{lemma}
\label{lemma:error_as_feedback}
For any system, $M \in \mathcal{J}^+$ that admits a state space representation with identity output:
\begin{equation}
\label{eq:linear_system_rep}
\begin{aligned}
\dot{z} &= \tilde{A} z + \tilde{B} \tilde{v}'(z,u) \\
y &= I z,
\end{aligned}
\end{equation}
with $\tilde{v}'(z,u)$ defined as:
\begin{equation}
\begin{aligned}
\tilde{v}'(z,u) \triangleq \tilde{v}(z,u) + \tilde{B}^{\dagger} \tilde{D}_{\mathrm{error}} \tilde{v}_{\mathrm{error}}(z),
\end{aligned}
\end{equation}
there exist systems $G_1$ and $G_2$ such that $M$ is equivalent to their positive closed-loop interconnection.
In particular, let $G_1$ be the stable LTI system with exogenous input $r$ and feedback input $v_1$:
\begin{equation}
\begin{aligned}
\dot{z}_1 &= \tilde{A} z_1 + \tilde{B}(r + v_1) \\
y_1 &= I z_1,
\end{aligned}
\end{equation}
and let $G_2$ be the memoryless operator
\begin{align}
y_2 = \tilde{B}^{\dagger}\tilde{D}_{\mathrm{error}} \tilde{v}_{\mathrm{error}}(u_2).
\end{align}
Under the positive-feedback interconnection $v_1 = y_2$ and $u_2 = y_1$, the resulting closed-loop map $r \mapsto y_1$ is $M$.
Moreover, specializing $r \triangleq \tilde{v}(z_1,u)$ recovers the original realization of $M$ in terms of the control input $u$.
\end{lemma}
\begin{proof}
Direct substitution verifies the claim: under $v_1=y_2$ and $u_2=y_1$, the $G_1$ state equation becomes
$\dot z_1=\tilde A z_1+\tilde B(\tilde v(z_1,u)+\tilde B^\dagger \tilde D_{\mathrm{error}}\tilde v_{\mathrm{error}}(z_1))$,
which is exactly \eqref{eq:linear_system_rep} with identity output. The specialization
$r\triangleq \tilde v(z_1,u)$ recovers the original realization.
\end{proof}

\begin{corollary}
For any system $M \in \mathcal{J}^+$, expressed as the closed-loop feedback between two systems, $G_1$ and $G_2$ as defined in Lemma \ref{lemma:error_as_feedback}, if $G_1$ and $G_2$ satisfy the small-gain condition, namely
\begin{align}
\| G_1 G_2 \|_{H_\infty} < 1,
\end{align}
then the error between $M$ and $G_1$ is bounded as:
\begin{align}
\|M - G_1\|_{H_\infty} \le
\frac{\|G_1G_2\|_{H_\infty}\,\|G_1\|_{H_\infty}}{1-\|G_1G_2\|_{H_\infty}}.
\end{align}
\end{corollary}
\begin{proof}
This follows directly from Observation \ref{observation:feedback_bound}.
\end{proof}
\begin{definition}
\label{def:def_of_xi}
Let $\xi(\cdot,\cdot)$ denote the small-gain gap bound, defined by
\begin{align}
\xi(G_1,G_2)
\triangleq
\frac{\|G_1G_2\|_{H_\infty}\,\|G_1\|_{H_\infty}}{1-\|G_1G_2\|_{H_\infty}}.
\end{align}
\end{definition}
\begin{remark}
Since $\|\,\cdot\,\|_{H_\infty(\mathcal{U})}\le \|\,\cdot\,\|_{H_\infty}$ for any $\mathcal{U}\subset L^2$, the quantity $\xi(G_1,G_2)$ also upper-bounds the corresponding restricted induced gap on any admissible input class for the reference signal.
\end{remark}

\begin{definition}
\label{def:embedded_systems}
For a control system $G \in \mathcal{J}^+$, with state space representation:
\begin{equation}
\begin{aligned}
\dot{z} &= \tilde{A} z + \tilde{B} \tilde{v}'(z,u) \\
y &= \tilde{C} z,
\end{aligned}
\end{equation}
let $G^I$, $G^P$, and $G^E$ be related systems defined as follows.
Let the state space for $G^I$ be denoted as:
\begin{equation}
\begin{aligned}
\dot{z}_I &= \tilde{A} z_I + \tilde{B} \tilde{v}'(z_I,u) \\
y_I &= I z_I.
\end{aligned}
\end{equation}
Let the state space for $G^P$ be denoted as:
\begin{equation}
\begin{aligned}
\dot{z}_P &= \tilde{A} z_P + \tilde{B}(\tilde{v}(z_P,u) + v_p) \\
y_P &= I z_P.
\end{aligned}
\end{equation}
And finally, let the state space for $G^E$ be denoted as:
\begin{equation}
\begin{aligned}
y_E &= \tilde{B}^{\dagger} \tilde{D}_{\mathrm{error}} \tilde{v}_{\mathrm{error}}(z_E).
\end{aligned}
\end{equation}
\end{definition}
\begin{remark}
Note, as in Lemma \ref{lemma:error_as_feedback}, that $G^P$ in positive closed-loop feedback with $G^E$ yields $G^I$.
The final bound derived in this section will separately bound $\|G - G^I\|_{H_\infty}$ and \mbox{$\|G^I - G^P\|_{H_\infty}$}.
\end{remark}

The subsequent estimates treat the associated LTI surrogate as an $L^2\!\to\!L^2$ operator driven by the endogenous signal $\tilde v'(\cdot)$. The small-gain hypothesis ensures the closed-loop lifted state lies in $L^2$, and the calibration properties then force $\tilde v'(\cdot)\in L^2$ whenever $u(\cdot)\in L^2$. This is stated formally in the next corollary.



\begin{corollary}[Conditions for $\tilde{v}'$ to be in $L^2$]
\label{corollary:error_as_feedback}
For system any $G \in \mathcal{J}^+$, if its related systems, $G^P$ and $G^E$ satisfy the small gain condition,
\begin{align}
\|G^P G^E\|_{H_{\infty}} < 1,
\end{align}
then for any $u(t) \in L^2$, the resulting $\tilde{v}'(t)$ for $G$ will also be an element of $L^2$.
\end{corollary}
\begin{proof}
Fix any input $u(\cdot)\in L^2$.
Consider the positive-feedback interconnection of the sister systems $G^P$ and $G^E$
(from Definition~\ref{def:embedded_systems}), with an external signal $r$ injected additively at the $G^P$ input channel
(as in Observation~\ref{observation:feedback_bound}).
Let $z(\cdot)$ denote the resulting lifted state (so $y_P(\cdot)=z(\cdot)$ since $G^P$ uses identity output).

Define the injected signal by
\[
r(t)\triangleq \tilde v(z(t),u(t)).
\]
By the pointwise calibration $\|\tilde v(z,u)\|_2=\|u\|_2$, we have $r\in L^2$ and $\|r\|_{L^2}=\|u\|_{L^2}$.

Under the small-gain condition $\|G^P G^E\|_{H_\infty}<1$, the feedback interconnection is finite-gain stable.
In particular, for any $r\in L^2$ the closed-loop lifted state satisfies $z\in L^2$.

Now define the feedback contribution
\[
e(t)\triangleq G^E(z(t))=\tilde B^\dagger \tilde D_{\mathrm{error}}\,\tilde v_{\mathrm{error}}(z(t)).
\]
Using $\|\tilde v_{\mathrm{error}}(z)\|_2=\|T^{-1}z\|_2\le \|T^{-1}\|_{2\to2}\|z\|_2$ and submultiplicativity,
\begin{equation}
\begin{aligned}
\|e(t)\|_2
&\le
\|\tilde B^\dagger \tilde D_{\mathrm{error}}\|_{2\to2}\,\|\tilde v_{\mathrm{error}}(z(t))\|_2 \\
&\le
\|\tilde B^\dagger \tilde D_{\mathrm{error}}\|_{2\to2}\,\|T^{-1}\|_{2\to2}\,\|z(t)\|_2,
\end{aligned}
\end{equation}
so $e\in L^2$ whenever $z\in L^2$.

Finally, along trajectories,
\[
\tilde v'(t)=\tilde v(z(t),u(t))+\tilde B^\dagger \tilde D_{\mathrm{error}}\,\tilde v_{\mathrm{error}}(z(t))
= r(t)+e(t),
\]
so $\tilde v' \in L^2$.
\end{proof}

\begin{remark}
Under the small-gain condition, the trajectory-generated signal $\tilde v'(\cdot)$ belongs to $L^2$ whenever $u(\cdot)\in L^2$. Consequently, the associated LTI surrogate $G^L$ may be evaluated on $\tilde v'(\cdot)$ as an $L^2\!\to\!L^2$ operator (with equilibrium initial condition) when bounding induced gains, even though $\tilde v'(\cdot)$ is not an exogenous input in the original nonlinear system.
\end{remark}

\subsubsection{Transitioning between $G$ and $G^I$}
In Definition~\ref{def:embedded_systems} we introduced sister systems associated with each $G\in\mathcal{J}^+$, including the identity-output system $G^I$. The feedback representation of the closure residual is stated most cleanly in this identity-output form. This subsubsection bounds the input--output gap introduced by replacing $\tilde C$ with the identity.
\begin{definition}
\label{def:def_of_phi}
Given an LTI system, $G^L$, with state space representation,
\begin{equation}
\begin{aligned}
\dot{x} = A x + B u \\
y = C x,
\end{aligned}
\end{equation}
let
\begin{align}
\Phi(j \omega) \triangleq (j\omega I - A)^{-1} B,
\end{align}
such that
\begin{align}
\|\Phi\|_{H_\infty} = \sup_{\omega \in \mathbb{R}} \| \Phi(j \omega) \|_{2\to2}.
\end{align}
\end{definition}

\begin{observation}[Bounding the $H_{\infty}$--norm gap for identical $(A,B)$ pairs]
\label{observation:different_c}
Let two stable continuous-time LTI systems share the same state and input matrices $(A,B)$ and have no direct term:
\begin{align}
G_1(s) &= C_1 (sI-A)^{-1} B, \\
G_2(s) &= C_2 (sI-A)^{-1} B,
\end{align}
with $A\in\mathbb{R}^{n\times n}$, $B\in\mathbb{R}^{n\times l}$, and $C_1,C_2$ of compatible dimensions.
Let $\Phi$ denote the resolvent input map associated with $(A,B)$ as in Definition~\ref{def:def_of_phi}.
Then
\begin{align}
\|G_1-G_2\|_{H_\infty}
\;\le\;
\|C_1-C_2\|_{2\to2}\,\|\Phi\|_{H_\infty}.
\end{align}
\end{observation}
\begin{proof}
Pointwise in frequency,
\begin{align}
G_1(j\omega)-G_2(j\omega) \;=\;(C_1-C_2)\,\Phi(j\omega),
\end{align}
so by sub-multiplicativity of the spectral norm,
\begin{align}
\|G_1(j\omega)-G_2(j\omega)\|_{2\to2}
      \;\le\;\|C_1-C_2\|_{2\to2}\,\|\Phi(j\omega)\|_{2\to2}.
\end{align}
Taking the supremum over $\omega$ gives
\begin{align}
\|G_1-G_2\|_{H_\infty}\le\|C_1-C_2\|_{2\to2}\,\|\Phi\|_{H_\infty},
\end{align}

which proves the stated bound.
\end{proof}

\begin{definition}[Output-map embedding]
\label{def:embedded_c}
Let $\tilde C\in\mathbb{R}^{m\times q}$ with $m\le q$, and define the embedding
\[
E_m \triangleq
\begin{bmatrix}
I_m\\
0_{(q-m)\times m}
\end{bmatrix}\in\mathbb{R}^{q\times m}.
\]
Define the embedded output map
\begin{equation}
\label{eq:C0_def}
\tilde C_0 \triangleq E_m \tilde C \in \mathbb{R}^{q\times q}.
\end{equation}
\end{definition}

\begin{corollary}[Different output dimensions]
\label{corollary:embed_c}
In Observation~\ref{observation:different_c}, if $C_1$ and $C_2$ have different output dimensions, replace the smaller map by its embedding $\tilde C_0$ from Definition~\ref{def:embedded_c}. The bound then applies with both output maps interpreted as having the same dimension.
\end{corollary}

\subsubsection{Reduced Nonlinear Systems}
The final bound is obtained by chaining comparisons
\[
G \;\to\; G^I \;\to\; G^P \;\to\; G_r^P \;\to\; G_r^I \;\to\; G_r,
\]
where the subscript $r$ denotes the order-$r$ reduced (truncated) counterpart of the corresponding system.
The nominal plant $G^P$ contains no closure-error feedback and therefore falls directly under Theorem~\ref{theorem:final_bound}, yielding a certified truncation $G^P\to G_r^P$.
However, $G_r^P$ alone does not represent a reduced model of the original system, because it omits the closure-error feedback channel.

Accordingly, the reduced-order construction must specify how this feedback channel is removed and then reintroduced after truncation. This is done by defining reduced-order counterparts of the identity-output and feedback blocks ($G_r^I$ and $G_r^E$) so that the same positive-feedback interconnection used at full order can be formed at reduced order.

In the exact-closure case (no Koopman modeling error), the feedback channel is absent, and the comparison reduces to the direct reduction bound for $G\to G_r$ given by Theorem~\ref{theorem:final_bound}.

\textbf{Theorem~\ref{theorem:bound_with_error}
(Restated for convenience).} \emph{Let $G\in\mathcal{J}^+$ and fix an admissible input class $\mathcal{U}\subset L^2$ (all induced gains evaluated at the equilibrium initial condition).}
\emph{Let $G^I$, $G^P$, and $G^E$ denote the identity-output and feedback sister systems associated with $G$ (Definition~\ref{def:embedded_systems}),}
\emph{and let $G_r^P$, $G_r^E$, and $G_r^I$ denote their order-$r$ reduced counterparts obtained by truncating the nominal LTI surrogate and re-forming the same feedback interconnection at reduced order.}

\emph{Assume the small-gain conditions hold for the full- and reduced-order feedback interconnections:}
\[
\|G^P G^E\|_{H_\infty} < 1,
\qquad
\|G_r^P G_r^E\|_{H_\infty} < 1.
\]
\emph{Assume further that for every $u(\cdot)\in\mathcal{U}$, all trajectories compared below exist for all $t\ge 0$ and remain in the prescribed compact set $\mathcal{X}$.}
\emph{Equivalently, in balanced lifted coordinates $z\in\mathbb{R}^q$ and reduced coordinates $z_r\in\mathbb{R}^r$, we have}
\[
R z(t)\in\mathcal{X},
\qquad
R E z_r(t)\in\mathcal{X}
\quad\forall t\ge 0,
\]
\emph{where $R$ is the state readout map in balanced coordinates and}
\[
E:\mathbb{R}^r\to\mathbb{R}^q,\qquad
E z_r \triangleq \begin{bmatrix} z_r \\ 0_{q-r}\end{bmatrix}
\]
\emph{is the canonical zero-padding embedding.}
\medskip \emph{Let $(G^P)^L$ denote the associated stable LTI surrogate obtained by treating the calibrated actuation signal as an exogenous input $w$:}
\[
\dot z = \tilde A z + \tilde B w,\qquad y = I z,
\]
\emph{and let $\{\nu_i\}_{i=1}^q$ be the Hankel singular values of $(G^P)^L$.}
\emph{Let $\Phi$ and $\Phi_r$ denote the resolvent input maps associated with the full- and order-$r$ truncated surrogates (Definition~\ref{def:def_of_phi}).}
\medskip \emph{Then the induced input--output error between $G$ and its order-$r$ reduced model $G_r$ satisfies}
\begin{equation}
\label{eq:bound_with_error_clean}
\begin{aligned}
\|G - G_r\|_{H_\infty(\mathcal{U})}
\;\le\;&
\|\tilde C_0 - I\|_{2\to2}\,\|\Phi\|_{H_\infty}
\\
&+\;\xi(G^P,G^E)
\\
&+\; 2\Big(\|(G^P)^L\|_{H_\infty} \;+\; \sum_{i=r+1}^{q}\nu_i\Big)
\\
&+\;\xi(G_r^P,G_r^E)
\\
&+\;\|\tilde C_{r,0} - I_{r\times r}\|_{2\to2}\,\|\Phi_r\|_{H_\infty},
\end{aligned}
\end{equation}
\emph{where $\xi(\cdot,\cdot)$ is the small-gain gap bound (Definition~\ref{def:def_of_xi}), and $\tilde C_0$ and $\tilde C_{r,0}$ are the zero-padded embeddings of the full and reduced output maps (Definition~\ref{def:embedded_c}).}

\begin{proof}
Throughout the proof, we bound restricted induced gains using the fact that
$\|G\|_{H_\infty(\mathcal{U})} \le \|G\|_{H_\infty}$ for any operator $G$, since $\mathcal{U}\subset L^2$.

\noindent\emph{Step 1: Replace $\tilde{C}$ with $I$.}
By Observation~\ref{observation:different_c},
\begin{align}
\|G - G^I \|_{H_\infty(\mathcal{U})} \le \| \tilde{C}_0 - I \|_{2\to2} \| \Phi \|_{H_\infty}.
\end{align}

\noindent\emph{Step 2: Remove feedback error.}
By the feedback representation (Lemma~\ref{lemma:error_as_feedback}) and the small-gain gap bound (Definition~\ref{def:def_of_xi}),
\begin{align}
\| G^I - G^P \|_{H_\infty(\mathcal{U})} &\le \xi(G^P, G^E).
\end{align}

\noindent\emph{Step 3: Truncate the nominal plant.}
With $v_P = 0$, the system $G^P$ is in the setting of Theorem~2 (exact-closure/non-feedback case), so applying Theorem~2 to $G^P$
yields an order-$r$ truncation $G_r^P$ such that
\begin{align}
\| G^P - G_r^P \|_{H_\infty(\mathcal{U})} \le 2\|(G^P)^L\|_{H_\infty} + 2\sum_{i=r+1}^q \nu_i.
\end{align}

\noindent\emph{Step 4: Reintroduce feedback error at reduced order.}
Define $G_r^P$ and $G_r^E$ as the reduced-order analogues of Lemma~\ref{lemma:error_as_feedback} so that their positive feedback interconnection
is $G_r^I$. Then, by the same small-gain gap bound,
\begin{align}
\| G_r^I - G_r^P \|_{H_\infty(\mathcal{U})} \le \xi(G_r^P, G_r^E).
\end{align}

\noindent\emph{Step 5: Replace $I$ with $\tilde{C}_r$.}
Define $\tilde{C}_{r,0}$ by applying Definition~\ref{def:embedded_c} to $\tilde{C}_r$, so that $\tilde{C}_{r,0}\in\mathbb{R}^{r\times r}$.

By Observation~\ref{observation:different_c},
\begin{align}
\| G_r^I - G_r \|_{H_\infty(\mathcal{U})} \le \| \tilde{C}_{r,0} - I_{r \times r} \|_{2\to2} \| \Phi_r \|_{H_\infty}.
\end{align}

\noindent
Summing the bounds from Steps~1--5 gives the claimed result.
\end{proof}

%% file: subsections/appendix_new.tex
\newcommand{\gNa}{g_\mathrm{Na}}
\newcommand{\gK}{g_\mathrm{K}}
\newcommand{\gL}{g_\mathrm{L}}
\newcommand{\ENa}{E_\mathrm{Na}}
\newcommand{\EL}{E_\mathrm{L}}
\newcommand{\EE}{E_\mathrm{E}}
\newcommand{\EI}{E_\mathrm{I}}
\newcommand{\Vscale}{V_{\mathrm{scale}}}
\newcommand{\Vstar}{V^\star}
\newcommand{\EChR}{E_{\mathrm{ChR}}}
\newcommand{\gChRmax}{g_{\mathrm{ChR,max}}}
\newcommand{\Kd}{K_\mathrm{d}}
\newcommand{\RTF}{R T/F}
\newcommand{\Krest}{K_{\mathrm{rest}}}
\newcommand{\Kinn}{K_{\mathrm{in}}}
\newcommand{\dt}{\Delta t}
\newcommand{\alphahh}{\alpha_{\mathrm{HH}}}

\subsection{Hodgkin-Huxley Network: Discrete, Centered Representation}

\subsubsection{Overview}
We model a network of Hodgkin-Huxley (HH) neurons using the same ionic channel kinetics as the classical continuous-time HH equations, but we integrate them with a fixed discrete time step and centered internal coordinates so that $x=0$ is an equilibrium. All gates (ion and synaptic) live in logit space to keep them in $(0,1)$ after discretization, and the control input is a nonnegative optogenetic drive shared across neurons.

\subsubsection{Baseline Single-Neuron HH Model}
The continuous-time HH membrane equation for voltage $V$ (mV) with sodium, potassium, and leak currents is
\[
C_m\,\dot V = -\big(I_{\mathrm{Na}} + I_{\mathrm{K}} + I_{\mathrm{L}} + I_{\mathrm{ext}} \big),
\]
with Ohmic currents
\begin{equation}
\begin{aligned}
I_\mathrm{Na} &= \gNa m^3 h (V - \ENa) \\
I_\mathrm{K} &= \gK n^4 (V - E_\mathrm{K}) \\  
I_\mathrm{L} &= \gL (V - \EL),
\end{aligned}
\end{equation}
where $m,h,n\in(0,1)$ are gating probabilities. The gating ODEs use voltage-dependent rate constants $a_p(V),b_p(V)$:
\begin{align}
\dot p &= a_p(V)(1-p) - b_p(V)p, \qquad p\in\{m,h,n\},
\end{align}
with the empirical formulas
\begin{align*}
a_m(V)&=0.1\,\frac{25-V}{e^{(25-V)/10}-1}, & b_m(V)&=4e^{-V/18},\\
a_h(V)&=0.07e^{-V/20}, & b_h(V)&=\frac{1}{e^{(30-V)/10}+1},\\
a_n(V)&=0.01\,\frac{10-V}{e^{(10-V)/10}-1}, & b_n(V)&=0.125e^{-V/80}.
\end{align*}
We also use the steady-state/time-constant view $p_\infty(V)=a_p/(a_p+b_p)$ and $\tau_p(V)=1/(a_p+b_p)$.

\subsubsection{Centered internal coordinates}
We work with a centered \emph{internal} state $x_k\in\mathbb{R}^{4N+N_E+N_I}$, but perform all physics and updates in an \emph{offset state}
\[
\hat x_{k} \;\triangleq\; x_k + x_{\star},
\]
where the constant offset $x_{\star}$ is chosen so that $x=0$ is a fixed point of the full discrete network map (see below).
We partition the offset-state as
\[
\hat x_{k} =
\big[\,\hat V_{k}^\top,\ z_{m,k}^\top,\ z_{h,k}^\top,\ z_{n,k}^\top,\ z_{s_E,k}^\top,\ z_{s_I,k}^\top\,\big]^\top,
\]
with $\hat V_{k}\in\mathbb{R}^N$, $z_{m,k},z_{h,k},z_{n,k}\in\mathbb{R}^N$, $z_{s_E,k}\in\mathbb{R}^{N_E}$, and $z_{s_I,k}\in\mathbb{R}^{N_I}$.

\paragraph{Physical reconstruction from the offset-state.}
Voltage is reconstructed by
\[
V_k \;=\; \EL\mathbf{1} + \Vscale\, \hat V_{k},
\]
and ion/synaptic gates are reconstructed via logits:
\begin{equation}
\begin{aligned}
m_k=\sigma(z_{m,k}),\quad h_k=\sigma(z_{h,k}),\\ n_k=\sigma(z_{n,k}),\quad s_{E,k}=\sigma(z_{s_E,k}),\quad s_{I,k}=\sigma(z_{s_I,k}),
\end{aligned}
\end{equation}
where $\sigma(z)=\tfrac12(1+\tanh(\tfrac{z}{2}))$.
In software, the inverse logit $z=\log\!\frac{p}{1-p}$ is evaluated after clipping $p$ to $[\varepsilon,1-\varepsilon]$ with $\varepsilon=10^{-6}$ to avoid overflow.

\paragraph{Initialization and refinement of $x_{\star}$.}
We initialize $x_{\star}$ from the single-cell rest guess $V^\star=-54.4$ mV and the classical HH steady states
$m^\star=3.44\!\times\!10^{-5}$, $h^\star=0.9998$, $n^\star=0.00416$, together with synaptic gates centered near closed ($s^\star=10^{-4}$).
This yields the initial offset blocks
\[
V_{\star} = \frac{V^\star-\EL}{\Vscale}\mathbf{1},\qquad
z_{p,\star}=\log\!\frac{p^\star}{1-p^\star},\;\; p\in\{m,h,n,s_E,s_I\}.
\]
Because random connectivity breaks symmetry, the true network equilibrium is generally \emph{neuron-specific}; accordingly, after refinement the voltage offset $V_{\star}\in\mathbb{R}^N$ need not be constant across neurons.
We refine $x_{\star}$ numerically by simulating the internal dynamics from $x_0=0$ with $u\equiv 0$ for many steps, obtaining a terminal drift $x_T$, and updating
\[ 
x_{\star} \leftarrow \hat x_{\star} + x_T,
\]
so that (empirically) $x=0$ becomes a fixed point of the full network map used for data generation and learning.

\subsubsection{Discrete-time update}
With step size $\dt=0.05$, the discrete map updates the \emph{offset-state} $\hat x_{k}$ and then recenters:
\[
\hat x_{k} = x_k + \hat x_{\star},\qquad
x_{k+1} = \hat x_{k+1} - \hat x_{\star}.
\]

\paragraph{Voltage update.}
We use forward Euler on the offset-voltage coordinate $\hat V_{k}$:
\[
\hat V_{k+1}
\;=\;
\hat V_{k}
-\dt\,\alphahh\Big(I_{\mathrm{Na}} + I_{\mathrm{K}} + I_{\mathrm{L}} + I_{\mathrm{syn}} + I_{\mathrm{ChR}}\Big),
\]
where all currents are computed in physical units from the reconstructed $V_k=\EL\mathbf{1}+\Vscale \hat V_{k}$ and gates.
The scalar
\begin{equation}
\begin{aligned}
\alphahh \;&\triangleq\; \frac{\tau_{\mathrm{scale}}\gL}{C_m \Vscale} \\
(\tau_{\mathrm{scale}}=1,\; C_m=1,\; \gL&=0.3,\; \Vscale=20\;\Rightarrow\;\alphahh=0.015)
\end{aligned}
\end{equation}
is the constant voltage time-scale factor used in the implementation; equivalently, the effective voltage step size is the product $\dt\,\alphahh$.

\paragraph{Ion-channel gate updates}
Ion-channel gates are updated in probability space using exponential Euler with $V_k$ held fixed over the step:
\[
p_{\infty}(V) = \frac{a_p(V)}{a_p(V)+b_p(V)}, \qquad
\tau_p(V) = \frac{1}{a_p(V)+b_p(V)},
\]
\begin{equation}
\begin{aligned}
p_{k+1} = p_{\infty}(V_k) + \big(p_k - p_{\infty}(V_k)\big)\,e^{-\dt/\tau_p(V_k)},\\
z_{p,k+1} = \log\!\frac{p_{k+1}}{1-p_{k+1}},\quad p\in\{m,h,n\}.
\end{aligned}
\end{equation}
The rate functions are the classical HH empirical forms
\begin{align*}
a_m(V)&=0.1\,\frac{25-V}{e^{(25-V)/10}-1}, & b_m(V)&=4e^{-V/18},\\
a_h(V)&=0.07e^{-V/20}, & b_h(V)&=\frac{1}{e^{(30-V)/10}+1},\\
a_n(V)&=0.01\,\frac{10-V}{e^{(10-V)/10}-1}, & b_n(V)&=0.125e^{-V/80},
\end{align*}
interpreted in their continuous extension at the removable singularities.

\paragraph{Synaptic gate updates}
Synaptic gates mirror the same relaxation form, driven by presynaptic voltages.
Let
\[
s_\infty(V;V_\theta,k)=\tfrac12\!\left(1+\tanh\!\frac{V-V_\theta}{2k}\right),
\]
with $V_{\theta,E}=V_{\theta,I}=-20$ and $k_E=k_I=2$.
In the implementation, the first $N_E$ neurons are labeled excitatory and the remaining $N_I$ inhibitory, so $V_{\mathrm{pre},E,k}$ and $V_{\mathrm{pre},I,k}$ are taken from those respective presynaptic voltage blocks.
We update
\[
s_{E,k+1} = s_\infty(V_{\mathrm{pre},E,k}) + \big(s_{E,k}-s_\infty(V_{\mathrm{pre},E,k})\big) e^{-\dt/\tau_E}, \; \tau_E=3,
\]
\[
s_{I,k+1} = s_\infty(V_{\mathrm{pre},I,k}) + \big(s_{I,k}-s_\infty(V_{\mathrm{pre},I,k})\big) e^{-\dt/\tau_I}, \; \tau_I=6,
\]
then map back to logits $z_{s_E,k+1}=\log\!\frac{s_{E,k+1}}{1-s_{E,k+1}}$ and $z_{s_I,k+1}=\log\!\frac{s_{I,k+1}}{1-s_{I,k+1}}$.

\subsubsection{Inputs: Optogenetic Drive}
The control $u\ge 0$ is an optogenetic drive broadcast to all neurons (or provided per neuron). The ChR2 photocurrent is
\begin{align*}
I_{\mathrm{ChR}} &= \gChRmax\, \frac{u}{u+\Kd}\,(V-\EChR), \\
\gChRmax&=50,\; \Kd=1,\; \EChR=0.
\end{align*}
Negative inputs are clamped to zero before the saturation map.

\subsubsection{Synapses and Network Coupling}
Each neuron is excitatory or inhibitory; the first $N_E=\lfloor f_E N \rceil$ are excitatory, the rest inhibitory, with $f_E=0.8$ by default. Connectivity matrices $G_E\in\mathbb{R}^{N\times N_E}$ and $G_I\in\mathbb{R}^{N\times N_I}$ are sampled with Erd\H{o}s-Renyi masks at probability $p=0.2$, scaled by $1/\sqrt{pN}$ and base gains $gE_{\mathrm{base}}=0.1$, $gI_{\mathrm{base}}=0.6$.

Synaptic gates obey voltage-driven sigmoids on the presynaptic voltages,
\[
s_\infty(V;V_\theta,k)=\tfrac12\!\left(1+\tanh\!\frac{V-V_\theta}{2k}\right),
\]
with $V_{\theta,E}=V_{\theta,I}=-20$, $k_E=k_I=2$. Discrete synaptic updates mirror the ion-gate exponential Euler:
\begin{equation}
\begin{aligned}
s_{E,k+1} &= s_\infty(V_{\mathrm{pre},E,k}) + \big(s_{E,k}-s_\infty(V_{\mathrm{pre},E,k})\big) e^{-\dt/\tau_E}, \\ \tau_E=3,\\
s_{I,k+1} &= s_\infty(V_{\mathrm{pre},I,k}) + \big(s_{I,k}-s_\infty(V_{\mathrm{pre},I,k})\big) e^{-\dt/\tau_I}, \\ \tau_I=6,\\
z_{s_E,k+1} &= \log\!\frac{s_{E,k+1}}{1-s_{E,k+1}}, \; z_{s_I,k+1} = \log\!\frac{s_{I,k+1}}{1-s_{I,k+1}}.
\end{aligned}
\end{equation}
Synaptic currents to postsynaptic neuron $i$ are
\begin{align*}
I_{\mathrm{syn},i} = (G_E s_E)_i(V_i-\EE) + (G_I s_I)_i(V_i-\EI), \\ \EE=0,\; \EI=-80.
\end{align*}

\subsubsection{Full, centered discrete state and the implemented map}
The centered internal state used by learning/control is
\begin{equation}
\begin{aligned}
x = \big[&\hat V^\top,\; 
\tilde{z}_m^\top,\;
\tilde{z}_h^\top,\;
\tilde{z}_n^\top,\;
\tilde{z}_{s_E}^\top,\;
\tilde{z}_{s_I}^\top\big]^\top,
\end{aligned}
\end{equation}
with dimension $4N+N_E+N_I$.
Here $\hat V$ and all $\tilde z$ variables are centered coordinates. The corresponding offset coordinates are obtained by the affine shift $\hat x = x + x_\star$.
Given $(x_k,u_k)$, the discrete map $x_{k+1}=f(x_k,u_k)$ is implemented by:
(i) forming $\hat x_{k}=x_k+ x_{\star}$,
(ii) reconstructing physical $(V,m,h,n,s_E,s_I)$ via $V=\EL\mathbf{1}+\Vscale \hat V$ and $\sigma(\cdot)$ on logits,
(iii) applying the voltage forward-Euler update with factor $\alphahh$,
(iv) applying exponential-Euler updates for ion and synaptic gates, with logits computed via $\log\!\frac{p}{1-p}$ (with probability clipping in software),
and (v) applying the inverse shift $x_{k+1}=\hat x_{k+1}-x_\star$.

All data generation stores physical coordinates $(V,m,h,n,s_E,s_I)$ obtained by this invertible reconstruction.